\documentclass[12pt]{amsart}

\usepackage{amsmath,amsfonts,amssymb,graphicx,hyperref, amsthm,mathabx,amsrefs,geometry}
\hypersetup{colorlinks=true,citecolor=blue,linkcolor=blue,urlcolor=blue,pdfstartview=FitH,bookmarksdepth=3}
\usepackage[noabbrev, nameinlink]{cleveref}
\usepackage{enumitem}
\setlist[enumerate]{label=\arabic*.}

\newcommand{\paren}[2][]{\expandafter\ifx\expandafter\relax\detokenize{#1}\relax
  \left( #2 \right)
  \else
  \left( #2 \right)^{#1}
  \fi}
\newcommand{\set}[1]{\ensuremath{\mathopen{}\left\{ #1 \right\}\mathclose{}}}

\newcommand{\innerprod}[1]{\ensuremath{\mathopen{}\left< #1 \right>\mathclose{}}}

\newcommand{\abs}[1]{\ensuremath{\mathopen{}\left| #1 \right|\mathclose{}}}
\newcommand{\norm}[1]{\ensuremath{\mathopen{}\left\| #1 \right\|\mathclose{}}}

\newcommand{\summod}[1]{\ensuremath{\,(\mathrm{mod}\,#1)}}

\newcommand{\Matrix}[1]{\begin{pmatrix}#1\end{pmatrix}}

\newcommand\revdots{\mathinner{\mkern1mu\raise1pt\vbox{\kern7pt\hbox{.}}\mkern2mu\raise4pt\hbox{.}\mkern2mu \raise7pt\hbox{.}\mkern1mu}}

\newcommand{\piecewise}[1]{\left\{\begin{matrix}#1\end{matrix}\right.}

\newcommand{\If}{\mbox{if }}
\newcommand{\Otherwise}{\mbox{otherwise}}

\renewcommand{\Re}{{\mathop{\mathgroup\symoperators Re}}}
\renewcommand{\Im}{{\mathop{\mathgroup\symoperators Im}}}
\newcommand{\sgn}{{\mathop{\mathgroup\symoperators \,sgn}}}

\newcommand{\Max}[1]{\ensuremath{\max \set{#1}}}
\newcommand{\Min}[1]{\ensuremath{\min \set{#1}}}

\newcommand{\Z}{\mathbb{Z}}
\newcommand{\R}{\mathbb{R}}
\newcommand{\N}{\mathbb{N}}

\newcommand{\C}{\mathbb{C}}

\newcommand{\wbar}[1]{\overline{#1}}
\newcommand{\wtilde}[1]{\widetilde{#1}}
\newcommand{\what}[1]{\widehat{#1}}
\newcommand{\wcheck}[1]{\widecheck{#1}} 
\newcommand{\BigO}[2][]{O_{#1}\paren{#2}}
\newcommand{\e}[1]{e\paren{#1}}
\newcommand{\pFqName}[2]{{_{#1}F_{#2}}}
\newcommand{\pFq}[5]{\pFqName{#1}{#2}\paren{\begin{matrix}#3;\\#4;\end{matrix}\,#5}}

\newcommand{\trans}[1]{{#1}^T}

\newcommand{\Poch}[2]{\paren{#1}_{#2}}

\DeclareMathOperator*{\res}{res}
\DeclareMathOperator{\diag}{diag}

\theoremstyle{plain} 
\newtheorem{thm}{Theorem}
\newtheorem*{thm*}{Theorem}
\newtheorem{cor}[thm]{Corollary}
\newtheorem*{cor*}{Corollary}
\newtheorem{lem}[thm]{Lemma}
\newtheorem*{lem*}{Lemma}

\newtheorem*{prop*}{Proposition}

\newtheorem*{conj*}{Conjecture}

\newtheorem*{ax*}{Axiom}

\theoremstyle{definition}
\newtheorem*{defn*}{Definition}

\theoremstyle{remark}
\newtheorem*{rem*}{Remark}

\newtheorem*{prob*}{Problem}

\crefname{thm}{Theorem}{Theorems}
\crefname{lem}{Lemma}{Lemmas}
\crefname{prop}{Proposition}{Propositions}
\crefname{cor}{Corollary}{Corollaries}
\crefname{conj}{Conjecture}{Conjectures}
\crefname{defn}{Definition}{Definitions}
\crefname{prob}{Problem}{Problems}

\newif\ifshowTODOs
\showTODOstrue
\newcommand{\TODO}[1]{\ifshowTODOs{\ifmmode \text{\color{blue}TODO: #1} \else {\ {\color{blue}TODO: #1}} \fi}\fi}
\newcommand{\pTODO}[1]{\ifshowTODOs{\ifmmode \text{\color{blue}(TODO: #1)} \else {\ {\color{blue}(TODO: #1)}} \fi}\fi}

\DeclareMathAlphabet{\mathcalligra}{T1}{calligra}{m}{n}

\newcommand{\pFqStarName}[2]{{_{#1}F_{#2}^*}}
\newcommand{\pFqStar}[5]{\pFqStarName{#1}{#2}\paren{\begin{matrix}#3;\\#4;\end{matrix}\,#5}}
\newcommand{\pFqDaggerName}[2]{{_{#1}F_{#2}^\dagger}}
\newcommand{\pFqDagger}[5]{\pFqDaggerName{#1}{#2}\paren{\begin{matrix}#3;\\#4;\end{matrix}\,#5}}

\renewcommand{\ggg}{>\!\!>\!\!>}

\newcommand{\Weyl}{\mathcal{W}}

\usepackage{subcaption}
\usepackage{xspace}

\newcommand{\mlb}{\\ \hspace*{0.3in}}
\newcommand{\mlbb}{\\ \hspace*{0.6in}}
\newcommand{\mlbbb}{\\ \hspace*{0.9in}}
\newcommand{\dspecmu}{\xspace d_\text{spec}\mu}

\title[Bessel functions on $GL(n)$, II]{Bessel functions on $GL(n)$, II - the case $n=4$}
\author{Jack Buttcane}
\date{22 March 2025}
\address{Department of Mathematics \& Statistics, 5752 Neville Hall, Orono, ME 04469, USA}
\email{jack.buttcane@maine.edu}
\thanks{During the time of this research, the author was supported by NSF grant DMS-2348653.}

\begin{document}

\begin{abstract}
The purpose of this article is to verify the conjectures of the previous paper in the particular case of $GL(4)$.
We accomplish this in general, but observe two failures of the conjectures:
First, that the Strong Interchange of Integrals conjecture is perhaps false for a single Weyl element $w_{2,2}$, though we prove the Weak Interchange of Integrals still holds.
Second, again for a single Weyl element $w_{2,1,1}$ and its conjugate $w_{1,1,2}$, it appears that the space of solutions to the Bessel differential equations may not be spanned by the Frobenius series solutions.
We discuss what refinements, namely to the Asymptotics Theorem, would be necessary to uniquely identify the Bessel functions for such Weyl elements, and prove them in the exceptional cases for $GL(4)$.
\end{abstract}


\maketitle

\section{Introduction}

In this paper, we continue the study of the $GL(n)$ Bessel functions in the form initiated in \cite{GLnI}.
The particular functions we are concerned with are those that appear in the Kuznetsov (aka relative trace) formulas for $GL(n)$.
The formulas generalize those of Kuznetsov on $GL(2)$ \cite{Kuz01} and the kernel functions appearing the integral transforms generalize in turn the classical $J$- and $K$-Bessel functions.

The history of the Kuznetsov formulas on $GL(n)$ starts with Li's formula appearing in \cite{Gold01} and several specializations to $GL(3)$ appearing in \cite{Val01,GoldKont02} along with a partial inversion appearing in the author's thesis \cite{MeThesis}, but the first to tackle the $GL(3)$ formula by directly studying the kernel functions was \cite{SpectralKuz}.
This latter was the approach employed in \cite{GLnI}, which gives a series of conjectures on generalizing the author's method from $GL(3)$ to $GL(n)$, and the current paper proves these conjectures for $GL(4)$.

This paper was intended to be a complete, worked example of the method which would help illuminate the general case, and it is largely successful with a few caveats, but a highly interesting discovery arose along the way:
In \cite[Sect. 5.5.1]{GLnI}, the author made a blatant guess as to an integral representation for the Bessel functions, and in this paper we see that guess is, in fact, correct; even stronger, the constant of proportionality is one.
This opens the possibility of simply constructing the Kuznetsov formula in such a way that those integrals directly occur as the kernel functions, and the author intends to pursue this approach in a subsequent paper.
Such a method would skip past the interchange of integrals (see below), differential equations, and power series representations, which are certainly desirable in their own right, and proceed directly to the integral representations, which are the most important for applications.

The caveats encountered are primarily minor exceptions to the conjectures of \cite{GLnI}, which are described below, but one relates to the purity of mathematical rigor:
The author was able to work through the proof of the Interchange of Integrals Conjecture by hand in every case but the long Weyl element (which was simply too long, even for the author), but the resulting proofs are trees of cases so involved that the author considers it to be unpublishable (certainly the author would not inflict the checking of it on a referee).
Instead, we rely non-trivially on a computer algebra package to reduce a certain Boolean combination of linear inequalities to just one or two particular vectors that need to be manually checked, and this corresponds to a reduction from over forty pages to just six (including the description of the algorithm) in \cref{sect:IoI}.
The expression to be reduced corresponds precisely to the tree of cases alluded to above, but a human (i.e. the author) would approach the problem by grouping common terms to simplify later cases, while the computer-based approach is a brute-force simplification of the entire expression at once.

\subsection{The Kuznetsov formula}
The Kuznetsov formula for a discrete subgroup $\Gamma \subset GL(4,\R)$ is an equality between a sum of the Fourier-Whittaker coefficients of automorphic forms in some spectral family and sums of generalized Kloosterman sums.
For integral diagonal matrices $m$ and $n$, the Kuznetsov formula has the rough form
\begin{align*}
	& \sum_\varphi \rho_\varphi(n) \wbar{\rho_\varphi(m)} f(\mu_\varphi)+\text{ Eisenstein series terms} \\
	&= \sum_{w\in W} \sum_c \frac{S_w(m,n,c)}{\abs{c_1 c_2 c_3}} H_w(f, mc wn^{-1} w^{-1}),
\end{align*}
where $\varphi$ runs through the family of cusp forms in $L^2(\Gamma\backslash GL(4,\R)/\R^+)$ sharing a common minimal weight $\Lambda$ and discrete spectral parameters $\delta$, $W$ is the Weyl group of $GL(4)$, and $c$ runs through the diagonal part of the Bruhat decomposition of $\Gamma$.
Here, $\mu_\varphi$ are the continuous spectral parameters of $\varphi$ and $f(\mu)$ is some nice test function, $\rho_\varphi(m)$ is the Fourier-Whittaker coefficient of $\varphi$ at $m$, $S_w(m,n,c)$ is the Kloosterman sum, and $H_w(f,y)$ is an integral transform of $f$.

A theorem of Friedberg \cite{Friedberg01} implies that only the ``relevant'' Weyl elements $w \in W^\text{rel} \subset W$, which are reverse-block diagonal matrices formed from identity matrices in the form
\begin{align}
\label{eq:RelevantWeylEles}
	w_{r_1,\ldots,r_\ell} = \Matrix{&&I_{r_1}\\&\revdots\\I_{r_\ell}}, \qquad r_1+\ldots+r_\ell = 4,
\end{align}
contribute to the Kuznetsov formula, and similarly, we define the Kloosterman sums to be zero unless a certain ``compatibility'' condition is met, which corresponds to the argument of $H_w(f,\cdot)$ lying in a subspace $Y_w$ of the diagonal matrices.

The most technically difficult theorem of this paper is that each $H_w(f,y)$ can be written as a kernel integral transform of $f$, and we call the kernel functions $K_w(y,\mu,\delta)$ the $GL(4)$ Bessel functions.
As this theorem is also somewhat technical to state, we wait until \cref{sect:Conjectures} to do so.

We tend to suppress commas in the subscript and double subscripts, so, e.g., $w_{121} := w_{1,2,1}$ and $K_{121}(y,\mu,\delta) := K_{w_{1,2,1}}(y,\mu,\delta)$.
In this notation, the relevant Weyl elements for $GL(4)$ are
\[ w_4=I, w_{13}, w_{31}, w_{22}, w_{112}, w_{121}, w_{211}, w_{1111} =: w_l. \]
From the symmetry \cite[Prop. 2]{GLnI} of the Bessel function $K_w(y,\mu,\delta)$ under the involution $w \mapsto w^\iota$ with $\iota:G \to G$ defined by $g^\iota = w_l \trans{\paren{g^{-1}}} w_l$, it suffices to consider the elements
\[ w_{31}, w_{22}, w_{121}, w_{211}, w_{1111}. \]
Note that $K_I(y,\mu,\delta)=1$, so there is no further need to study this function.

\subsection{Results}
The end goal of this paper is to provide usable integral and series representations of the $GL(4)$ Bessel functions $K_w(y,\mu,\delta)$.
Applications of the $GL(4)$ Kuznetsov formulas will require bounds or asymptotics for the Bessel functions and their integral transforms, which we leave to future papers.
In the current paper, we first show that the integral kernels in the $GL(4)$ Kuznetsov formula exist as functions, see \cref{sect:IoI}.
That is, we prove the Interchange of Integrals Conjecture of \cite{GLnI} for $n=4$.

Next, we construct the differential equations satisfied by the Bessel functions, see \cref{sect:DEPS}.
\begin{thm}
The Bessel functions at $w=w_{31},w_{22},w_{121},w_{211},w_l$ are annihilated by the differential operators \eqref{eq:w31DE}, \eqref{eq:w22DE}, \eqref{eq:w121DE1} and \eqref{eq:w121DE2}, \eqref{eq:211DEs1} and \eqref{eq:211DEs2}, and \eqref{eq:1111DEs1}-\eqref{eq:1111DEs3}, respectively.
\end{thm}

Also in \cref{sect:DEPS}, we find the power series (Frobenius series) solutions in the principal series case.
\begin{thm}
Using the usual (terminating) $\pFqName{p}{q}$ hypergeometric series, the power series solutions $J_w(y,\mu)$ to the Bessel differential equations of the previous theorem are given by
\begingroup
\allowdisplaybreaks
\begin{align}
\label{eq:Jw31}
	J_{31}(y,\mu) =& \abs{y_3}^{\frac{3}{2}-\mu_4} \sum_{m=0}^\infty a_{31,m}(\mu) y_3^m, \\
\label{eq:Jw22}
	J_{22}(y,\mu) =& \abs{y_2}^{2+\mu_1+\mu_2} \sum_{m=0}^\infty a_{22,m}(\mu) y_2^m, \\
\label{eq:Jw121}
	J_{121}(y,\mu) =& \abs{y_1}^{\frac{3}{2}+\mu_1} \abs{y_3}^{\frac{3}{2}-\mu_4} \sum_{m_1=0}^\infty \sum_{m_2=0}^\infty a_{121,m_1,m_2}(\mu) y_1^{m_1} y_3^{m_2}, \\
\label{eq:Jw211}
	J_{211}(y,\mu) =& \abs{y_2}^{2+\mu_1+\mu_2} \abs{y_3}^{\frac{3}{2}-\mu_4} \sum_{m_1=0}^\infty \sum_{m_2=0}^\infty a_{211,m_1,m_2}(\mu) y_2^{m_1} y_3^{m_2}, \\
\label{eq:Jw1111}
	J_{1111}(y,\mu) =& \abs{y_1}^{\frac{3}{2}+\mu_1} \abs{y_2}^{2+\mu_1+\mu_2} \abs{y_3}^{\frac{3}{2}-\mu_4} \sum_{m_1=0}^\infty \sum_{m_2=0}^\infty \sum_{m_3=0}^\infty a_{1111,m_1,m_2,m_3}(\mu) y_1^{m_1} y_2^{m_2} y_3^{m_3},
\end{align}
\endgroup
with coefficients
\begingroup
\allowdisplaybreaks
\begin{align*}
	a_{31,m}(\mu) :=& \frac{(16\pi^4)^{-\mu_4} (-16\pi^4)^m}{m! \, \Gamma(1+\mu_1-\mu_4+m) \Gamma(1+\mu_2-\mu_4+m) \Gamma(1+\mu_3-\mu_4+m)}, \\
	a_{22,m}(\mu) :=& \frac{(16\pi^4)^{\mu_1+\mu_2+m} \Gamma(1+2\mu_1+2\mu_2+2m)}{m! \, \Gamma(1+2\mu_1+2\mu_2+m) \Gamma(1+\mu_1-\mu_3+m)} \\
	& \qquad \times \frac{1}{\Gamma(1+\mu_1-\mu_4+m) \Gamma(1+\mu_2-\mu_3+m) \Gamma(1+\mu_2-\mu_4+m)}, \\
	a_{121,m}(\mu) =& \frac{(8\pi^3)^{\mu_1-\mu_4} (8\pi^3 i)^{m_1} (-8\pi^3 i)^{m_2} \Gamma(1+\mu_1-\mu_4+m_1+m_2)}{m_1! \, m_2! \, \Gamma(1+\mu_1-\mu_2+m_1) \Gamma(1+\mu_1-\mu_3+m_1)\Gamma(1+\mu_1-\mu_4+m_1)} \\*
	& \qquad \times \frac{1}{\Gamma(1+\mu_1-\mu_4+m_2) \Gamma(1+\mu_2-\mu_4+m_2) \Gamma(1+\mu_3-\mu_4+m_2)}, \\
	a_{211,m}(\mu) =& \frac{(2\pi)^{3\mu_1+3\mu_2-2\mu_4} (-8\pi^3 i)^{m_1} (4\pi^2)^{m_2}}{m_1! \, m_2! \, \Gamma\paren{1+\mu_1-\mu_4+m_1} \Gamma\paren{1+\mu_2-\mu_4+m_1} \Gamma\paren{1+\mu_3-\mu_4+m_2}} \\*
	& \qquad \times \frac{1}{\Gamma\paren{1+\mu_1-\mu_3} \Gamma\paren{1+\mu_2-\mu_3}} \\*
	& \qquad \times \pFq32{ -m_1,1+2\mu_1+2\mu_2+m_1,\mu_4-\mu_3-m_2}{1+\mu_1-\mu_3, 1+\mu_2-\mu_3}{1}, \\
	a_{1111,m}(\mu) =& \frac{(2\pi)^{3\mu_1+\mu_2-\mu_3-3\mu_4+2m_1+2m_2+2m_3}}{m_1! \, m_2! \, m_3! \Gamma\paren{1+\mu_3-\mu_4} \Gamma\paren{1 + \mu_1 - \mu_2+m_1} \Gamma\paren{1 + \mu_1 - \mu_3+m_1}} \\*
	& \qquad \times \frac{\Gamma\paren{1 + 2\mu_1 + 2\mu_2+m_2 + m_3} \Gamma\paren{1 + \mu_1 - \mu_3+m_1 + m_2}}{\Gamma\paren{1 + 2\mu_1 + 2\mu_2+m_2} \Gamma\paren{1 + \mu_1 - \mu_3+m_2}} \\*
	& \qquad \times \frac{1}{\Gamma\paren{1 + \mu_2 - \mu_3+m_2} \Gamma\paren{1 + \mu_1 - \mu_4+m_3} \Gamma\paren{1 + \mu_2 - \mu_4+m_3}} \\*
	& \qquad \times \pFq43{-m_1 - \mu_1 + \mu_3, -m_2 - \mu_2 + \mu_3, -m_2 - \mu_1 + \mu_3, -m_3}{1 + \mu_3 - \mu_4, -m_1 - m_2 - \mu_1 + \mu_3, -m_2 - m_3 - 2 (\mu_1 + \mu_2)}{1}.
\end{align*}
\endgroup
\end{thm}
Having the power series solutions provides the following benefits:
\begin{enumerate}
\item As a technical point of the current paper, in combination with the Asymptotics Theorem (see \cref{sect:Asymptotics}), knowing the power series expansions allows us to show that the Mellin-Barnes integrals we construct are, in fact, the Bessel functions.

\item When all of the $y$ arguments are small in terms of the spectral parameters $\mu$, the power series expansion is the asymptotic expansion.

\item Knowing that the Mellin-Barnes integral is equal to a particular sum of power series dramatically simplifies shifting its contour past the poles as there is no need work out residues at the actual poles, or to show that certain potential poles do not, in fact, exist.
This is a particularly trying procedure, see \cref{sect:BesselMBProof}.

\item Min Lee has communicated to the author a method through which one can use the power series to show that $H_w(f,y)$ is not merely negligible, but actually equal to zero when one of the coordinates of $y$ is small enough, provided $f$ has zeros at proscribed locations.
\end{enumerate}

For $\ell \in \Z$, $s \in \C$ we define a multiplicative character of $\R^\times$ by
\[ \chi_s^\ell(a) := \sgn(a)^\ell \abs{a}^s. \]
For $\mu \in \C^4$ with $\mu_1+\ldots+\mu_4=0$ and $\delta \in \Z^n$ (without restriction), we extend this to a character of diagonal matrices, aka the ``power function'' by
\[ I_{\mu,\delta}(\diag(a_1,\ldots,a_4)) := \prod_{i=1}^4 \chi_{\rho_i+\mu_i}^{\delta_i}(a_i), \]
where $\rho=(3,1,-1,-3)/2$ is the half-sum of the positive roots for $GL(4)$.
Finally, we also use the power function in the form
\[ \wtilde{I}_{\mu,\delta}\paren{a_1,a_2,a_3} := \prod_{i=1}^3 \chi_{-1+\mu_{i+1}-\mu_i}^{\delta_i+\delta_{i+1}}(a_i). \]
The spectral parameters $\mu$ are called the ``continuous part'' of the character $I_{\mu,\delta}$, and the parameters $\delta$ are called the ``discrete part''.
\emph{Throughout the paper, unless otherwise specified, we assume the coordinates of $\mu$ are distinct modulo $\Z$.}

For $\eta \in \Z$ and $s \in \C$, we define
\[ G_\eta(s) = \pi^{\frac{1}{2}-s} i^{\eta'} \frac{\Gamma\paren{\frac{\eta'+s}{2}}}{\Gamma\paren{\frac{1+\eta'-s}{2}}}, \]
where $\set{0,1} \ni \eta' \equiv \eta \pmod{2}$, and for $\ell \in \Z, t \in C^n, \eta \in \Z^n$, we define
\[ G_\ell(s,t,\eta) = \prod_{j=1}^n G_{\ell+\eta_j}(s+t_j). \]

Then in \cref{sect:IntRepns}, we construct Mellin-Barnes integrals for the Bessel functions.
\begin{thm}
Writing $\Delta=\delta_1+\delta_2+\delta_3+\delta_4$, the Bessel functions $K_w(y,\mu,\delta)$ are given by the Mellin-Barnes integrals
\begingroup
\allowdisplaybreaks
\begin{gather}
\label{eq:Kw31MB}
\begin{aligned}
	K_{31}(y,\mu,\delta) =& \frac{(-1)^\Delta}{4} \sum_{\ell\in\set{0,1}} \int_{\Re(s)=\frac{1}{5}} \chi_{\frac{3}{2}-s}^\ell(-y_3) G_\ell(s,-\mu,\Delta-\delta) \frac{ds}{2\pi i},
\end{aligned} \\
\label{eq:Kw22MB}
\begin{aligned}
	K_{22}(y,\mu,\delta) =& \frac{1}{4} \sum_{\ell\in\set{0,1}} \int_{\Re(s)=\frac{1}{4}} \chi_{2-s}^{\ell}(y_2) \frac{G_\ell(s,(\mu_i+\mu_j)_{i<J}, (\delta_i+\delta_j)_{i<j})}{G_\Delta(2s)} \frac{ds}{2\pi i},
\end{aligned} \\
\label{eq:Kw121MB}
\begin{aligned}
	& K_{121}(y,\mu,\delta) = \\*
	& \frac{(-1)^\Delta}{4} \sum_{\ell\in\set{0,1}^2} \int_{\Re(s)=(\frac{1}{7},\frac{1}{7})} \chi_{\frac{3}{2}-s_1}^{\ell_1}(-y_1) \chi_{\frac{3}{2}-s_2}^{\ell_2}(y_3) \frac{G_{\ell_1}(s_1,\mu,\delta) G_{\ell_2}(s_2,-\mu,\Delta-\delta)}{G_{\ell_1+\ell_2+\Delta}(s_1+s_2)} \frac{ds}{(2\pi i)^2},
\end{aligned} \\
\label{eq:Kw211MB}
\begin{aligned}
	& K_{211}(y,\mu,\delta) = \\*
	& \frac{(-1)^\Delta}{8} \sum_{\ell\in\set{0,1}^3} \int_{\Re(s)=(\frac{1}{7},\frac{1}{7},\frac{1}{7})} \chi_{2-s_1}^{\ell_1}(y_2) \chi_{\frac{3}{2}-s_3}^{\ell_3}(y_3) \\*
	& \times  G_{\ell_1}(s_1,-(\mu_1+\mu_2,\mu_1+\mu_3,\mu_2+\mu_3),\Delta-(\delta_1+\delta_2,\delta_1+\delta_3,\delta_2+\delta_3)) G_{\ell_3+\Delta-\delta_4}(s_3-\mu_4) \\*
	& \times \frac{G_{\ell_1+\ell_2+\delta_4}(s_1-s_2-\mu_4) G_{\ell_2+\ell_3}(s_3-s_2) G_{\ell_2}(s_2,-(\mu_1,\mu_2,\mu_3),\Delta-(\delta_1,\delta_2,\delta_3))}{G_{\ell_1+\ell_2+\Delta-\delta_4}(s_1+s_2+\mu_4)} \frac{ds}{(2\pi i)^3},
\end{aligned} \\
\label{eq:Kw1111MB}
\begin{aligned}
	& K_{1111}(y,\mu,\delta) = \\*
	& \frac{(-1)^\Delta}{16} \sum_{\ell\in\set{0,1}^4} \int_{\Re(s)=(2\epsilon,2\epsilon,2\epsilon,\epsilon)} \chi_{\frac{3}{2}-s_1}^{\ell_1}(-y_1) \chi_{2-s_2}^{\ell_2}(-y_2) \chi_{\frac{3}{2}-s_3}^{\ell_3}(-y_3) \\*
	& G_{\ell_1}(s_1,(\mu_1,\mu_2),(\delta_1,\delta_2)) G_{\ell_2}(s_2,(\mu_1+\mu_2,\mu_3+\mu_4),(\delta_1+\delta_2,\delta_3+\delta_4)) \\*
	& G_{\ell_3}(s_3,-(\mu_1,\mu_2),\Delta-(\delta_1,\delta_2)) G_{\ell_4}(s_4,(\mu_3,\mu_4),(\delta_3,\delta_4)) \\*
	& \frac{G_{\ell_1+\ell_4}(s_1-s_4) G_{\ell_2+\ell_4}(s_2-s_4,(\mu_1,\mu_2),(\delta_1,\delta_2)) G_{\ell_3+\ell_4+\delta_1+\delta_2}(s_3-s_4+\mu_1+\mu_2)}{G_{\ell_1+\ell_2+\ell_4+\delta_1+\delta_2}(s_1+s_2-s_4+\mu_1+\mu_2) G_{\ell_2+\ell_3+\ell_4+\Delta}(s_2+s_3-s_4)} \frac{ds}{(2\pi i)^4}.
\end{aligned}
\end{gather}
\endgroup
\end{thm}
The Mellin-Barnes integral representations can be used to provide a simple bound on the Bessel functions, but are more generally used to provide a truncated power series expansion by shifting the contours to the left.
The former, in concert with known bounds on the Kloosterman sums \cite{DabReeder,BlomerMan,Linn}, can be used to prove a Weyl law with power-saving error term.
The latter gives a Paley-Weiner theorem for $H_w(f,y)$, i.e. that $H_w(f,y)$ has better decay in $y$ near zero than $K_w(y,\mu,\delta)$ itself, and this is essential in proving that the sums of Kloosterman sums (i.e. the $c$ sum) can be truncated (in particular, the arithmetic/geometric side of the Kuznetsov formula converges).

For $\eta \in \Z$ and $a>0$, we define
\begin{align*}
	\mathcal{Z}_s^\eta(-a) =& 4 i^\eta K_s\paren{4\pi \sqrt{a}} \cos\tfrac{\pi}{2}\paren{s-\eta}, \\
	\mathcal{Z}_s^\eta(+a) =& \pi i^\eta \frac{J_{-s}\paren{4\pi \sqrt{a}}-(-1)^\delta J_s\paren{4\pi \sqrt{a}}}{\sin\frac{\pi}{2}\paren{s+\eta}},
\end{align*}
where $K_s$ and $J_s$ are the classical Bessel functions.
In the current notation, this would be a particular normalization of the $GL(2)$ Bessel function $K_{11}(y,\mu,\delta)$.

In \cref{sect:IntRepns}, we construct Stade-type integral representations where the $GL(4)$ Bessel function is given as an integral of $GL(2)$ Bessel functions.
These are useful to bound $K_w(y,\mu)$ and show decay when one of the coordinates of $y$ is large.
In \cite{Subconv,GPSSubconv}, these were used to cut off the sums of Kloosterman sums, but the author hopes to replace that with analysis of the inverse Mellin transform, see below.
\begin{thm}
We have the Stade-type multiple Bessel integrals
\begingroup
\allowdisplaybreaks
\begin{gather}
\label{eq:K31Stade}
\begin{aligned}
	K_{31}(y,\mu,\delta) =& (-1)^{\delta_4} \chi_{\frac{3-\mu_1-\mu_2}{2}}^{\delta_1+\delta_3+\delta_4}(y_3) \int_{\R} \chi_{\mu_3+\mu_4-1}^{\delta_2+\delta_4}(u) \mathcal{Z}_{\mu_2-\mu_1}^{\delta_1+\delta_2}(y_3 u) \mathcal{Z}_{\mu_4-\mu_3}^{\delta_3+\delta_4}(-1/u) du,
\end{aligned} \\
\label{eq:K22Stade}
\begin{aligned}
	K_{22}(y,\mu,\delta) =& \chi_{2+\frac{\mu_4-\mu_2}{2}}^{\delta_1+\delta_4}(y_2) \int_{\R^2} \chi_{-1-\frac{\mu_1-2\mu_2+\mu_3}{2}}^{\delta_1+\delta_2}(u_1) \chi_{-1+\mu_4-\mu_2}^{\delta_2+\delta_4}(u_2) \chi_{-1+\frac{\mu_1-2\mu_2+\mu_3}{2}}^{\delta_2+\delta_3}(1-u_1) \\*
	& \qquad \times \mathcal{Z}_{\mu_3-\mu_1}^{\delta_1+\delta_3}(y_2 u_2/u_1) \mathcal{Z}_{\mu_3-\mu_1}^{\delta_1+\delta_3}(1/(u_2 (1-u_1))) du,
\end{aligned} \\
\label{eq:K121Stade}
\begin{aligned}
	K_{121}(y,\mu,\delta) =& (-1)^{\Delta-\delta_2} \chi_{\frac{3-\mu_2-\mu_3}{2}}^{\delta_1}(y_1) \chi_{\frac{3-\mu_1-\mu_4}{2}}^{\Delta-\delta_4}(y_3) \int_{\R^3} \chi_{-1+\mu_2-\mu_3}^{\delta_2+\delta_3}(u_1) \chi_{-1+\frac{\mu_1-2\mu_2+\mu_4}{2}}^{\delta_2+\delta_4}(u_2) \\*
	& \qquad \times  \chi_{-1-\frac{\mu_1-2\mu_2+\mu_4}{2}}^{\delta_1+\delta_2}(u_3)\chi_{-1+\mu_3-\mu_2}^{\delta_1+\delta_3}(1-u_1) \\*
	& \qquad \times  \mathcal{Z}_{\mu_4-\mu_1}^{\delta_1+\delta_4}\paren{\frac{y_1}{u_3}} \mathcal{Z}_{\mu_4-\mu_1}^{\delta_1+\delta_4}\paren{\frac{y_3}{u_2}} \e{\frac{u_2}{u_1}-\frac{u_3}{1-u_1}} du,
\end{aligned} \\
\label{eq:K211Stade}
\begin{aligned}
	& K_{211}(y,\mu,\delta) = \\*
	& (-1)^{\delta_4} \chi_{2+\frac{\mu_4-\mu_2}{2}}^{\delta_1+\delta_4}(y_2) \chi_{\frac{3-\mu_1-\mu_4}{2}}^{\delta_1+\delta_2+\delta_3}(y_3) \int_{\R^3} \chi_{-1+\mu_3-\mu_2}^{\delta_2+\delta_3}(u_1) \chi_{-1+\frac{\mu_1-2\mu_3+\mu_4}{2}}^{\delta_3+\delta_4}(u_2) \\*
	& \qquad \times \chi_{-1-\frac{\mu_1+2\mu_2-\mu_3-2\mu_4}{2}}^{\delta_2+\delta_4}(u_3) \chi_{-1+\mu_4-\mu_1}^{\delta_1+\delta_4}(1-u_1) \chi_{-1+\mu_3-\mu_1}^{\delta_1+\delta_3}(1-u_2) \\*
	& \qquad \times \mathcal{Z}_{\mu_3-\mu_1}^{\delta_1+\delta_3}\paren{y_2 u_3} \mathcal{Z}_{\mu_4-\mu_1}^{\delta_1+\delta_4}\paren{-\frac{y_3}{u_2}} \e{\frac{1}{u_3(1-u_1)}+\frac{u_2}{u_1 u_3(1-u_1)(1-u_2)}} du,
\end{aligned} \\
\label{eq:K1111Stade}
\begin{aligned}
	& K_{1111}(y,\mu,\delta) = \\*
	& (-1)^{\delta_2+\delta_4} \chi_{\frac{3+\mu_1+\mu_2}{2}}^{\delta_1}(y_1) \chi_2^{\delta_1+\delta_3}(y_2) \chi_{\frac{3+\mu_1+\mu_2}{2}}^{\delta_1+\delta_2+\delta_3}(y_3) \int_{\R^2} \chi_{-1+\mu_3+\mu_4}^{\delta_1+\delta_4}(u_1) \chi_{-1+\mu_3+\mu_4}^{\delta_1+\delta_3}(u_2) \\*
	& \qquad \chi_0^{\delta_1+\delta_2}((1-u_1)(1-u_2)) \mathcal{Z}_{\mu_2-\mu_1}^{\delta_1+\delta_2}\paren{y_1(1-u_2^{-1})} \mathcal{Z}_{\mu_2-\mu_1}^{\delta_1+\delta_2}\paren{y_2 (1-u_1^{-1})(1-u_2)} \\*
	& \qquad \mathcal{Z}_{\mu_2-\mu_1}^{\delta_1+\delta_2}(-y_3(1-u_1)) \mathcal{Z}_{\mu_4-\mu_3}^{\delta_3+\delta_4}(-y_2 u_2/u_1) du.
\end{aligned}
\end{gather}
\endgroup
\end{thm}

We will use the following auxiliary functions:
\begingroup
\allowdisplaybreaks
\begin{align*}
	Z_{22}(t) :=& \mathcal{Z}_0^0(t), \\
	Z_{121}(t) :=& \int_{\R^2} \e{t_1 u_1+\frac{u_2}{u_1}+\frac{1}{u_2}} \frac{du_2 \, du_1}{\abs{u_2 u_1}} = \int_{\R} \mathcal{Z}_0^0\paren{t_1 u} \e{\frac{1}{u}} \frac{du}{\abs{u}}, \\
	Z_{211}(t) :=& \int_{\R^2} \e{u_1+u_2+\frac{t_1}{u_1}+\frac{t_2}{u_2}+t_3 \frac{u_1}{u_2}+t_4 \frac{u_2}{u_1}+\frac{t_5}{u_1 u_2}} \frac{du_2 \, du_1}{\abs{u_1 u_2}} \\*
	=& \int_{\R} \mathcal{Z}_0^0\paren{\paren{1+\frac{t_3}{u}}\paren{t_4 u+t_1+\frac{t_5}{u}}} \e{u+\frac{t_2}{u}} \frac{du}{\abs{u}},\\
	Z_{1111}(y,z) :=& \int_{\R^3} e\biggl(-y_3 \frac{u_1 u_2 z_1-z_2+u_1 (1-u_2) z_3}{u_1^2(1-u_2) u_2 u_3}-y_2 u_1-y_1 u_3 \\*
	& \qquad+\frac{u_1 z_1 - z_2}{u_1 u_3 (1-u_2) z_1}+\frac{z_2-u_1(1-u_2)z_3}{u_1 u_2 z_2}+\frac{u_1 u_2 u_3}{z_3}\biggr) \frac{du}{\abs{u_1 u_2 (1-u_2) u_3}} \\*
	=& \int_{\R^2} \mathcal{Z}_0^0\paren{\paren{y_1-\frac{u_1 u_2}{z_3}}\paren{y_3 \frac{u_1 u_2 z_1-z_2+u_1 (1-u_2) z_3}{u_1^2(1-u_2) u_2}-\frac{u_1 z_1 - z_2}{u_1 (1-u_2) z_1}}} \\*
	& \qquad\e{-y_2 u_1+\frac{z_2-u_1(1-u_2)z_3}{u_1 u_2 z_2}} \frac{du}{\abs{u_1 u_2 (1-u_2)}}.
\end{align*}
\endgroup
We note that $Z_{22}(t)$ is a normalization of the $GL(2)$ Bessel function $K_{11}(t,0,0)$ and $Z_{121}(t)$ is a normalization of the $GL(3)$ Bessel function $K_{21}(t,0,0)$, but $Z_{211}(t)$ and $Z_{1111}(y,z)$ appear to be new functions.
These integrals converge conditionally, but nicely enough, see \cref{sect:IntRepns}.

In \cref{sect:IntRepns}, we find good integral representations for the inverse Mellin transform.
We define the inverse Mellin transforms (i.e. Mellin expansion) of the Bessel functions by
\begin{align}
\label{eq:wcheckKdef}
	K_w(y,\mu,\delta) =& I_{-\mu,\delta}(y^{\iota}) \int_{\R^3} \wcheck{K}_w(y,z) \wtilde{I}_{\mu,\delta}(z) dz,
\end{align}
where, in terms of the coordinates \eqref{eq:GwlCoordsY},
\begin{align*}
	I_{\mu,\delta}(y^\iota) =& \chi_{\frac{3}{2}-\mu_4}^{\delta_4}(y_1) \chi_{2-\mu_3-\mu_4}^{\delta_3+\delta_4}(y_2) \chi_{\frac{3}{2}-\mu_2-\mu_3-\mu_4}^{\delta_2+\delta_3+\delta_4}(y_3).
\end{align*}
\begin{thm}
We have the inverse Mellin transforms
\begingroup
\allowdisplaybreaks
\begin{align}
\label{eq:K31IM}
	\wcheck{K}_{31}(y,z) =& (-1)^{\delta_1+\delta_2+\delta_3} \e{-y_3 z_1-\frac{z_2}{z_1}-\frac{z_3}{z_2}+\frac{1}{z_3}}, \\
\label{eq:K22IM}
	\wcheck{K}_{22}(y,z) =& Z_{22}\paren{\frac{(1-y_2 z_2)(z_2+z_1 z_3)(z_1-z_3)}{z_1 z_2 z_3}}, \\
\label{eq:K121IM}
	\wcheck{K}_{121}(y,z) =& (-1)^{\delta_1} Z_{121}\paren{\frac{(z_1-z_2)(z_2-z_3)(1+y_3 z_1)(1-y_1 z_3)}{z_1 z_2 z_3}}, \\
\label{eq:K211IM}
	\wcheck{K}_{211}(y,z) =& (-1)^{\delta_1} Z_{211}\paren{y_2\paren{\frac{z_2}{z_1}-y_3 z_1}, \frac{1}{z_3}-y_3,\frac{1}{y_2 z_2}, y_2 z_3, -\frac{y_2 y_3 z_2}{z_3}}, \\
\label{eq:K1111IM}
	\wcheck{K}_{1111}(y,z) =& (-1)^{\delta_1+\delta_2} Z_{1111}(y,z).
\end{align}
\endgroup
\end{thm}

Finally, also in \cref{sect:IntRepns}, we find integral representations for the Fourier transform of the inverse Mellin transform.
We define this transform by
\begin{align}
\label{eq:wtildeKdef}
	\wtilde{K}_w(t,z) =& \int_{Y_w} \wcheck{K}_w(y,z) \prod_{i=1}^\ell \e{-y_{\hat{r}_i} t_i} dy_i,
\end{align}
for $w=w_{r_1,\ldots,r_\ell}$, $t\in\R^\ell$, where $\hat{r}_i := r_1+\ldots+r_i$.
\begin{thm}
We have the Fourier-inverse Mellin transforms
\begingroup
\allowdisplaybreaks
\begin{align}
\label{eq:K22FIM}
	\wtilde{K}_{22}(t,z) =& \frac{(-1)^{\delta_1+\delta_3}}{\abs{t_1}} \e{-\frac{t_1}{z_2}-\paren{\frac{z_2+z_1 z_3}{t_1}}\paren{\frac{1}{z_3}-\frac{1}{z_1}}}, \\
\label{eq:K121FIM}
	\wtilde{K}_{121}(t,z) =& \frac{(-1)^{\delta_1}}{\abs{t_1 t_2}} \e{\frac{t_2}{z_1}+\frac{(z_1-z_2)(z_3-z_2)}{z_2 t_1 t_2}-\frac{t_1}{z_3}}, \\
\label{eq:K211FIM}
	\wtilde{K}_{211}(t,z) =& \frac{(-1)^{\delta_1}}{\abs{t_1 t_2+z_1}} \e{-\frac{z_1}{z_2}+\frac{z_2 t_2-z_1 z_3}{z_1(z_1+t_1 t_2)}-\frac{t_1 t_2}{z_2}-\frac{z_2+t_1 z_3}{z_3(z_1+t_1 t_2)}}. \\
\label{eq:K1111FIM}
	\wtilde{K}_{1111}(t,z) =& \delta_{U \ge 0} \frac{(-1)^{\delta_1+\delta_2}}{\sqrt{U}} \sum_\pm e\biggl(\frac{t_2 t_3}{z_1}+(z_1-z_3)\paren{\frac{1}{t_3 z_3}-\frac{1}{t_1 z_1}}+\frac{z_1+z_3}{z_2} \\*
	& \qquad +\frac{t_1 t_2(z_2+t_3 z_3)}{z_2 z_3} \pm \sqrt{U}\paren{\frac{1}{t_1 z_1}-\frac{1}{z_2}+\frac{1}{t_3 z_3}}\biggr), \nonumber
\end{align}
\endgroup
where
\[ U = (t_1 t_2 t_3+z_1 - z_3)^2+4 t_1 t_3 (z_2+t_2 z_3), \]
while the Fourier-inverse Mellin transform of $K_{31}$ is singular.
\end{thm}

In terms of the depth of the corresponding analysis, the Weyl law mentioned above is the weakest application of the Kuznetsov formula; it is a significant step up to show subconvexity for the corresponding $L$-functions, and the author envisions applying these integral representations as follows:
\begin{enumerate}[label=\alph*.]
\item To provide a sharp cutoff for the sums of Kloosterman sums (i.e. each sum is negligibly small when any coordinate of $c$ is large enough).
With the inverse Mellin transform we have
\[ H_w(f,y) = \int_{\R^3} \wcheck{K}_w(y,z) \wcheck{f}(y,z) dz, \qquad \wcheck{f}(y,z) := \int_{\Re(\mu)=0} f(\mu) I_{-\mu,\delta}(y^{\iota}) \wtilde{I}_{\mu,\delta}(z) \dspecmu, \]

Analyzing $\wcheck{f}(y,z)$ is trivial (if the effective support of $f$ is large, then $\wcheck{f}$ has small effective support near eight points) and the $GL(2)$ Bessel functions appearing in $\wcheck{K}_w(y,z)$ are (the classical Bessel functions) $K_0(\cdot), J_0(\cdot), Y_0(\cdot)$ as opposed to the functions $K_{\mu_j-\mu_k}(\cdot), J_{\mu_j-\mu_k}(\cdot), Y_{\mu_j-\mu_k}(\cdot)$ appearing in the Stade-type integrals above.
The former are far simpler to analyze, with the trade-off of considering a slightly higher-dimensional integral.

The inverse Mellin transform was also used directly in \cite{MeFan01} to provide a simple, but strong upper bound on $H_w(f,y)$.
(The present method offers a considerably simpler expression than what appeared in that paper.)

\item To provide a bound on the Fourier transform of $H_w(f,y)$ after applying Poisson summation in the index variables $m$ and $n$.
At, say, $w=w_l$, using the Fourier-inverse Mellin transform we have
\begingroup
\allowdisplaybreaks
\begin{gather*}
\begin{aligned}
	\int_{\R^3} \int_{\R^3} &H_{w_l}\paren{f,(m_1n_3,m_2n_2,m_3n_1)} g(m,n) \e{\sum_{i=1}^3(\xi_i m_i+\zeta_i n_i)} dm \, dn \\*
	&= \int_{\R^3} \int_{\R^3} \wtilde{K}_{w_l}(t,z) \wtilde{f}(g,m,n) dt dz,
\end{aligned} \\
	\wtilde{f}(g,m,n) := \int_{\R^3} \int_{\R^3} \wcheck{f}(m_1n_3,m_2n_2,m_3n_1,z) g(m,n) \e{\sum_{i=1}^3(\xi_i m_i+\zeta_i n_i+m_i n_{4-i} t_i)} \, dm \, dn.
\end{gather*}
\endgroup
Again, $\wtilde{f}(g,m,n)$ is (relatively) trivial to analyze and $\wtilde{K}_{w_l}(t,z)$ is quite simple in most cases.
\end{enumerate}

\subsection{Some notation}
Let $G := GL(n,\R)/\R^+$ and $K := O(n)$.
As usual, we treat elements of $G$ as matrices and write $Y \subset G$ for the diagonal matrices, $Y^+ \subset Y$ for the positive diagonal matrices, $U(\R) \subset G$ for the upper-triangular unipotent matrices with standard generic character $\psi_I$, $W \subset G$ for the Weyl group, $\wbar{U}_w = U \cap (w^{-1} \trans{U} w)$ and $U_w = U \cap (w^{-1} U w)$.

As always, the left and right actions of the Weyl group on $\mu$ and $\delta$ are
\[ I^*_{\mu^w,\delta^w}(y) = I^*_{\mu,\delta}(wyw^{-1}) = I^*_{w^{-1}(\mu),w^{-1}(\delta)}(y), \]
where $I^*_{\mu,\delta} := I_{\mu-\rho,\delta}$ is the unnormalized power function (which always looks confusing, but is, in fact, the correct order for the left and right actions), and we use the conventions $\mu^w_i := (\mu^w)_i = \mu_{i^w} = \mu_{w^{-1}(i)} = w^{-1}(\mu)_i$.

We extend the power function via the Iwasawa decomposition to
\begin{align*}
	I^I_{\mu,\delta,\sigma}(xyk) =& I_{\mu,\delta}(y) \Sigma_{\delta,\sigma} \sigma(k), \qquad \Sigma_{\delta,\sigma} := \frac{1}{\abs{V}} \sum_{v \in V} I_{-\rho,\delta}(v) \sigma(v), \qquad \sigma \in \what{K}, \qquad V := Y \cap K,
\end{align*}
and via the Bruhat decomposition on $G_{w_l} := U(\R) Y \trans{U(\R)}$ to
\begin{align*}
	I^B_{\mu,\delta}(xy\trans{u}) =& I_{\mu,\delta}(y).
\end{align*}
Lastly, we define a Whittaker-like power function on $G_{w_l}$ by
\begin{align*}
	I^W_{\mu,\delta}(xy\trans{u}) =& \psi_I(x) I_{\mu,\delta}(y).
\end{align*}

For an irreducible, unitary representation $\sigma$ of $K$, the matrix-valued Jacquet-Whittaker function of $g \in G$ is
\[ W_\sigma(g,\mu,\delta) = \int_{U(\R)} I^I_{\mu,\delta,\sigma}(w_l u g) \wbar{\psi_I(u)} du, \]
where
\[ \psi_I(x) = \e{x_{1,2}+x_{2,3}+x_{3,4}}, \qquad x\in U(\R), \qquad \e{t}:= e^{2\pi it}, \]
and $w_l$ the long Weyl element of $G$.

We will refer to (principal series) representations of $G$ induced from the minimal parabolic as ``principal series'' (aka ``almost spherical'') and (principal series) representations of $G$ induced from a non-minimal parabolic as ``generalized principal series''.
Families of such representations $\pi$ are parameterized by the minimum weight of the irreducible representations of $K$ occurring in $\pi$ and share a common discrete parameter $\delta$.
The weight of a representation of $K$ is either trivial or parameterized either by a single integer $2 \le k_1 \in \N$ or two integers $2 \le k_2 \le k_1 \in \N$.
For such weights, define $\Lambda=(0,0,0,0)$, $\Lambda=\paren{\tfrac{k_1-1}{2},-\tfrac{k_1-1}{2},0,0}$ or $\Lambda=\paren{\tfrac{k_1-1}{2},-\tfrac{k_1-1}{2},\tfrac{k_2-1}{2},-\tfrac{k_2-1}{2}}$, respectively, and we require $\delta_1+\delta_2\equiv k_1 \pmod{2}$ in the second and third case and $\delta_3+\delta_4\equiv k_2 \pmod{4}$ in the third case (see \cite[eq. (23)]{GLnI}).
We tend to, somewhat incorrectly, refer to $\Lambda$ as the weight of the representation, and the spectral parameters $\mu$ of automorphic forms in the family associated to $\Lambda$ and $\delta$ satisfy $\abs{\Re(\mu_i)-\Lambda_i} < \frac{1}{2}$ for $i=1,\ldots,4$ with $\mu_1-\Lambda_1=\mu_2-\Lambda_2$ in the second and third cases, and $\mu_3-\Lambda_3=\mu_4-\Lambda_4$ in the third case.
The tempered forms have $\Re(\mu_i)=\Lambda_i$, $i=1,\ldots,4$, and we take $i \mathfrak{a}_0^*(\Lambda)$ to be the set of all such $\mu$.
The principal series case is then $\Lambda=0$, while the spherical case is $\Lambda=\delta=0$.

Write $\dspecmu$ for the spectral (Plancherel) measure on $i \mathfrak{a}_0^*(\Lambda)$.
With this notation, the kernel integral appearing on the arithmetic/geometric side of the $GL(n)$ Kuznetsov formula (see \cite[Sect. 3]{GLnI}) is
\begin{align}
\label{eq:HwStarDef}
	I_{0,0}(y) H_w^*(f,y,g) =& \int_{\wbar{U}_w(\R)} \int_{i \mathfrak{a}_0^*(\Lambda)} f(\mu) W_\sigma(ywug,\mu,\delta) \dspecmu \, \wbar{\psi_I(u)} du,
\end{align}
for a ``nice'' (e.g. Schwartz-class and holomorphic on a tube domain containing $i \mathfrak{a}_0^*(\Lambda)$) test function $f(\mu)$, and $y,g \in G$.
We generally restrict to $y$ in that subset $Y_w$ of the diagonal matrices $Y$ where the integral is well-defined, viewing $\wbar{U}_w$ as $U/U_w$.

We note that Proposition 2 of \cite{GLnI} should read
\begin{align*}
	K_w(y,\mu,\delta) =& K_{w^\iota}(\tilde{v}y^\iota,-\mu^{w_l},\delta^{w_l}),
\end{align*}
where $\tilde{v} = v w v w^{-1}$ with $v \in Y$ having all coordinates $v_i = -1$, $i=1,\ldots,n$.
(The involution $w^\iota$ is missing from the subscript in that paper.)

\subsection{The conjectures}
\label{sect:Conjectures}
We summarize the ideas of \cite[Sect. 4 and 5.5.1]{GLnI} as follows
\begin{itemize}
\item \emph{The Interchange of Integrals Conjecture.}
If we were able to interchange the $u$ and $\mu$ integals in $H_w^*$ we could conclude
\begin{align}
\label{eq:HwStarDef}
	I_{0,0}(y) H_w^*(f,y,g) = \int_{i \mathfrak{a}_0^*(\Lambda)} f(\mu) K_w(y,\Lambda,\mu,\delta) W_\sigma(g,\mu,\delta) \dspecmu,
\end{align}
for some function $K_w(y,\Lambda,\mu,\delta)$ defined by 
\begin{align}
\label{eq:KwDef}
	K_w(y,\Lambda,\mu,\delta) W_\sigma(g,\mu,\delta) = \int_{\wbar{U}_w(\R)} W_\sigma(ywug,\mu,\delta) \wbar{\psi_I(u)} du.
\end{align}
Here we have applied \cite[Prop. 1]{GLnI} which uses Shalika's multiplicity one theorem to separate the functions $K_w$ and $W_\sigma$, and shows that $K_w$ will not depend on $\sigma$ by considering the action of the Lie algebra.

The latter integral does not converge absolutely, so we cannot apply the theorems of Fubini and Tonelli. Instead we interpret the integral conditionally by considering a particular choice of coordinates on $u$, introducing a smooth partition of unity, and treating the smooth, compactly supported integral as an oscillatory integral.
The conjecture is then that the series in the dyadic partition converges absolutely (the weak conjecture), which allows the interchange, and hopefully converges rapidly (the strong conjecture).

\begin{thm}
The $GL(4)$ Bessel functions $K_w(y,\Lambda,\mu,\delta)$ at $w=w_{121},w_{211}$ and $w=w_l$ and the $GL(n)$ Bessel functions at $w=w_{n,1}$ satisfy the Strong Interchange of Integrals Conjecture.
The Bessel function at $w=w_{22}$ satisfies the Weak Interchange of Integrals Conjecture.
\end{thm}

\item \emph{The Analytic Continuation Conjecture.}
The Strong Interchange of Integrals Conjecture implies the function $K_w(y,\Lambda,\mu,\delta)$ does not depend on $\Lambda$, as this is true of the Whittaker function, which, with rapid convergence, implies that \eqref{eq:KwDef} at different $\Lambda$ is given by analytic continuation from $\Lambda=0$ (see \cite[Prop. 3]{GLnI}).
We drop $\Lambda$ from the notation and write $K_w(y,\mu,\delta)$.
Furthermore, the leftover Whittaker function can be removed from the Kuznetsov formulas by a Stone-Weierstrass-type theorem so that $H_w^*(f,y,g)$ becomes
\begin{align}
\label{eq:HwDef}
	I_{0,0}(y) H_w(f,y) = \int_{i \mathfrak{a}_0^*(\Lambda)} f(\mu) K_w(y,\mu,\delta) \dspecmu.
\end{align}

\begin{thm}
The $GL(4)$ Bessel functions satisfy the Analytic Continuation Conjecture.
\end{thm}

\item \emph{The Differential Equations and Power Series Conjecture.}
As the (iterated) integrals \eqref{eq:HwStarDef} and \eqref{eq:HwStarDef} \emph{do} converge absolutely, it follows (by taking $f$ to be an approximation to the identity) that $K_w$ should satisfy some partial differential equations inherited from the action of the Casimir operators $\Delta_i$, $i=2,\ldots,n$, on the Whittaker function.
These differential equations, which are conjecturally independent of $\delta$, can then be solved through the method of Frobenius to arrive at power series (really Frobenius series) solutions $J_w(y,\mu)$ and then $K_w$ must be a linear combination of those, but there are some complications for $w$ different from the long Weyl element $w_l$.

We extend $K_w$ from $Y_w$ to $G_w := U(\R)Y_w (w \wbar{U}_w(\R) w^{-1})$ by the natural symmetry
\begin{align}
\label{eq:KwSymmetry}
	K_w\paren{uy(wu'w^{-1}),\mu,\delta} =& \psi_I(uu') K_w(y,\mu,\delta),
\end{align}
but still find ourselves short of symmetries.
That is, for $w \ne w_l$, the total number of coordinates in $G_w$ is less than that of $G$ and so a naive application of a Casimir operator would result in a differential equation involving Bessel-type functions defined using the derivatives of the Whittaker function.

This conjecture is then that there are operators which avoid the missing coordinates, and below the statement of the conjecture in \cite[Sect. 4]{GLnI}, the author suggests looking for elements of the form
\[ \sum_{i=2}^n a_i X_i \circ (\Delta_i-\lambda_i) \]
with $a_i = a_i(\lambda_2,\ldots,\lambda_n) \in \C[y_1,\ldots,y_{n-1}]$ and $X_i$ an element of the commutative algebra
\[ \C[\Delta_2,\ldots,\Delta_n,E_{1,1},\ldots,E_{n-1,n-1}], \]
using the usual basis $E_{i,j}$ of the Lie algebra of $G$; in the current paper, it seems that the coefficients $a_i$ depending on $y$ are, in fact, not necessary.
The conjecture also predicts the number of solutions to be $\abs{W/W_w}$ where $W_w$ are those Weyl elements that fix all $y \in Y_w$ under conjugation.

\item \emph{The Asymptotics Theorem} of \cite{GLnI} computes the coefficients involved in writing $K_w$ as a linear combination of the $J_w$ series at $\Lambda=0$, assuming all of the conjectures; these are the first-term asymptotics of $K_w$ as $y \to 0$ in $Y_w$.

\item \emph{The Direct Integral Representation Conjecture.}
Section 5.5.1 of \cite{GLnI} conjectures that $K_w(y,\mu,\delta)=C(\mu,\delta) \mathcal{I}_w(y,\mu,\delta)$ for some constant $C(\mu,\delta)$ where
\[ \mathcal{I}_w(y,\mu,\delta) = \int_{\wbar{U}_w(\R)} I^W_{\mu,\delta}(y,w,u) \wbar{\psi_I(u)} du. \]
On the results of this paper, the author would now strengthen the conjecture to $C(\mu,\delta)=1$.
As usual, this integral does not converge absolutely, but it does converge in a nice enough manner for a nice choice of coordinates.

\item \emph{The Mellin-Barnes Integrals Conjecture.}
We extend the usual definition of a Mellin-Barnes integral a bit to include finite sums of multi-dimensional inverse-Mellin transforms of quotients of gamma functions.
(Some of the variables in the inverse-Mellin transforms will actually be one.)
It is a fact that the $J_w$ functions can be written as Mellin-Barnes-type integrals where the contours are not vertical, but horizontal; the conjecture here is that certain linear combinations of the $J_w$ function, and in particular the $K_w$ function, can be written as Mellin-Barnes integrals where the contours are vertical, though still bent slightly for convergence.
\end{itemize}

\subsection{Caveats}
Now, where we need exceptions to the conjectures:
\begin{enumerate}
\item The higher weight case for $w_{22}$.
Though we prove the Weak Interchange of Integrals (existence of an integral kernel for the defining integral transform) for the $w_{22}$ Weyl element, it appears the Strong Interchange of Integrals (rapid convergence of the defining integral) is likely just false in this case.
In hindsight, this is perhaps not too surprising and follows from two difficulties:
$\wbar{U}_{22}$ is symmetric about the reverse diagonal and has minimal presence (a single nonzero entry) on the first super-diagonal, leading to a small number of terms in the phase of the oscillatory integral.

The computations given here cover the principal-series (i.e. $\Lambda=0$) case, but since the Weak Interchange of Integrals does not imply the Analytic Continuation Conjecture, this means we need to extend the Asymptotics Theorem to the generalized principal series case and argue that the Differential Equations and Power Series Conjecture holds in that case, as well; this is the approach we use for $w_{22}$ in \cref{sect:w22hiccup}.
For the former, the author only retreated in face of notational complexities (and the hope that it was unnecessary), and the latter should also pose no trouble as this is a standard fact about Frobenius series solutions.

\item The complete solution to the differential equations in the $w_{211}$ case.
In the other cases, we are able to show the Frobenius series span the solution set by dimensionality arguments -- i.e. showing the dimension of the solution space is equal to the number (dimension) of Frobenius series solutions, but for $w_{211}$, though we have a complete set of Frobenius series solutions, the author was unable to compute the dimension of the full space of solutions.

It may be possible to show (directly from the definition) that the Bessel functions are linear combinations of Frobenius series, as one would expect.
Then clearly those Frobenius series must be solutions of the differential equations -- this bypasses the need to discuss the full solution space.

One might hope to accomplish this by leveraging the Interchange of Integrals against the known power series and Mellin-Barnes expansions of the Whittaker function (using \cite[eq. (49)]{GLnI}) and the complex exponential, but the author was unable to solve a technical detail in this approach, and instead \cref{sect:w211hiccup} proceeds by a (messy) Mellin-Barnes integral representation for the product of the Bessel and Whittaker functions.
\end{enumerate}
The proofs in \cref{sect:hiccups} are less than precise because the full proof would be quite lengthy, and in any case, we anticipate subsequent papers will make these issues irrelevant.

Finally, two remarks:
\begin{enumerate}[topsep=0pt]
\item The goal for applications are the integral representations as in \cref{sect:IntRepns}.
For every Weyl element of $GL(4)$, as described above, we show that the integral representation of \cite[Sect. 5.5.1]{GLnI} holds and in fact, it holds with $C(\mu,\delta)=1$.
(Though we don't prove it here, the same applies for $GL(3)$ and $GL(2)$.)
A simple heuristic justification is that the Bruhat-based $I^B_{\mu,\delta}$ power function (see \cref{sect:Chars}) is a limit of $K$-finite power functions $I^I_{\mu,\delta,\sigma}$ and the Jacquet integral of $I^B_{\mu,\delta}$ is essentially the Whittaker-like power function $I^W_{\mu,\delta}$.
Thus the integral representation of the Bessel function would follow, if one could justify the necessary interchanges of this limit and the relevant integrals.
It seems easier, in the general case, to simply prove the heuristic.

Also, one might consider to simply construct the Kuznetsov formula by replacing the Whittaker function in the inverse Whittaker transform (see \cite[Sect. 3]{GLnI}) with $I^W_{\mu,\delta}$.
This would avoid the Interchange of Integrals Conjecture, but of course, this also loses Whittaker inversion from Wallach's Whittaker Plancherel formula \cite{Wallach} (and may interfere with the localization) and so the spectral side of the formula would then require a similar interchange of limit process.

\item The power series for the higher-rank Bessel functions should tend to involve multi-dimensional hypergeometric functions at 1.
Because of this, it may seem like the arguments here could not be generalized, but the author would like to point out that the well-known contiguous relations for one-dimensional hypergeometric series imply the same for the multi-dimensional case.
\end{enumerate}

\section{Background}
We use coordinates
\begingroup
\allowdisplaybreaks
\begin{align}
\label{eq:GwlCoordsX}
	x=X(x_1,x_2,x_3,x_4,x_5,x_6):=&\Matrix{1&x_1&x_2&x_4\\&1&x_3&x_5\\&&1&x_6\\&&&1}, \\ \label{eq:GwlCoordsY}
	y=Y(y_1,y_2,y_3,y_4):=&\Matrix{y_1 y_2 y_3 y_4\\&y_2 y_3 y_4\\&&y_3 y_4\\&&&y_4},
\end{align}
\endgroup
as well as
\[ K_{ij}(x) := \Matrix{I_{i-1}\\&\frac{1}{\sqrt{1+x^2}}&&-\frac{x}{\sqrt{1+x^2}}\\&&I_{j-i-1}\\&\frac{x}{\sqrt{1+x^2}}&&\frac{1}{\sqrt{1+x^2}}\\&&&&I_{4-j}}. \]

On $Y$, it is sometimes useful to instead apply the coordinates
\begin{align}
\label{eq:GwlCoordsA}
	y = A(a_1,a_2,a_3,a_4) := \Matrix{a_1/a_2\\&a_2/a_3\\&&a_3/a_4\\&&&a_4}.
\end{align}

The Weyl elements are permutation matrices, and when describing elements of $W \cong S_4$ we occasionally use cycle notation, e.g.
\[ w_{(1\,4\,2\,3)} = \Matrix{&&1\\&&&1\\&1\\1} = w_{2,1,1}, \]
which comes from the action on the standard basis
\[ e_{w_\sigma(i)} := w_\sigma e_i = e_{\sigma(i)} = e_{i^{\sigma^{-1}}} = e_{i^{w_\sigma^{-1}}}, \qquad \sigma \in S_4. \]

\subsection{Characters}
\label{sect:Chars}
We note that characters of $U(\R)$ have the form
\[ \psi_m\Matrix{1&x_1&x_2&x_4\\&1&x_3&x_5\\&&1&x_6\\&&&1} = \e{m_1 x_1+m_2 x_3+m_3 x_6}, \qquad \e{t} = e^{2\pi i t}, \]
for some $m \in \R^3$.
For $y \in Y$, we also use the notation $\psi_y(x) = \psi_I(yxy^{-1}) = \psi_{1,1,1}(yxy^{-1})$.

For the multiplicative character $\chi_s^\ell$ of $\R^\times$, notice that we have the usual properties of power functions
\[ \chi_s^\ell(a) \chi_{s'}^{\ell'}(a) = \chi_{s+s'}^{\ell+\ell'}(a), \qquad \chi_s^\ell(a) \chi_s^\ell(a') = \chi_s^\ell(a a'), \]
and all of the functions $I^I_{\mu,\delta}$, $I^B_{\mu,\delta}$, and $I^W_{\mu,\delta}$ are eigenfunctions of the Casimir operators, essentially for the same reason:
In an expression
\[ E_{i_1, i_2} \circ E_{i_2, i_3} \circ \ldots \circ E_{i_r,i_1}, \]
if any $E_{ij}$ has $j>i$, then another must have $j<i$, and vice versa, but after conjugating across $\trans{u}$ or $k$ (as we may with bi-$G$-invariant operators), the restrictions to $UY$ have $E_{ij} I^I_{\mu,\delta,\sigma} = 0$ if $j>i$, $E_{ij} I^W_{\mu,\delta}=0$ if $j < i$ and $E_{ij} I^B_{\mu,\delta} = 0$ if $j \ne i$.

We sometimes express the power function in the coordinates \eqref{eq:GwlCoordsA} as
\begin{align*}
	I_{\mu,\delta}(A(\pm1,a_2,a_3,a_4)) &= (\pm1)^{\delta_1} \wtilde{I}_{\mu,\delta}(a_2,a_3,a_4).
\end{align*}

\subsection{Gamma Functions}
\label{sect:GammaFunctions}
In addition to the $G_\delta(s)$ function defined above, we will need
\[ R_\eta(s) = i^{\eta'} \cos\frac{\pi (s-\eta')}{2}. \]
Notice that
\begin{align}
\label{eq:GdToRGamma}
	G_\eta(s) =& 2 (2\pi)^{-s} R_\eta(s) \Gamma(s).
\end{align}

From the reflection property and known residues of the gamma function as well as the usual properties of trigonometric functions, we have the following useful identities:
\begingroup
\allowdisplaybreaks
\begin{align}
\label{eq:GdShift}
	G_\ell(s,t,\eta) =& G_\ell(s+u,(t_1-u,\ldots,t_n-u),\eta), \\
\label{eq:GdReflect}
	G_\eta(s) G_\eta(1-s) =& (-1)^\eta \\
\label{eq:RdShift}
	R_\eta(s+n) =& i^n R_{\eta+n}(s), \qquad n \in \Z, \\
\label{eq:RdReflect}
	R_\eta(s) =& (-1)^\eta R_\eta(-s), \\
	R_\eta(s) R_\eta(1-s) =& \frac{1}{2}(-1)^\eta \sin(\pi s), \\
	\res_{s=n} \frac{1}{R_\eta(s)} =& \piecewise{\frac{2i}{\pi} (-1)^\eta i^n & \If n \equiv \eta+1 \pmod{2}, \\ 0 & \Otherwise,} \qquad n \in \Z, \\
\label{eq:RGtoPoch}
	R_\eta(s) \Gamma(s+n) =& R_\eta(s) \Gamma(s) \Poch{s}{j} = \frac{1}{2} (2\pi)^s G_\eta(s) \Poch{s}{j},
\end{align}
\endgroup
where $\Poch{s}{j} = \Gamma(s+j)/\Gamma(s)$ is the (rising) Pochhammer symbol.
Furthermore, $G_\eta(s)$ is entire except for simple poles at each $s=-n$, $n \in \N_0$, $n \equiv \eta\pmod{2}$ with residue
\[ \res_{s=-n} G_\eta(s) = 2\frac{(2 \pi i)^n}{n!}. \]
The poles of $1/G_\eta(s)$ are at $s=n$, $n \in \N$, $n\equiv \eta+1\pmod{2}$ by the symmetry \eqref{eq:GdReflect}.

We note that Stirling's formula implies
\begin{align*}
	\log\abs{\Gamma(\sigma+it)} =& \log\Gamma(\sigma)+\paren{\sigma-\frac{1}{2}}\log\abs{1+i\frac{t}{\sigma}}-\abs{t}\arctan\frac{\abs{t}}{\sigma}+\BigO{\frac{1}{\sigma}},
\end{align*}
for $\sigma > 0$ and $t \in \R$, so that
\begin{align}
\label{eq:Stirling}
	\abs{\Gamma(\sigma+it)} \asymp_\epsilon& \Gamma(\sigma) \abs{1+i\frac{t}{\sigma}}^{\sigma-\frac{1}{2}} \exp\paren{-\abs{t}\arctan\frac{\abs{t}}{\sigma}},
\end{align}
for $\sigma > \epsilon > 0$ and $t \in \R$.
We refer to $\Gamma(\sigma)$ as the ``factorial part'', the factor $\abs{1+i\frac{t}{\sigma}}^{\sigma-\frac{1}{2}}$ as the ``polynomial part'', and the factor $\exp\paren{-\abs{t}\arctan\frac{\abs{t}}{\sigma}}$ as the exponential part of Stirling's formula.

\subsection{Hypergeometric Series}
The (generalized) hypergeometric series for $p,q \in \N_0$, $a \in \C^p, b \in \C^q, z \in \C$ is defined by
\[ \pFq{p}{q}{a_1,\ldots,a_p}{b_1,\ldots,b_q}{z} := \sum_{n=0}^\infty \frac{z^n}{n!} \frac{\Poch{a_1}{n}\cdots\Poch{a_p}{n}}{\Poch{b_1}{n}\cdots\Poch{b_q}{n}}, \]
provided the series converges.
We sometimes use the completed (``regularized'') series
\[ \pFqStar{p}{q}{a_1,\ldots,a_p}{b_1,\ldots,b_q}{z} := \sum_{n=0}^\infty \frac{z^n}{n!} \frac{\Poch{a_1}{n}\cdots\Poch{a_p}{n}}{\Gamma(b_1+n)\cdots\Gamma(b_q+n)} = \pFq{p}{q}{a_1,\ldots,a_p}{b_1,\ldots,b_q}{z}/\prod_{i=1}^q \Gamma(b_i), \]
which has no poles (when it converges).
Finally, the fully completed
\[ \pFqDagger{p}{q}{a_1,\ldots,a_p}{b_1,\ldots,b_q}{z} := \paren{\prod_{i=1}^p \Gamma(s)} \pFqStar{p}{q}{a_1,\ldots,a_p}{b_1,\ldots,b_q}{z}, \]
will prove convenient for recurrence relations.

Some useful hypergeometric identities are
\begin{align}
\label{eq:Gauss2F1}
	\pFqStar21{a_1,a_2}{b_1}{1} =& \frac{\Gamma(b_1-a_1-a_2)}{\Gamma(b_1-a_1)\Gamma(b_1-a_2)} \qquad \text{\cite[eq. (1)]{Bailey}}, \\
\label{eq:pFqFirstZero}
	\pFq{p}{q}{0,a_2,\ldots,a_p}{b_1,\ldots,b_q}{z} =& 1,
\end{align}
for $-1 \le n \in \Z$ with $a,b,c,d,e,f \notin \Z$,
\begin{align}
\label{eq:4F3Denom1}
	& \pFqStar43{a,b,c,d}{-n,e,f}{1} \\
	&= \sum_{k=n+1}^\infty \frac{\Poch{a}{k} \Poch{b}{k} \Poch{c}{k} \Poch{d}{k}}{\Gamma(k+1) \Gamma(k-n) \Gamma(e+k) \Gamma(f+k)} \nonumber \\
	&= \sum_{k=0}^\infty \frac{\Gamma(a+k+n+1) \Gamma(b+k+n+1) \Gamma(c+k+n+1) \Gamma(d+k+n+1)}{\Gamma(k+n+2) \Gamma(k+1) \Gamma(e+k+n+1) \Gamma(f+k+n+1) \Gamma(a)\Gamma(b)\Gamma(c)\Gamma(d)} \nonumber \\
	&= \Poch{a}{n+1}\Poch{b}{n+1}\Poch{c}{n+1}\Poch{d}{n+1} \, \pFqStar43{1+n+a,1+n+b,1+n+c,1+n+d}{2+n,1+n+e,1+n+f}{1}, \nonumber
\end{align}
for $m,n \in \N_0$ with $a,b,c,d,e \notin \Z$,
\begin{align}
\label{eq:4F3Denom2}
	\pFqStar43{-m,a,b,c}{-m-n,d,e}{1}=0.
\end{align}

\subsection{Mellin-Barnes Integrals}
If we define the signed Mellin transform
\begin{equation}
\label{eq:SignedMDef}
\begin{aligned}
	\hat{f}(s,\ell) :=& \int_{\R} f(x) \chi_{s-1}^\ell(x) dx \\
	=& \int_0^\infty \paren{f(x) +(-1)^\ell f(-x)} x^{s-1} dx,
\end{aligned}
\end{equation}
then Mellin inversion becomes
\begin{align}
\label{eq:SignedMInv}
	f(x) = \frac{1}{2} \sum_{\ell \in \set{0,1}} \int_{\Re(s)=\frac{1}{2}} \hat{f}(s,\ell) \chi_{-s}^\ell(x) \frac{ds}{2\pi i}.
\end{align}

As in \cite[eq. (3.17)]{HWI}, we have
\begin{equation}
\label{eq:PsiThetaInvMellin}
\begin{aligned}
	\e{x} =& \lim_{\theta\to\frac{\pi}{2}^-} \int_{\Re(s) = c} \abs{2\pi x}^{-s} e^{is\theta \sgn(x)} \Gamma\paren{s} \, \frac{ds}{2\pi i}, \\
	=& \lim_{\theta\to\frac{\pi}{2}^-} \frac{1}{2}\sum_{\ell \in \set{0,1}} \int_{\Re(s) = c} \frac{\cos(\theta s-\frac{\pi}{2}\ell)}{R_\ell(s)} \chi_{-s}^\ell(x) G_\ell(s) \, \frac{ds}{2\pi i}
\end{aligned}
\end{equation}
for $x \ne 0$ and $c > 0$, so in a formal sense,
\[ \what{e}(s,\ell) = G_\ell(s). \]

For $a,b\in\R^\times$ and $\abs{\Re(\mu)} < 1$,
\begin{align}
\label{eq:mcZInt}
	\int_{-\infty}^\infty \chi_{s-1}^\eta(x) \e{a x+\frac{b}{x}} dx =& \chi_{-\frac{s}{2}}^\eta(a) \chi_{\frac{s}{2}}^0(b) \mathcal{Z}_s^\eta(ab).
\end{align}
which follows from \cite[3.871.1-4, 8.403.1]{GradRyzh}.
Note that $\mathcal{Z}_s^{\eta+2}(a)=\mathcal{Z}_s^\eta(a)$.

For $a>0$ and $\abs{\Re(s)} < 2c < 1$, by \cite[17.43.16, 17.43.18]{GradRyzh},
\begin{align*}
	\mathcal{Z}_t^\eta(-a) =& 2 i^\delta \int_{\Re(s)=c} \cos\frac{\pi}{2}\paren{t-\eta} \Gamma\paren{s+\frac{t}{2}} \Gamma\paren{s-\frac{t}{2}} \paren{4\pi^2 a}^{-s} \frac{ds}{2\pi i}, \\
	\mathcal{Z}_t^\eta(+a) =& 2 i^\delta \int_{\Re(s)=c} \cos\frac{\pi}{2}\paren{2s-\eta} \Gamma\paren{s+\frac{t}{2}} \Gamma\paren{s-\frac{t}{2}} \paren{4\pi^2 a}^{-s} \frac{ds}{2\pi i}.
\end{align*}
It follows that for $\eta \in \set{0,1}, \ell \in \Z$,
\begin{align*}
	\what{\mathcal{Z}}_t^\eta(s,\ell) =& G_\ell\paren{s,\paren{\frac{t}{2},-\frac{t}{2}},(\eta,0)}.
\end{align*}
Furthermore, for any $\delta_1,\delta_2 \in \set{0,1}, \ell \in \Z$, we have the more symmetric
\begin{align}
\label{eq:ZMellin}
	\what{\mathcal{Z}}_t^{\eta_1+\eta_2}(s,\ell+\eta_2) =& G_\ell\paren{s,\paren{\frac{t}{2},-\frac{t}{2}},(\eta_1,\eta_2)}.
\end{align}
Note: In the case of the $J$-Bessel function, i.e. $\mathcal{Z}_t^\eta(+a)$, the Mellin expansion converges absolutely if the tails of the contour are bent to pass to the left of $\Re(s)=0$, which implies the integral over $\Re(s)=c$ converges conditionally by contour shifting.
The $K$-Bessel function case has exponential convergence.

From the definition of the Euler beta function $B(a,b)$ and some trigonometry, we have
\begin{equation}
\label{eq:BetaEval}
\begin{aligned}
	& \int_{\R} \chi_{s_1-1}^{\eta_1}(x) \chi_{s_2-1}^{\eta_2}(1-x) dx \\
	&= B(s_1,s_2)+(-1)^{\eta_2} B(1-s_1-s_2,s_2)+(-1)^{\eta_1} B(s_1,1-s_1-s_2) \\
	&= \frac{G_{\eta_1}(s_1) G_{\eta_2}(s_2)}{G_{\eta_1+\eta_2}(s_1+s_2)}.
\end{aligned}
\end{equation}

\subsection{The $J$-Bessel Power Series}
Let $w \in W^\text{rel}$.
Define $Y_w$ to be the set of $Y$ matrices which satisfy the ``compatibility condition'' $\psi_y(w u w^{-1}) = \psi_I(u)$ for all $u \in U_w(\R)$, and set $G_w := U(\R)Y_w \paren{w \wbar{U}_w(\R) w^{-1}}$.
Then the $J$-Bessel power series (Frobenius series) is defined on $G_w$ as the solution to the Bessel differential equations \cite[Sect. 4]{GLnI} satisfying
\begin{align}
	J_w^*(uy(wu'w^{-1}), \mu) =& \psi_I(uu') J_w(y, \mu), & u \in U(\R), u' \in \wbar{U}_w(\R), \\
	J_w^*(y,\mu) \sim& I_{\mu,0}(y)
\end{align}
as $y \to 0$ along the (non-fixed) coordinates of $Y_w$.

Lastly, we define
\begin{align*}
	J_w(y,\mu) = J_w^*(y,\mu)/\Lambda_w(\mu)
\end{align*}
and
\begin{align}
\label{eq:JwmudeltaDef}
	J_w(y,\mu,\delta) :=& I_{0,\delta}(y) J_w(y,\mu),
\end{align}
where
\begin{align}
	\Lambda_w(\mu^{w_l}) = \prod_{\substack{j<k\\ k^w < j^w}} (2\pi)^{\mu_j-\mu_k} \Gamma\paren{1+\mu_k-\mu_j}.
\end{align}

\subsection{The Asymptotics Theorem}
\label{sect:Asymptotics}
The Asymptotics Theorem of \cite[Sect. 4]{GLnI} is an evaluation of the constants $C_w(\mu,\delta)$ occurring in the power series expansion
\begin{align}
\label{eq:KwToJw}
	K_w(y,\mu,\delta) =& \sum_{w' \in W/W_w} C_w(\mu^{w'},\delta^{w'}) J_w(y,\mu^{w'},\delta^{w'}).
\end{align}
The theorem applies under the assumption of the Interchange of Integrals and Differential Equations and Power Series Conjectures, which are proved in this paper.
The somewhat unwieldy expression obtained is as follows:
Define
\[ C_w^*(\mu^{w_l},\delta^{w_l}) = \prod_{\substack{j < k\\k^w < j^w}} (-1)^{\delta_k} G_{\delta_j+\delta_k}(\mu_j-\mu_k), \]
then $C_w(\mu,\delta) = \Lambda_w(\mu) C_w^*(\mu,\delta)$.

Explicitly, these combine to
\begin{align}
\label{eq:CwToR}
	C_w(\mu,\delta) =& \prod_{(j,k) \in \mathcal{S}_w} \frac{(-1)^{\delta_j} \pi}{R_{\delta_j+\delta_k}(1+\mu_j-\mu_k)},
\end{align}
where
\begin{align*}
	\mathcal{S}_{31} =& \set{(1,4),(2,4),(3,4)}, \\
	\mathcal{S}_{22} =& \set{(1,3),(2,3),(1,4),(2,4)}, \\
	\mathcal{S}_{121} =& \set{(1,2),(1,3),(1,4),(2,4),(3,4)}, \\
	\mathcal{S}_{211} =& \set{(1,3),(1,4),(2,3),(2,4),(3,4)}, \\
	\mathcal{S}_{1111} =& \set{(1,2),(1,3),(1,4),(2,3),(2,4),(3,4)}.
\end{align*}

\section{Matrix Decompositions}
These can be obtained from the appendix to \cite{GLnI} for the Iwasawa decomposition, while the Bruhat decomposition can be obtained from the GL(n)pack of Broughan, which appears in an appendix to \cite{Gold01}, and they can be verified directly.

The subgroups of $U$ for each Weyl element are
\begin{gather*}
\begin{aligned}
	\wbar{U}_{31} =& \Matrix{1&x_1&x_2&x_4\\&1&0&0\\&&1&0\\&&&1}, &
	\wbar{U}_{22} =& \Matrix{1&0&x_2&x_4\\&1&x_3&x_5\\&&1&0\\&&&1}, \\
	\wbar{U}_{121} =& \Matrix{1&x_1&x_2&x_4\\&1&0&x_5\\&&1&x_6\\&&&1}, &
	\wbar{U}_{211} =& \Matrix{1&x_1&x_2&x_4\\&1&x_3&x_5\\&&1&0\\&&&1}
\end{aligned}
\end{gather*}

\subsection{Iwasawa Decompositions}
\label{sect:Iwasawa}
For convenience of notation, we write
\[ \xi_A := \prod_{i \in A} \sqrt{1+x_i^2}, \]
and we drop the set notation and commas so that, e.g.
\[ \xi_{136} = \sqrt{1+x_1^2}\sqrt{1+x_3^2}\sqrt{1+x_6^2}. \]

After carefully choosing the coordinates of $x \in \wbar{U}_w$ as in \cite[Sect. 8]{GLnI}, the Iwasawa decompositions of each $wx$ are
\begingroup
\allowdisplaybreaks
\begin{align*}
	& w_{31} X\paren{x_1,x_2\xi_1,0,x_4\xi_{12},0,0} = \\*
	& \qquad X\paren{-\frac{x_1 x_2}{\xi_1},-\frac{x_1 x_4}{\xi_{12}},-\frac{x_2 x_4}{\xi_2},\frac{x_1}{\xi_{124}^2},\frac{x_2}{\xi_1\xi_{24}^2}, \frac{x_4}{\xi_{12}\xi_4^2}} Y\paren{\frac{\xi_2}{\xi_1},\frac{\xi_4}{\xi_2}, \frac{1}{\xi_{12} \xi_4^2}, \xi_{124}} \\*
	& \qquad \times K_{34}(x_4) K_{24}(x_2) K_{14}(x_1) w_{31} \\
	& w_{22} X\paren{0,x_2 \xi_3, x_3, x_2 x_3 x_5+x_4\xi_{25},x_5 \xi_3, 0} = \\*
	& \qquad X\paren{-\frac{x_3 x_5}{\xi_3}-\frac{x_2 x_4 \xi_5}{\xi_{23}},-\frac{x_3 x_4 x_5}{\xi_{235}\xi_4^2}+\frac{x_2}{\xi_3 \xi_{24}^2},\frac{x_4}{\xi_{25} \xi_4^2}, \frac{x_3}{\xi_{35}^2},\frac{x_5}{\xi_3 \xi_5^2},\frac{x_2 x_3}{\xi_3}+\frac{x_4 x_5 \xi_2}{\xi_{35}}} \\*
	& \qquad \times Y\paren{\frac{\xi_{45}}{\xi_{23}},\frac{1}{\xi_{25}\xi_4^2}, \frac{\xi_{24}}{\xi_{35}},\xi_{35}} K_{23}(x_4) K_{13}(x_2) K_{24}(x_5) K_{14}(x_3) w_{22} \\
	& w_{121} X\paren{x_1, x_2\xi_1, 0, x_4 \xi_{12}, \frac{x_1 x_4- x_1 x_2 x_6\xi_4}{\xi_{12}}+\frac{x_5\xi_{46}}{\xi_1}, \frac{x_2 x_4+x_6 \xi_4}{\xi_2}} = \\*
	& \qquad X\paren{\frac{x_5 \xi_1}{\xi_{46} \xi_5^2},\frac{x_6\xi_2}{\xi_4 \xi_6^2},\frac{x_5 x_6 \xi_2}{\xi_{16}}-\frac{x_1 x_2}{\xi_1}, \frac{x_4}{\xi_{12}\xi_4^2}, \frac{x_1}{\xi_{12}^2} - \frac{x_1 x_2 x_4 x_6}{\xi_4 \xi_{12}^2}+\frac{x_4 x_5\xi_6}{\xi_{24}\xi_1^2}, \frac{x_2}{\xi_1 \xi_2^2} + \frac{x_4 x_6}{\xi_{14} \xi_2^2}} \\*
	& \qquad \times Y\paren{\frac{\xi_1}{\xi_{46}\xi_5^2},\frac{\xi_{25}}{\xi_{16}}, \frac{\xi_6}{\xi_{14}\xi_2^2},\xi_{124}} K_{12}(x_5) K_{13}(x_6) K_{14}(x_4) K_{34}(x_2) K_{24}(x_1) w_{121} \\
	& w_{211} X\paren{x_1,x_2\xi_1,\frac{x_1 x_2 + x_3\xi_2}{\xi_1},x_4\xi_{12},\frac{x_2 x_3 x_4+x_1 x_4\xi_2+x_5\xi_{34}}{\xi_1},0} = \\*
	& \qquad X\paren{-\frac{x_2 x_4}{\xi_2}-\frac{x_3 x_5 \xi_4}{\xi_{23}},\frac{x_3\xi_1}{\xi_2 \xi_{35}^2}-\frac{x_2 x_4 x_5\xi_1}{\xi_{234}\xi_5^2},\frac{x_5\xi_1}{\xi_{34}\xi_5^2},\frac{x_2}{\xi_{14}\xi_2^2},\frac{x_4}{\xi_{12}\xi_4^2},\frac{x_1}{\xi_1^2}+\frac{x_2 x_3}{\xi_2 \xi_1^2}+\frac{x_4 x_5\xi_3}{\xi_{24} \xi_1^2}} \\*
	& \qquad \times Y\paren{\frac{\xi_{45}}{\xi_{23}},\frac{\xi_1}{\xi_{34}\xi_5^2},\frac{\xi_{35}}{\xi_{24}\xi_1^2},\xi_{124}} K_{23}(x_5) K_{13}(x_3) K_{24}(x_4) K_{14}(x_2) K_{34}(x_1) w_{211} \\
	& w_{1111} X\paren{x_1, x_2 \xi_1, \frac{x_1 x_2+x_3 \xi_2}{\xi_1}, x_4 \xi_{12}, \frac{x_2 x_3 x_4+x_1 x_4 \xi_2+x_5 \xi_{34}}{\xi_1}, \frac{x_2 x_4 \xi_3+x_3 x_5 \xi_4+x_6 \xi_{45}}{\xi_{23}}} = \\*
	& \qquad X\paren{\frac{x_6 \xi_{23}}{\xi_{45} \xi_6^2}, \frac{x_5 \xi_1}{\xi_{34} \xi_5^2}, \frac{x_3 \xi_{15}+x_5 x_6 \xi_1}{\xi_{25} \xi_3^2}, \frac{x_4}{\xi_{12} \xi_4^2}, \frac{x_3 x_4 x_5+x_2 \xi_{34}+x_4 x_6 \xi_5}{\xi_{134}\xi_2^2}, \frac{x_1 \xi_{24}+x_2 x_3 \xi_4+x_4 x_5 \xi_3}{\xi_{24} \xi_1^2}} \\*
	& \qquad \times Y\paren{\frac{\xi_{23}}{\xi_{45}\xi_6^2}, \frac{\xi_{16}}{\xi_{25}\xi_3^2}, \frac{\xi_{35}}{\xi_{24}\xi_1^2}, \xi_{124}} K_{12}(x_6) K_{13}(x_5) K_{23}(x_3) K_{14}(x_4) K_{24}(x_2) K_{34}(x_1) w_{1111}
\end{align*}
\endgroup

\subsection{Bruhat Decompositions}
Note that we can compute the decomposition $g = xy\trans{u}$ from the Bruhat decomposition of $g w_l = xywx'$ by setting $\trans{u}=w_l x' w_l$ when $g w_l$ lies in the long-element Weyl cell $w=w_l$, which is true outside a set of measure zero.
We start by parameterizing $G_{w_l} = U(\R) Y \trans{U(\R)}$ with coordinates
\begin{align}
\label{eq:GwlCoordsG}
	g=& xy\trans{u},
\end{align}
$x,u \in U(\R), y \in Y$ using the coordinates \eqref{eq:GwlCoordsX} and similar for $u$ and for $y$, we use the coordinates \eqref{eq:GwlCoordsA}.

The Bruhat decompositions of $wx$ for $x \in \wbar{U}_w$ (using the original coordinates) are
\begingroup
\allowdisplaybreaks
\begin{align*}
	w_{31} x =& X\paren{-\frac{x_2}{x_1},0,-\frac{x_4}{x_2},0,0,\frac{1}{x_4}} A\paren{-1,x_1-x_2,x_4}\trans{X\paren{\frac{1}{x_1},\frac{1}{x_2},\frac{x_1}{x_2},\frac{1}{x_4},\frac{x_1}{x_4},\frac{x_2}{x_4}}}, \\
	w_{22} x =& X\paren{-\frac{x_4}{x_2},-\frac{x_5}{\zeta_1},\frac{x_3}{\zeta_1},0,\frac{1}{x_5},\frac{x_4}{x_5}} A\paren{1,-x_2,-\zeta_1,x_5} \trans{X\paren{-\frac{x_3}{x_2},-\frac{x_5}{\zeta_1},\frac{x_4}{\zeta_1},0,\frac{1}{x_5},\frac{x_3}{x_5}}}, \\
	w_{121} x =& X\paren{-\frac{x_1}{\zeta_3},-\frac{x_2}{\zeta_2},-\frac{x_2 x_5}{\zeta_2},\frac{1}{x_4},\frac{x_5}{x_4},\frac{x_6}{x_4}} A\paren{-1,\zeta_3,\zeta_2,x_4} \\*
	& \qquad \times \trans{X\paren{-\frac{x_5}{\zeta_3},-\frac{x_6}{\zeta_2},-\frac{x_1 x_6}{\zeta_2},\frac{1}{x_4},\frac{x_1}{x_4},\frac{x_2}{x_4}}}, \\
	w_{211} x =& X\paren{-\frac{\zeta_4}{\zeta_5},\frac{x_4}{\zeta_1},-\frac{x_2}{\zeta_1},0,\frac{1}{x_4},\frac{x_5}{x_4}} A\paren{-1,\zeta_5,\zeta_1,x_4} \trans{X\paren{-\frac{x_3}{\zeta_5},-\frac{x_5}{\zeta_1},\frac{\zeta_4}{\zeta_1},\frac{1}{x_4},\frac{x_1}{x_4},\frac{x_2}{x_4}}} \\
	w_{1111} x =& X\paren{-\frac{\zeta_5}{\zeta_6}, -\frac{x_2}{\zeta_1},\frac{\zeta_2}{\zeta_1},\frac{1}{x_4}, \frac{x_6}{x_4}, \frac{x_5}{x_4}} A\paren{1,-\zeta_6,\zeta_1,x_4} \trans{X\paren{-\frac{\zeta_7}{\zeta_6},-\frac{x_5}{\zeta_1},\frac{\zeta_4}{\zeta_1},\frac{1}{x_4},\frac{x_1}{x_4},\frac{x_2}{x_4}}},
\end{align*}
where
\begin{align*}
	\zeta_1 =& x_3 x_4-x_2 x_5, & \zeta_2 =& x_4-x_2 x_6, & \zeta_3 =& x_4-x_1 x_5-x_2 x_6, & \zeta_4 =& x_4-x_1 x_5,  \\
	\zeta_5 =& x_2 - x_1 x_3, & \zeta_6 =& x_4-x_1 x_5 - x_2 x_6 + x_1 x_3 x_6, & \zeta_7 =& x_5-x_3 x_6.
\end{align*}
\endgroup

\section{Differential Equations and Power Series}
\label{sect:DEPS}
In this section, we give constructions (proofs) on $GL(4)$ of the Bessel functions described in the Differential Equations and Power Series Conjecture, ignoring the degenerate cases $\mu_i - \mu_j \in \Z, i \ne j$.
As mentioned above, for brevity, we drop commas in subscripts where possible, i.e. writing $E_{11}$ rather than $E_{1,1}$


Parameterizing $G_{w_l}$ using the coordinates of \eqref{eq:GwlCoordsG} with \eqref{eq:GwlCoordsX} for $x$ and $u$, while using \eqref{eq:GwlCoordsY} for $y$ at $\abs{y_4}=1$ (i.e. in $GL(4,\R)/\R^+$), the differential operators $E_{ij}$ acting on functions of $G_{w_l}$ have the form
\begin{align}
	\paren{E_{ij}}_{i,j} =& \Matrix{
	y_1 \partial_{y_1} & y_1 \partial_{x_1} & y_1 y_2 \partial_{x_2} & y_1 y_2 y_3 \partial_{x_4}\\
	\partial_{u_1} & y_2 \partial_{y_2}-y_1 \partial_{y_1} & y_2 \partial_{x_3}+y_2 x_1 \partial_{x_2} & y_2 y_3 \partial_{x_5}+y_2 y_3 x_1 \partial_{x_4}\\
	\partial_{u_2} & \partial_{u_3}+u_1 \partial_{u_2} & y_3 \partial_{y_3}-y_2 \partial_{y_2} & y_3 \partial_{x_6}+y_3 x_2 \partial_{x_4}+y_3 x_3 \partial_{x_5}\\
	\partial_{u_4} & \partial_{u_5}+u_1 \partial_{u_4} & \partial_{u_6}+u_2 \partial_{u_4} + u_3 \partial_{u_5} & -y_3 \partial_{y_3}},
\end{align}
which may be computed either directly (write $xy\trans{u}\exp (tE) = xy \exp(t\trans{u}E u^{-T}) \trans{u}$ and then write $\trans{u}E u^{-T} = \sum_{i,j} f_{ij}(u) E_{ij}$; the $UY\trans{U}$ decomposition of $xy \exp(tE_{ij}) \trans{u}$ is trivial) or using the appropriate modification of \cite[App. B]{HyperKl}.

If we renormalize the Casimir operators slightly to define
\begin{align*}
	\Delta_1 =& \sum_{1 \le i \le 4} E_{ii}, \\
	-2\Delta_2 =& \sum_{1 \le i,j \le 4} E_{ij} \circ E_{ji}, \\
	3\Delta_3 =& \sum_{1 \le i,j,k \le 4} E_{ij} \circ E_{jk} \circ E_{ki}+4\Delta_2, \\
	-4\Delta_4 =& \sum_{1 \le i,j,k,l \le 4} E_{ij} \circ E_{jk} \circ E_{kl} \circ E_{li}-12\Delta_3+\Delta_2,
\end{align*}
then the eigenvalues of the power function $\Delta_i I_{\mu,\delta} = \lambda_i(\mu) I_{\mu,\delta}$ are
\begin{align*}
	\lambda_1(\mu) =& 0, \\
	\lambda_2(\mu) =& \frac{5}{2}-\frac{\mu_1^2+\mu_2^2+\mu_3^2+\mu_4^2}{2}, \\
	\lambda_3(\mu) =& \mu_1\mu_2\mu_3+\mu_1\mu_2\mu_4+\mu_1\mu_3\mu_4+\mu_2\mu_3\mu_4, \\
	\lambda_4(\mu) =& \frac{41}{16}-\frac{\mu_1^4+\mu_2^4+\mu_3^4+\mu_4^4}{4}.
\end{align*}

The long Weyl element was treated in \cite[Sect. 12.1]{GLnI}, and for the other Weyl elements we apply the algorithm discussed after the Differential Equations and Power Series Conjecture in \cite[Sect. 4]{GLnI}.

\subsection{Differential equations and power series for $w_{31}$}
We have
\[ W_{31}=\set{I, w_{(1\,2)},w_{(2\,3)}, w_{(3\,2\,1)},w_{(1\,2\,3)},w_{(1\,3)}}, \]
and $Y_{31}$ is characterized by $y_1=y_2=1$.
We consider a test function of the form
\[ f(xy\trans{u}) = \e{x_1+x_2+x_3+u_6} f(y_1, y_2, y_3, u_1, u_2, u_4), \]
and find that $J_{31}(y,\mu)$ is annihilated by the operator
\begin{equation}
\label{eq:w31DE}
\begin{aligned}
	& \bigl(4(\Delta_4-\lambda_4)+2(3+2E_{44})\circ(\Delta_3-\lambda_3) \\
	& \qquad \left.-(4 E_{44}\circ(3+E_{44})+2\Delta_2+2\lambda_2-1)\circ(\Delta_2-\lambda_2)\bigr) f \right|_{x=u=I} \\
	&= \bigl(4y_3^4\partial_{y_3}^4+4\lambda_2 y_3^2 \partial_{y_3}^2+4(\lambda_3-2\lambda_2)y_3\partial_{y_3} \\
	& \qquad \left.+\paren{4\paren{16\pi^4y_3-\lambda_4}-6\lambda_3+2\lambda_2^2-\lambda_2}\bigr) f \right|_{x=u=I}
\end{aligned}
\end{equation}
whose kernel has dimension at most $4 = \abs{W/W_{31}}$.
Comparing the powers of $y_3$ not offset by a derivative $\partial_{y_3}$, it becomes somewhat more convenient to work in the variable $t=16\pi^4 y_3$, and in that variable, the power series (Frobenius series) coefficients satisfy a recurrence relation
\[ a_{m-1}+m(\mu_1-\mu_4+m)(\mu_2-\mu_4+m)(\mu_3-\mu_4+m) a_{m}=0, \]
and hence for $y \in Y_{31}$, we have \eqref{eq:Jw31}.

\subsection{Differential equations and power series for $w_{22}$}
\label{sect:w22DEs}
We have
\[ W_{22}=\set{I, w_{(1\,2)},w_{(3\,4)}, w_{(1\,2)(3\,4)}}, \]
and $Y_{22}$ is characterized by $y_1=y_3=1$.
We consider a test function of the form
\[ f(xy\trans{u}) = \e{x_1+x_2+x_3+u_6-u_3 u_4} f(y_1, y_2, y_3, u_1, u_3), \]
and find that $J_{22}(y,\mu)$ is annihilated by the operator
\begin{equation}
\label{eq:w22DE}
\begin{aligned}
	& \left.\bigl(A^2\paren{\paren{\Delta_2+\lambda_2-2A^2-5}(\Delta_2-\lambda_2)-4(\Delta_4-\lambda_4)}+(\Delta_3+\lambda_3)(\Delta_3-\lambda_3)\bigr) f \right|_{x=u=I} \\
	&= \bigl(y_2^6\partial_{y_2}^6+3y_2^5\partial_{y_2}^5+2\lambda_2 y_2^4 \partial_{y_2}^4-4\lambda_2y_2^3\partial_{y_2}^3+(4\lambda_4-\lambda_2^2+19\lambda_2-64\pi^4 y_2) y_2^2\partial_{y_2}^2 \\
	& \qquad \left.+3\paren{-4\lambda_4+\lambda_2^2-15\lambda_2+32\pi^4 y_2}y_2\partial_{y_2}+\paren{16\lambda_4-\lambda_3^2-4\lambda_2^2+52\lambda_2-96\pi^4 y_2}\bigr)f \right|_{x=u=I}, \\
	&A := E_{11}+E_{22}-2,
\end{aligned}
\end{equation}
whose kernel has dimension at most $6 = \abs{W/W_{22}}$.
In the variable $t=16\pi^4 y_2$, the power series coefficients satisfy a recurrence relation
\begin{align*}
	&2(\mu_1+\mu_2+m)(-1+2\mu_1+2\mu_2+2m) a_{m-1} =\\
	&m(2\mu_1+2\mu_2+m)(\mu_1-\mu_3+m)(\mu_1-\mu_4+m)(\mu_2-\mu_3+m)(\mu_2-\mu_4+m) a_{m},
\end{align*}
and hence for $y \in Y_{22}$, we have \eqref{eq:Jw22}.

\subsection{Differential equations and power series for $w_{121}$}
We have
\[ W_{121}=\set{I, w_{(2\,3)}}, \]
and $Y_{121}$ is characterized by $y_2=1$.
We consider a test function of the form
\[ f(xy\trans{u}) = \e{x_1+x_2+x_3+u_4+u_5-u_2 u_1} f(y_1, y_2, y_3, u_2), \]
and find that $J_{121}(y,\mu)$ is annihilated by the operators
\begin{gather}
\label{eq:w121DE1}
\begin{aligned}
	& \left.\paren{(\Delta_3-\lambda_3) +(E_{22}+E_{33})(\Delta_2-\lambda_2)} f \right|_{x=u=I} \\
	&= \bigl(y_1^3\partial_{y_1}^3-y_3^3 \partial_{y_3}^3-y_1^2 y_3 \partial_{y_1}^2 \partial_{y_3}+y_1 y_3^2 \partial_{y_1} \partial_{y_3}^2 +\lambda_2 y_1 \partial_{y_1}-\lambda_2 y_3 \partial_{y_3} \\
	& \qquad \left.-\paren{\lambda_3+8\pi^3i(y_3+y_1)}\bigr)f \right|_{x=u=I},
\end{aligned} \\
\label{eq:w121DE2}
\begin{aligned}
	& \bigl(4(\Delta_4-\lambda_4)+2(2E_{11}-3))(\Delta_3-\lambda_3) \\
	& \qquad \left.-(4E_{11}^2-12E_{11}+2\Delta_2+2\lambda_2-1)(\Delta_2-\lambda_2)\bigr) f \right|_{x=u=I} \\
	&= \bigl(4y_1^4 \partial_{y_1}^4-32\pi^3 i y_1 y_3 \partial_{y_3}+4\lambda_2 y_1^2 \partial_{y_1}^2-4(\lambda_3+2\lambda_2+8\pi^3i y_1) y_1 \partial_{y_1} \\
	& \qquad \left.+\paren{-4\lambda_4+6\lambda_3+2\lambda_2^2-\lambda_2+64\pi^3i y_1}\bigr)f \right|_{x=u=I}.
\end{aligned}
\end{gather}
Switching to the variables $t_1=8\pi^3i y_1$, $t_2=8\pi^3i y_3$ simplifies the differential operators somewhat, and it can be shown that $J_{121}(y,\mu)$ satisfies a (somewhat complicated) twelfth-order differential equation in $t_1$ alone, whose solution then determines the function in $t_2$ (up to a constant), giving (a basis of) $12 = \abs{W/W_{121}}$ solutions.
In the $t$ variables, the power series coefficients satisfy the recurrence relations
\begin{align*}
	&a_{121,m_1,m_2-1}+a_{121,m_1-1,m_2} =\\
	&\bigl(m_1(\mu_1-\mu_2+m_1)(\mu_1-\mu_3+m_1)-m_2(\mu_2-\mu_4+m_2)(\mu_3-\mu_4+m_2) \\
	& \qquad +m_1m_2(2\mu_2+2\mu_3-m_1+m_2)\bigr)a_{121,m_1,m_2}, \\
	& (\mu_1-\mu_4+m_1+m_2) a_{121,m_1-1,m_2} = m_1(\mu_1-\mu_2+m_1)(\mu_1-\mu_3+m_1)(\mu_1-\mu_4+m_1) a_{121,m_1,m_2}.
\end{align*}
Substituting
\[ a_{121,m_1,m_2} = \frac{\Gamma(1+\mu_1-\mu_4+m_1+m_2)}{m_1! \Gamma(1+\mu_1-\mu_2+m_1)\Gamma(1+\mu_1-\mu_3+m_1)\Gamma(1+\mu_1-\mu_4+m_1)} b_{m_2} \]
into the first recurrence gives
\[ b_{m_2-1} = -m_2 (\mu_1-\mu_4+m_2)(\mu_2-\mu_4+m_2)(\mu_3-\mu_4+m_2) b_{m_2}, \]
and hence for $y \in Y_{121}$, we have \eqref{eq:Jw121}.

\subsection{Differential equations and power series for $w_{211}$}
\label{sect:w211DEPS}
The story for $w_{211}$ is a bit more complicated; we have
\[ W_{211}=\set{I, w_{(1\,2)}}, \]
and $Y_{211}$ is characterized by $y_1=1$.
We consider a test function of the form
\[ f(xy\trans{u}) = \e{x_1+x_2+x_3+u_3+u_4} f(y_1, y_2, y_3, u_1), \]
and find that $J_{211}(y,\mu)$ is annihilated by the operators
\begin{gather}
\label{eq:211DEs1}
\begin{aligned}
	& \left.\paren{(\Delta_3-\lambda_3) +(E_{11}+E_{22}-2)(\Delta_2-\lambda_2)} f \right|_{x=u=I} \\
	&= \bigl(4y_3^2\partial_{y_3}^2 -y_2^3\partial_{y_2}^3-2 y_2 y_3^2 \partial_{y_2}\partial_{y_3}^2+2 y_2^2 y_3 \partial_{y_2}^2\partial_{y_3}-2y_2 y_3 \partial_{y_2}\partial_{y_3}-(\lambda_2-8\pi^2 y_3) y_2 \partial_{y_2} \\
	& \qquad \left.-\paren{\lambda_3-2\lambda_2+8\pi^2(\pi i y_2+2y_3)}\bigr)f \right|_{x=u=I},
\end{aligned} \\
\label{eq:211DEs2}
\begin{aligned}
	& \bigl(-4(\Delta_4-\lambda_4)-2(2E_{44}+1)(\Delta_3-\lambda_3) \\
	& \qquad \left.+\paren{4E_{44}^2+4E_{44}-4E_{33}-16\pi^2 y_3-9+2\Delta_2+2\lambda_2}(\Delta_2-\lambda_2)\bigr) f \right|_{x=u=I} \\
	&= \bigl(-4 y_2^3 \partial_{y_2}^3-4 y_3^4 \partial_{y_3}^4+8y_2^2y_3\partial_{y_2}^2\partial_{y_3}-8y_2y_3^2\partial_{y_2}\partial_{y_3}^2+16\pi^2 y_2^2 y_3 \partial_{y_2}^2-8y_2y_3\partial_{y_2}\partial_{y_3} \\
	& \qquad +4(4-\lambda_2+8\pi^2y_3)y_3^2\partial_{y_3}^2 -4\lambda_2 y_2\partial_{y_2}-4(\lambda_3-2\lambda_2+8\pi^2y_3)y_3\partial_{y_3} \\
	& \qquad \left. +\paren{\lambda_2(9-2\lambda_2+16\pi^2 y_3)+2\lambda_3+4\lambda_4-32\pi^2(\pi i y_2+y_3+2\pi^2 y_3^2)}\bigr)f \right|_{x=u=I}.
\end{aligned}
\end{gather}
Switching to the variables $t_1=8\pi^3i y_2$, $t_2=4\pi^2 y_3$ simplifies the differential operators somewhat, and by the method of Frobenius, any series solution of the form
\[ f(t) = \abs{t_1}^{s_1} \abs{t_2}^{s_2} \sum_{m_1=0}^\infty \sum_{m_2=0}^\infty a_{211,m_1,m_2} t_1^{m_1} t_2^{m_2} \]
satisfies $s_1=2+\mu^w_1+\mu^w_2$, $s_2 = \frac{3}{2}-\mu^w_4$ for some $w \in W$, hence there are $\abs{W/W_{211}}=12$ Frobenius series solutions (i.e. solutions that are well-behaved near $y_3=y_2=0$).
Taking $s_1=2+\mu_1+\mu_2$, $s_2 = \frac{3}{2}-\mu_4$, the coefficients in general satisfy the recurrence relations
\begingroup
\allowdisplaybreaks
\begin{gather}
\label{eq:a211recur1}
\begin{aligned}
	&a_{m_1-1,m_2}-2 \paren{\mu_1+\mu_2+m_1} a_{m_1,m_2-1} = \\*
	&\paren{-m_1 \paren{\mu_1-\mu_3+m_1} \paren{\mu_2-\mu_3+m_1}-2 m_2 \paren{\mu_1+\mu_2+m_1} \paren{\mu_3-\mu_4+m_2-m_1}} a_{m_1,m_2},
\end{aligned}\\
\begin{aligned}
	& a_{m_1-1,m_2+1}-a_{m_1,m_2-1} = \\*
	&\Bigl(2 m_2 \paren{-2\mu_4+m_2+1}+m_1 \paren{3-2 \mu_3-2\mu_4+m_1}+\paren{\mu_1+\mu_3+1} \paren{\mu_2+\mu_3+1} \\*
	& \qquad \qquad +2(\mu_1+\mu_2)(2-\mu_4)\Bigr) a_{m_1,m_2} \\*
	& \qquad -\Bigl(\paren{m_2+1} \paren{\mu_1-\mu_4+m_2+1} \paren{\mu_2-\mu_4+m_2+1} \paren{\mu_3-\mu_4+m_2+1} \\*
	& \qquad \qquad +(\mu_1+\mu_2+m_1) \paren{2m_2(2+\mu_3-\mu_4+m_2-m_1)-m_1(2+2\mu_3-m_1)} \\*
	& \qquad \qquad + m_1 \paren{\mu_3(4-\mu_4)+2\mu_1+2\mu_2+\mu_1\mu_2}+2(\mu_1+\mu_2)(1+\mu_3-\mu_4) \Bigr) a_{m_1,m_2+1}
\end{aligned}\nonumber
\end{gather}
\endgroup
Finally, subtracting the first from $\frac{1}{4}$ times the second at $m_2\mapsto m_2-1$ gives
\begin{equation}
\label{eq:a211recur2}
\begin{aligned}
	&m_2 (\mu_1-\mu_4+m_2)(\mu_2-\mu_4+m_2)(\mu_3-\mu_4+m_2) a_{m_1,m_2}+a_{m_1,m_2-2} = \\
	& \bigr(2 \mu_1^2+\mu_1+\mu_2 \paren{5 \mu_1+2 \mu_2+1}+\mu_3 \paren{-2 \mu_3-3 \mu_4+2}+\paren{m_1+1} \paren{2 \mu_1+2 \mu_2+m_1} \\
	& \qquad +2 \paren{m_2-1} \paren{m_2-2 \mu_4}+1\bigl) a_{m_1,m_2-1}.
\end{aligned}
\end{equation}

Define
\begingroup
\allowdisplaybreaks
\begin{gather}
\label{eq:a211def}
\begin{aligned}
	a^*_{211,m_1,m_2}(\mu) :=& \frac{(-1)^{m_1}}{m_1! \, m_2! \, \Poch{1+\mu_1-\mu_4}{m_1} \Poch{1+\mu_2-\mu_4}{m_1} \Poch{1+\mu_3-\mu_4}{m_2}} \\*
	& \qquad \times \pFq32{ -m_1,1+2\mu_1+2\mu_2+m_1,\mu_4-\mu_3-m_2}{1+\mu_1-\mu_3, 1+\mu_2-\mu_3}{1},
\end{aligned}\\
\label{eq:b211def}
\begin{aligned}
	b^*_{211,m_1,m_2}(\mu) :=& \frac{(-1)^{m_1}}{m_1! \, m_2! \, \Poch{1+\mu_1-\mu_3}{m_1} \Poch{1+\mu_2-\mu_3}{m_1} \Poch{1+\mu_1-\mu_4}{m_2}} \\*
	& \qquad \times \pFq32{ \mu_3-\mu_1-m_1,1+\mu_2-\mu_4+m_1,-m_2}{1+\mu_2-\mu_4, 1+\mu_3-\mu_4}{1},
\end{aligned}\\
\label{eq:c211def}
\begin{aligned}
	c^*_{211,m_1,m_2}(\mu) :=& \frac{(-1)^{m_1}}{m_1! \, m_2! \, \Poch{1+\mu_1-\mu_3}{m_1} \Poch{1+\mu_2-\mu_3}{m_1} \Poch{1+\mu_3-\mu_4}{m_2}} \\*
	& \qquad \times \pFq32{-m_1,1+2\mu_1+2\mu_2+m_1,-m_2}{1+\mu_1-\mu_4, 1+\mu_2-\mu_4}{1}.
\end{aligned}
\end{gather}
\endgroup
Clearly $a^*_{211,0,0}(\mu) = b^*_{211,0,0}(\mu) = c^*_{211,0,0}(\mu) = 1$ by \eqref{eq:pFqFirstZero}, but we claim that
\begin{align}
\label{eq:F211Symmetries}
	a^*_{211,m_1,m_2}(\mu) =& b^*_{211,m_1,m_2}(\mu) = c^*_{211,m_1,m_2}(\mu),
\end{align}
for $m \in \N_0^2$, and these solve \eqref{eq:a211recur1} and \eqref{eq:a211recur2}.
Hence we have \eqref{eq:Jw211}, using
\[	a_{211,m_1,m_2}(\mu) = (8\pi^3 i)^{m_1} (4\pi^2)^{m_2} a^*_{211,m_1,m_2}(\mu)/\Lambda_{211}(\mu). \]
We note that form of \eqref{eq:a211def} was not derived by guessing or any method of solving recurrence relations, but rather by taking the coefficients from the Mellin-Barnes integral, as in \cref{sect:BesselMBProof}.
As is customary, we define $a^*_{211,m_1,m_2}(\mu) = 0$ whenever $m_1 < 0$ or $m_2 < 0$, but we can avoid checking base cases by applying $m_j! = \Gamma(m_j+1)$ (the Pochhammer symbols may similarly be expressed in terms of gamma functions), which naturally enforces this condition (treating $a^*_{211,m_1,m_2}(\mu)$ as a meromorphic function of $m_1,m_2$).

We now show $a^*_{211,m_1,m_2}(\mu)$ solves \eqref{eq:a211recur1} and \eqref{eq:a211recur2}.
One recurrence relation for $\pFqName32$ at $1$ is \cite[eq. (4.3)]{Bailey02}
\begin{equation}
\label{eq:3F2recur1}
\begin{aligned}
	0 =& (a_3-b_1+1)(a_3-b_2+1) \pFq32{a_1,a_2,a_3}{b_1,b_2}{1} \\
	& -(b_1b_2+(a_3+1)(3a_3-2b_1-2b_2+4)-(a_3-a_2+1)(a_3-a_1+1))\pFq32{a_1,a_2,a_3+1}{b_1,b_2}{1} \\
	& +(a_3+1)(a_3+a_2+a_1-b_1-b_2+2)\pFq32{a_1,a_2,a_3+2}{b_1,b_2}{1},
\end{aligned}
\end{equation}
and taking
\begin{equation}
\label{eq:a221to3F2subs}
\begin{aligned}
	a_1 =& -m_1, & a_2 =& 1+2\mu_1+2\mu_2+m_1, & a_3 =& \mu_4-\mu_3-m_2-1, \\
	b_1 =& 1+\mu_1-\mu_3, & b_2 =& 1+\mu_2-\mu_3,
\end{aligned}
\end{equation}
this implies \eqref{eq:a211recur2}.

We also have \cite[Sect. 48, eq 14]{Rainville}
\begin{equation}
\label{eq:3F2recur2}
\begin{aligned}
	0 =& a_1\, \pFq32{a_1+1,a_2,a_3}{b_1,b_2}{1}-a_2\, \pFq32{a_1,a_2+1,a_3}{b_1,b_2}{1}+(a_2-a_1) \pFq32{a_1,a_2,a_3}{b_1,b_2}{1}.
\end{aligned}
\end{equation}
Showing \eqref{eq:a211recur1} now becomes essentially a linear algebra problem.
Let $C_1$ be the right-hand side of \eqref{eq:3F2recur2} after substituting $a_2 \mapsto a_2-1$, let $C_2$ be the result of swapping $a_1$ and $a_3$ in $C_1$, let $C_3$ be the result of substituting $a_3 \mapsto a_3+1$ in $C_2$, and finally, let $C_4$ be the result of substituting $a_2 \mapsto a_2-1$ in the right-hand side of \eqref{eq:3F2recur1}.
Then
\[ A_1 C_1 + A_2 (A_3 C_2 - A_4 (A_5 C_3 - C_4)) = 0, \]
where
\begin{align*}
	A_1 =& (b_1-a_2) (b_2-a_2), \\
	A_2 =& \frac{1 + a_1 - a_2}{a_2-1}, \\
	A_3 =& (1+a_3) (4 + 3 a_3 - 2 b_1 - 2 b_2) - (2 - a_2 + a_3) (2 + a_2 + 2 a_3 - b_1 - b_2) + b_1 b_2, \\
	A_4 =& a_3, \\
	A_5 =& 1 + a_1 + a_2 + a_3 - b_1 - b_2.
\end{align*}
Writing this out gives the relation
\begin{align*}
	0=&-((a_2-1) A_1 + (1 + a_1 - a_2) A_3) \pFq32{a_1,a_2,a_3}{b_1,b_2}{1} + (1 + a_1 - a_2) a_3 A_5 \, \pFq32{a_1,a_2,a_3+1}{b_1,b_2}{1} \\
	& \qquad - a_1 A_1 \, \pFq32{a_1+1,a_2-1,a_3}{b_1,b_2}{1},
\end{align*}
and taking the values \eqref{eq:a221to3F2subs}, this implies \eqref{eq:a211recur1}, after some algebra.

That $b^*_{211,m_1,m_2}(\mu)$ and $c^*_{211,m_1,m_2}(\mu)$ solve \eqref{eq:a211recur1} and \eqref{eq:a211recur2} is seen by replacing \eqref{eq:a221to3F2subs} with the appropriate substitutions in \eqref{eq:3F2recur1} and \eqref{eq:3F2recur2}.

\subsection{Differential equations and power series for $w_{1111}$}
\label{sect:w1111DEPS}
In the case of the long Weyl element, we have $W_{1111}=\set{I}$, $Y_{1111}=Y$, the symmetries
\[ f(xy\trans{u}) = \e{x_1+x_3+x_6+u_1+u_3+u_6} f(y_1,y_2,y_3), \]
and $J_{1111}(y,\mu)$ is annihilated by the operators
\begingroup
\allowdisplaybreaks
\begin{gather}
\begin{aligned}
\label{eq:1111DEs1}
	& \left.(\Delta_2-\lambda_2)f\right|_{x=u=I} \\
	& = \bigl(4 \pi ^2 (y_1+y_2+y_3)-\lambda _2-y_1^2 \partial _{y_1}^2+y_1 y_2 \partial _{y_1} \partial _{y_2}-y_2^2 \partial _{y_2}^2 +y_2 y_3 \partial _{y_2} \partial _{y_3}-y_3^2 \partial _{y_3}^2\bigr)f|_{x=u=I}
\end{aligned}\\
\begin{aligned}
\label{eq:1111DEs2}
	& \left.(\Delta_3-\lambda_3)f\right|_{x=u=I} \\
	& = \bigl(-\lambda _3+y_1^2 y_2 \partial _{y_1}^2 \partial _{y_2}-2 y_1^2 \partial _{y_1}^2-y_1 y_2^2 \partial _{y_1} \partial _{y_2}^2+4 \pi ^2 y_2 (y_3-y_1) \partial _{y_2} +4 \pi ^2 y_1 y_2 \partial _{y_1} \\*
	& \qquad +y_1 y_2 \partial _{y_1} \partial _{y_2}+8 \pi ^2 (y_1-y_3)+y_2^2 y_3 \partial _{y_2}^2 \partial _{y_3}-y_2 y_3^2 \partial _{y_2} \partial _{y_3}^2-4 \pi ^2 y_2 y_3 \partial _{y_3} \\*
	& \qquad -y_2 y_3 \partial _{y_2} \partial _{y_3}+2 y_3^2 \partial _{y_3}^2\bigr)f|_{x=u=I}
\end{aligned}\\
\begin{aligned}
\label{eq:1111DEs3}
	& \left.(\Delta_4-\lambda_4)f\right|_{x=u=I} = \\
	& \Bigl(8 \pi^4 \paren{y_1^2+y_2^2+y_3^2+2y_1 y_2+2y_2 y_3}-\pi ^2 \paren{13 y_1+y_2 +13 y_3} -\lambda _4 \\*
	& \qquad +\tfrac{1}{2} y_1^4 \partial _{y_1}^4-y_1^3 y_2 \partial _{y_1}^3 \partial _{y_2}+2 y_1^3 \partial _{y_1}^3+\tfrac{3}{2} y_1^2 y_2^2 \partial _{y_1}^2 \partial _{y_2}^2 -y_1 y_2^3 \partial _{y_1} \partial _{y_2}^3-3 y_1^2 y_2 \partial _{y_1}^2 \partial _{y_2} \\*
	& \qquad -\tfrac{1}{4} y_1^2 \left(16 \pi ^2 (y_1+y_2)-17\right) \partial _{y_1}^2-\tfrac{1}{4} y_2^2 \left(16 \pi ^2 (y_1+y_2+y_3)-5\right) \partial _{y_2}^2 \\*
	& \qquad +4 \pi ^2 y_1 y_2 y_3 \partial _{y_1} \partial _{y_3} -4 \pi ^2 y_2 (-2 y_1+y_2-2 y_3) \partial _{y_2} -4 \pi ^2 y_1 (y_1+y_2) \partial _{y_1} \\*
	& \qquad +\tfrac{1}{4} y_1 y_2 \left(16 \pi ^2 (y_1+y_2)-5\right) \partial _{y_1} \partial _{y_2}+\tfrac{1}{2} y_2^4 \partial _{y_2}^4-y_2^3 y_3 \partial _{y_2}^3 \partial _{y_3}+2 y_2^3 \partial _{y_2}^3+\tfrac{3}{2} y_2^2 y_3^2 \partial _{y_2}^2 \partial _{y_3}^2 \\*
	& \qquad -y_2 y_3^3 \partial _{y_2} \partial _{y_3}^3-\tfrac{1}{4} y_3^2 \left(16 \pi ^2 (y_2+y_3)-17\right) \partial _{y_3}^2-3 y_2 y_3^2 \partial _{y_2} \partial _{y_3}^2-4 \pi ^2 y_3 (y_2+y_3) \partial _{y_3} \\*
	& \qquad +\tfrac{1}{4} y_2 y_3 \left(16 \pi ^2 (y_2+y_3)-5\right) \partial _{y_2} \partial _{y_3}+\frac{1}{2} y_3^4 \partial _{y_3}^4+2 y_3^3 \partial _{y_3}^3\Bigr)f|_{x=u=I}.
\end{aligned}
\end{gather}
\endgroup

In the notation of \cite{GLnI} (i.e., here $G$ and $R$ are not those of \cref{sect:GammaFunctions}), Hashizume's recurrence relation for the spherical Whittaker function \cite{Hashi} implies the coefficients of the long Weyl element power series on $GL(n)$ are characterized by
\begin{align*}
	a_m(\mu) =& (4 \pi^2)^{m_1+\ldots+m_{n-1}} G_{n,m}(\mu), \\
	G_{n,0}(\mu) =& 1, \\
	R_{n,k}(\mu) G_{n,k}(\mu) =& \sum_{j=1}^{n-1} G_{n,k-e_j}(\mu),
\end{align*}
where $G_{n,m}(\mu)=0$ when any $m_i < 0$ and
\[ R_{n,k}(\mu) = \frac{1}{2} \sum_{j=1}^n \paren{(k_j-k_{j-1})^2+2 k_j (\mu_j-\mu_{j+1})}, \]
using $m_j = 0$ for $j \notin \set{1,\ldots, n-3}$ and $k_j=0$ for $j \notin \set{1,\ldots,n-1}$.
Note: This corrects the formula for $R_{n,k}(\mu)$ given in \cite[Sect. 12.2.1]{GLnI}.

In particular, for $n=4$, we have the recurrence relation
\begin{align}
\label{eq:wlRecurRel}
	G_{4,m}(\mu) =& \frac{G_{4,m_1-1,m_2,m_3}(\mu)+G_{4,m_1,m_2-1,m_3}(\mu)+G_{4,m_1,m_2,m_3-1}(\mu)}{m_1^2+m_2^2+m_3^2-m_1 m_2-m_2 m_3+m_1(\mu_1-\mu_2)+m_2(\mu_2-\mu_3)+m_3(\mu_3-\mu_4)}.
\end{align}
Stade \cite{Stade04} solved this explicitly, and Ishi \cite[remark, p489]{Ishii} noted that $c^*_{1111,m}(\mu)$, below, is identical to Stade's answer by the identity \cite[Sect. 7.2, eq. (1)]{Bailey} for a Saalsch\"utzian $\pFqName43(1)$.

We will momentarily show that any of the following six expressions satisfy the above recurrence relation.
In particular, the expressions agree for all $m \in \N_0^3$.
\begingroup
\allowdisplaybreaks
\begin{align*}
	a^*_{1111,m}(\mu) =& \frac{\Poch{1 + 2 (\mu_1 + \mu_2)}{m_2 + m_3} \Poch{1 + \mu_1 - \mu_3}{ m_1 + m_2}}{m_1! \, m_2! \, m_3! \Poch{1 + \mu_1 - \mu_2}{m_1} \Poch{1 + 2 (\mu_1 + \mu_2)}{m_2} \Poch{1 + \mu_1 - \mu_3}{m_1} \Poch{1 + \mu_1 - \mu_3}{m_2}} \\*
	& \qquad \times \frac{1}{\Poch{1 + \mu_2 - \mu_3}{m_2} \Poch{1 + \mu_1 - \mu_4}{m_3} \Poch{1 + \mu_2 - \mu_4}{m_3}} \\*
	& \qquad \times \pFq43{-m_1 - \mu_1 + \mu_3, -m_2 - \mu_2 + \mu_3, -m_2 - \mu_1 + \mu_3, -m_3}{1 + \mu_3 - \mu_4, -m_1 - m_2 - \mu_1 + \mu_3, -m_2 - m_3 - 2 (\mu_1 + \mu_2)}{1}, \\
	b^*_{1111,m}(\mu) =& \frac{\Poch{1 + \mu_1 - \mu_4}{m_2 + m_3} \Poch{1 + \mu_1 - \mu_3}{m_1 + m_2}}{m_1! \, m_2! \, m_3! \Poch{1 + \mu_1 - \mu_2}{m_1} \Poch{1 + \mu_1 - \mu_3}{m_1} \Poch{1 + \mu_1 - \mu_3}{m_2} \Poch{1 + \mu_2 - \mu_3}{m_2}} \\*
	& \qquad \times \frac{1}{\Poch{1 + \mu_1 - \mu_4}{m_2} \Poch{1 + \mu_1 - \mu_4}{m_3} \Poch{1 + \mu_3 - \mu_4}{m_3}} \\*
	& \qquad \times \pFq43{-m_1 - \mu_1 + \mu_2, -m_2, -m_2 - \mu_1 + \mu_3, -m_3}{1 + \mu_2 - \mu_4, -m_1 - m_2 - \mu_1 + \mu_3, -m_2 - m_3 - \mu_1 + \mu_4}{1}, \\
	c^*_{1111,m}(\mu) =& \frac{\Poch{1 + \mu_1 - \mu_4}{m_2 + m_3} \Poch{1 + \mu_1 - \mu_4}{m_1 + m_2}}{m_1! \, m_2! \, m_3! \Poch{1 + \mu_1 - \mu_2}{m_1} \Poch{1 + \mu_1 - \mu_3}{m_2} \Poch{1 + \mu_1 - \mu_4}{m_1} \Poch{1 + \mu_1 - \mu_4}{m_2}} \\*
	& \qquad \times \frac{1}{\Poch{1 + \mu_1 - \mu_4}{m_3} \Poch{1 + \mu_2 - \mu_4}{m_2} \Poch{1 + \mu_3 - \mu_4}{m_3}} \\*
	& \qquad \times \pFq43{-m_1 - \mu_1 + \mu_2, -m_2, -m_2 - \mu_1 + \mu_4, -m_3 - \mu_3 + \mu_4}{1 + \mu_2 - \mu_3, -m_1 - m_2 - \mu_1 + \mu_4, -m_2 - m_3 - \mu_1 + \mu_4}{1}, \\
	d^*_{1111,m}(\mu) =& \frac{\Poch{1 + \mu_2 - \mu_4}{m_2 + m_3} \Poch{1 + \mu_1 - \mu_3}{m_1 + m_2}}{m_1! \, m_2! \, m_3! \Poch{1 + \mu_1 - \mu_2}{m_1} \Poch{1 + \mu_1 - \mu_3}{m_1} \Poch{1 + \mu_1 - \mu_3}{m_2} \Poch{1 + \mu_2 - \mu_3}{m_2}} \\*
	& \qquad \times \frac{1}{\Poch{1 + \mu_2 - \mu_4}{m_2} \Poch{1 + \mu_2 - \mu_4}{m_3} \Poch{1 + \mu_3 - \mu_4}{m_3}} \\*
	& \qquad \times \pFq43{-m_1, -m_2, -m_2 - \mu_2 + \mu_3, -m_3}{1 +\mu_1 - \mu_4, -m_1 - m_2 - \mu_1 + \mu_3, -m_2 - m_3 - \mu_2 + \mu_4}{1}, \\
	e^*_{1111,m}(\mu) =& \frac{\Poch{1 + \mu_2 - \mu_4}{m_2 + m_3} \Poch{1 + \mu_1 - \mu_4}{m_1 + m_2}}{m_1! \, m_2! \, m_3! \Poch{1 + \mu_1 - \mu_2}{m_1} \Poch{1 + \mu_2 - \mu_3}{m_2} \Poch{1 + \mu_1 - \mu_4}{m_1} \Poch{1 + \mu_1 - \mu_4}{m_2}} \\*
	& \qquad \times \frac{1}{\Poch{1 + \mu_2 - \mu_4}{m_2} \Poch{1 + \mu_2 - \mu_4}{m_3} \Poch{1 + \mu_3 - \mu_4}{m_3}} \\*
	& \qquad \times \pFq43{-m_1, -m_2, -m_2 - \mu_2 + \mu_4, -m_3 - \mu_3 +\mu_4}{1 + \mu_1 - \mu_3, -m_1 - m_2 - \mu_1 + \mu_4, -m_2 - m_3 - \mu_2 + \mu_4}{1}, \\
	f^*_{1111,m}(\mu) =& \frac{\Poch{1 + \mu_2 - \mu_4}{m_2 + m_3} \Poch{1 + 2 (\mu_1 + \mu_2)}{m_1 + m_2}}{m_1! \, m_2! \, m_3! \Poch{1 + 2 (\mu_1 + \mu_2)}{m_2} \Poch{1 + \mu_1 - \mu_3}{m_1} \Poch{1 + \mu_2 - \mu_3}{m_2} \Poch{1 + \mu_1 - \mu_4}{m_1}} \\*
	& \qquad \times \frac{1}{\Poch{1 + \mu_2 - \mu_4}{m_2} \Poch{1 + \mu_2 - \mu_4}{m_3} \Poch{1 + \mu_3 - \mu_4}{m_3}} \\*
	& \qquad \times \pFq43{-m_1, -m_2 - \mu_2 + \mu_3, -m_2 - \mu_2 + \mu_4, -m_3 - \mu_2 + \mu_4}{1 + \mu_1 - \mu_2, -m_1 - m_2 - 2 (\mu_1 + \mu_2), -m_2 - m_3 - \mu_2 + \mu_4}{1}.
\end{align*}
\endgroup
As usual, we define $z_{1111,m}(\mu) = (4 \pi^2)^{m_1+m_2+m_3} z^*_{1111,m}(\mu)/\Lambda_{w_l}(\mu)$ for $z=a,b,c,d,e,f$.
Hence we have \eqref{eq:Jw1111}.

The recurrence relation \eqref{eq:wlRecurRel} for $z^*_{1111,m}(\mu)$ at $z=a,b,d,e,f$ can be reduced (through substitutions) to
\begin{equation}
\label{eq:wlRecurRel2}
\begin{aligned}
	0 =& b_1 b_2 (a_1 a_3 + a_1 a_4 - a_1 b_1 + a_2 a_3 + a_2 a_4 - a_2 b_2 - a_3 a_4) \pFq43{a_1, a_2, a_3, a_4}{b_1, b_2, b_3}{1} \\
	&+ a_3 a_4 (a_1 - b_1) (a_2 - b_2) \pFq43{a_1, a_2, 1 + a_3, 1 + a_4}{1 + b_1, 1 + b_2, b_3}{1} \\
	&- a_2 b_1 (a_3 - b_2) (b_2 - a_4) \pFq43{a_1, 1 + a_2, a_3, a_4}{b_1,1 + b_2, b_3}{1} \\
	&- a_1 b_2 (a_3 - b_1) (b_1 - a_4) \pFq43{1 + a_1, a_2, a_3, a_4}{1 + b_1, b_2, b_3}{1},
\end{aligned}
\end{equation}
with $a_1 \in -\N_0$, while those at $z=b,c,d,e$ can be reduced to \eqref{eq:wlRecurRel2} with $a_4 \in -\N_0$, where in both cases, we take $1+a_1+a_2+a_3+a_4=b_1+b_2+b_3$, i.e. the hypergeometric functions are ``Saalsch\"utzian'' or ``balanced''.
The arguments of Stade and Ishii (as mentioned above) imply \eqref{eq:wlRecurRel2} with $a_4 \in -\N_0$ upon noting that only $k_2$ is treated as an integer in Stade's proof and we may absorb a fixed value of $k_2$ into, say, $a_1\mapsto a_1-k_2$.
Unfortunately, the overlap at $z=b,d,e$ does not imply \eqref{eq:wlRecurRel2} with $a_1 \in -\N_0$ except when $a_4 \in -\N_0$ also, which is not sufficient for our purposes.
Furthermore, the relevant formulas are not readily available in the literature, so we will prove \eqref{eq:wlRecurRel2} directly.

Such formulas follow in general from dimension arguments and the effect of the operators $\frac{d}{dz}$ and $z\frac{d}{dz}$ (see \cite{Bailey02,Rainville}), which we apply in two stages.
First, for general (not necessarily Saalsch\"utzian, not necessarily terminating and not necessarily at $z=1$) $\pFqName43$, setting
\[ f(z) = \pFq43{a_1, a_2, a_3, a_4}{b_1+1, b_2+1, b_3}{z}, \qquad T_a := 1+\frac{1}{a} z\frac{d}{dz}, \]
we see that
\[ T_{a_1}(f) = \pFq43{a_1+1, a_2, a_3, a_4}{b_1+1, b_2+1, b_3}{z}, \qquad T_{b_1}(f) = \pFq43{a_1, a_2, a_3, a_4}{b_1, b_2+1, b_3}{z}, \]
etc., since
\[ \Poch{a+1}{k}=\paren{1+\frac{1}{a}k}\Poch{a}{k}. \]
Then
\[ T_{b_2} T_{b_1} f, \qquad T_{a_3} T_{a_4} f, \qquad T_{a_2} T_{b_1} f, \qquad T_{a_1} T_{b_2} f \]
all lie in the span of $f,f',f''$ (with coefficients in $\R(a,b,z)$), so there must be a linear dependence relation among them.
That linear dependence relation is
\begin{equation}
\label{eq:4F3genrel}
\begin{aligned}
	0 =& b_1 b_2 (a_3 a_4 (b_2-a_2) + b_1 (a_3 (a_2 - a_4) + a_2(a_4 - b_2)) \\
 	& \qquad + 
 a_1 (a_3 a_4 + (a_2 - a_3 - a_4) b_2 + b_1 (-a_2 + b_2))) \pFq43{a_1, a_2, a_3, a_4}{b_1, b_2, b_3}{z} \\
	&+ a_3 a_4 (a_1 - b_1) (a_2 - b_2)(b_1-b_2) \pFq43{a_1, a_2, 1 + a_3, 1 + a_4}{1 + b_1, 1 + b_2, b_3}{z} \\
	&- a_2 b_1 (b_1-a_1)(a_3 - b_2) (b_2 - a_4) \pFq43{a_1, 1 + a_2, a_3, a_4}{b_1,1 + b_2, b_3}{z} \\
	&- a_1 b_2 (a_3 - b_1) (b_1 - a_4) (a_2-b_2) \pFq43{1 + a_1, a_2, a_3, a_4}{1 + b_1, b_2, b_3}{z}.
\end{aligned}
\end{equation}
Now multiplying \eqref{eq:wlRecurRel2} by $(b_1-b_2)$ and subtracting \eqref{eq:4F3genrel} at $z=1$ gives
\begin{equation}
\label{eq:wlRecurRel3}
\begin{aligned}
	0 =& b_1 b_2 (a_1 (a_3 - b_1) (a_4 - b_1) - a_2 (a_3 - b_2) (a_4 - b_2) - a_1 a_2 (b_1 - b_2)) \pFq43{a_1, a_2, a_3, a_4}{b_1, b_2, b_3}{1} \\
	&+ a_2 b_1 (a_1-b_2) (a_3 - b_2) (a_4-b_2) \pFq43{a_1, 1 + a_2, a_3, a_4}{b_1,1 + b_2, b_3}{1} \\
	&+ a_1 b_2 (a_2-b_1) (a_3 - b_1) (a_4-b_1) \pFq43{1 + a_1, a_2, a_3, a_4}{1 + b_1, b_2, b_3}{1}.
\end{aligned}
\end{equation}

Finally, we use the contiguous relations for Saalsch\"utzian $\pFqName43(1)$, and these are found in \cite[Ch. IV]{WilsonThesis}.
Unfortunately, \cite[eq. (4.22)]{WilsonThesis} has a minor typographical error in that the factor $(f-b)$ should be $(f-a)$; correcting this gives \eqref{eq:wlRecurRel3}.
The correctness of the formula may be checked by noting that the three (terminating, Saalsch\"utzian) $\pFqName43(1)$ belong to a two-dimensional vector space and the vectors of values at $a_4=-1,-2,-3$ are generically linearly independent.
(So it suffices to check the formula at $a_4=-1,-2,-3$, which is simple to do.)

\section{Integral Representations}
\label{sect:IntRepns}
We consider the (conditionally convergent) integral
\begin{align}
\label{eq:IwDef}
	\mathcal{I}_w = \mathcal{I}_w(y,\mu,\delta) := \int_{\wbar{U}_w(\R)} I^W_{\mu,\delta}(ywu) \wbar{\psi_I(u)} du,
\end{align}
as in \cite[Sect. 5.5.1]{GLnI}.
We intend to show that for each relevant $w$, $K_w(y,\mu,\delta) = \mathcal{I}_w(y,\mu,\delta)$; in particular, the function $C(\mu,\delta)$ of \cite[Sect. 5.5.1]{GLnI} is simply one.
The general method is to compute a Mellin-Barnes integral for $\mathcal{I}_w$ and show first that it lies in the span of the Frobenius series and second that it has the correct asymptotics as $y \to 0$.
We prove this for $\Lambda=0$ and leave the higher-weight cases to the Analytic Continuation Conjecture.
(Which we have proven, except in the case of $w=w_{22}$.)

Once the equality $K_w=\mathcal{I}_w$ is established via the Mellin-Barnes integrals, we also consider Stade-type multiple-Bessel integrals, the inverse Mellin transform and the simultaneous Fourier-inverse Mellin transform.
We are somewhat careless with the details of convergence, but the precise, correct path would be
\begin{enumerate}
\item For some nice (e.g. Schwartz-class and holomorphic on a tube domain containing $\Re(\mu)=0$) test function $f$, define
\begin{align}
\label{eq:InvIStar}
	F(g) = F_{\mu,\delta}(g) := \int_{\Re(\mu)=0} f(\mu) I^W_{\mu,\delta}(g) \dspecmu,
\end{align}
and consider instead
\begin{align*}
	\mathcal{H}_w = \mathcal{H}_w(y,\mu,\delta) := \int_{\wbar{U}_w(\R)} F_{\mu,\delta}(ywu) \wbar{\psi_I(u)} du.
\end{align*}

\item Argue that (the $u$ integral in) $\mathcal{H}_w$ converges nicely enough and apply the Mellin expansion \eqref{eq:PsiThetaInvMellin} to each term of the complex exponentials, as well as any desired substitutions.
By ``converges nicely enough'', we mean that the oscillatory integral converges uniformly in $\mu$ at $\theta=\frac{\pi}{2}$, while taking $\theta < \frac{\pi}{2}$ improves convergence without affecting any applications of integration by parts; this is most easily seen from the inverse Mellin transform, i.e. the representation of the Bessel function as a Mellin transform.
More precisely, if we use a Schwartz-class $\alpha:\R^m \to \R$ with $\alpha(0)=1$ to define the smoothed Riemann integral
\[ \int_{\R^m}^\alpha \cdots dx := \lim_{R \to \infty} \int_{\R^m} \alpha\paren{\frac{x}{R}} \cdots dx, \]
then the $R$ limit converges uniformly in $\theta \in (0,\frac{\pi}{2}]$ and polynomially in $\norm{\mu}$.

\item Apply Fubini-Tonelli to the now absolutely-convergent $u$-$\mu$ integrals and compute the resulting Mellin transforms, giving beta functions, etc.

\item Take the limit in $\theta$ from \eqref{eq:PsiThetaInvMellin} by dominated convergence.

\item Argue that the Mellin-Barnes integral inside the $\mu$ integral is $K_w$ by its asymptotics and differential equations.
That is,
\begin{align}
\label{eq:HwKw}
	\mathcal{H}_w(y,\mu,\delta) = \int_{\Re(\mu)=0} f(\mu) K_w(y,\mu,\delta) \dspecmu.
\end{align}
Note: We are not applying an approximation to the identity here, but arguing directly that the resulting Mellin-Barnes integral is the Bessel function, which implies the above.
The reasons for the somewhat mysterious changes of coordinates can be deduced directly from the integral, but are best thought of in terms of a general computation as in \cite[Sect. 8]{GLnI}.
We postpone this step until \cref{sect:BesselMBProof}.

\item Similarly, the other integral representations are derived by manipulating $\mathcal{H}_w$, which we now know satisfies \eqref{eq:HwKw}.
Note: Here we are applying an approximation to the identity in \eqref{eq:HwKw} to argue that the resulting integral representations also give the Bessel function.
\end{enumerate}
Furthermore, as the details are somewhat repetitive, we restrict to simply stating the steps involved and the final integral representation.

To start with, if we write $wu=xy^*\trans{x'}$, we have the representation
\begin{align}
\label{eq:IwRepn1}
	\mathcal{I}_w = I_{\mu,\delta}(y) \int_{\wbar{U}_w(\R)} I_{\mu,\delta}(y^*) \psi_y(x) \wbar{\psi_I(u)} du,
\end{align}
but if we initially conjugate $(w^{-1} y w) u \mapsto u (w^{-1} y w)$, we can also write
\begin{align}
\label{eq:IwRepn2}
	\mathcal{I}_w = I_{-\mu,\delta}(y^\iota) \int_{\wbar{U}_w(\R)} I_{\mu,\delta}(y^*) \psi_I(x) \wbar{\psi_{y^\iota}(u)} du,
\end{align}
using the facts $w^{-1} y w = (y^\iota)^{-1}$ (for a relevant Weyl element $w$ and $y \in Y_w$, since the conjugation reverses the order of the blocks in \cite[eq. (8)]{GLnI}), $I_{0,0}(y^\iota)=I_{0,0}(y)$ (since $\what{\rho}_{n-i}=\what{\rho}_i$ in \cite[Sect. 6.2]{GLnI}), and the Jacobian for the change of variables $u \mapsto yuy^{-1}$ is $p_{\rho-\rho^w}(y)$ \cite[Sect. 6.1]{GLnI}.
We assume throughout that $\Re(\mu)=0$; in particular, we have
\[ I_{-\mu,\delta}(y^\iota) = \chi^{\delta_4}_{\frac{3}{2}-\mu_1-\mu_2-\mu_3}(y_1) \chi^{\delta_3+\delta_4}_{2-\mu_1-\mu_2}(y_2) \chi^{\delta_2+\delta_3+\delta_4}_{\frac{3}{2}-\mu_1}(y_3). \]

Note: In the construction of the Kuznetsov formula for the $\Lambda=0$ case, it might be useful to consider localizing by taking \eqref{eq:InvIStar} as the kernel of the Poincar\'e series and arguing by Mellin inversion that
\[ \sum_{v \in \pi_{0,\wtilde{\mu},\delta}} \innerprod{F, v} = f(\wtilde{\mu}) \]
on the spectral side.
This would avoid the Iwasawa decomposition in the Interchange of Integrals, but the author knows of no reason why the Langlands spectral expansion for such a Poincar\'e series should converge.

Throughout this section, we set
\begin{align*}
	\Delta=&\delta_1+\delta_2+\delta_3+\delta_4.
\end{align*}

\subsection{Integral representations for $w_{31}$.}
After substituting $(u_1,u_2,u_4) \mapsto (x_1, x_2, x_3)$, we have
\begin{align}
\label{eq:Iw31}
	\mathcal{I}_{31} =& (-1)^{\delta_1+\delta_2+\delta_3} I_{-\mu,\delta}(y^\iota) \int_{\R^3} \wtilde{I}_{\mu,\delta}\paren{x_1, x_2, x_3} \e{-y_3 x_1-\frac{x_2}{x_1}-\frac{x_3}{x_2}+\frac{1}{x_3}} dx.
\end{align}
It follows that the inverse Mellin transform (recall \eqref{eq:wcheckKdef}) is \eqref{eq:K31IM} and the Fourier-inverse Mellin transform is singular.

If instead we evaluated the $x_1$ and $x_3$ integrals in \eqref{eq:Iw31} using \eqref{eq:mcZInt}, we have the Stade-type multiple Bessel integral \eqref{eq:K31Stade}.

In \eqref{eq:mcZInt}, if one were to Mellin expand the complex exponentials using \eqref{eq:PsiThetaInvMellin} and evaluate the $x$ integral using \eqref{eq:BetaEval}, one would necessarily arrive at the same expression as Mellin expanding $\mathcal{Z}_\mu^\delta(ab)$ using \eqref{eq:ZMellin}.
So in developing the Mellin-Barnes integrals, we apply the latter approach to the Stade-type integral, as it involves less writing.
Mellin expanding the $GL(2)$ Bessel functions with \eqref{eq:ZMellin}, applying Mellin inversion in $u$ and substituting $s \mapsto s-\frac{\mu_1+\mu_2}{2}$, $\ell \mapsto \ell+\delta_3+\delta_4$ gives \eqref{eq:Kw31MB}.

\subsection{Integral representations for $w_{22}$.}
After substituting $(u_2,u_3,u_4, u_5) \mapsto (x_1, x_2, x_3, x_4)$, we have
\begin{align}
\label{eq:Iw22}
	\mathcal{I}_{22} =& (-1)^{\delta_1+\delta_3} I_{-\mu,\delta}(y^\iota) \int_{\R^4} \wtilde{I}_{\mu,\delta}\paren{x_1,x_2 x_3-x_1 x_4,x_4} \e{-y_2 x_2-\frac{x_3}{x_1}+\frac{x_2}{x_2 x_3-x_1 x_4}+\frac{x_3}{x_4}} dx.
\end{align}
Substituting $(x_1,x_2,x_2 x_3-x_1 x_4, x_4) \mapsto (z_1,x_2,z_2, z_3)$ gives the inverse Mellin transform as
\begin{align*}
	\wcheck{K}_{22}(y,z) =& (-1)^{\delta_1+\delta_3} \int_{\R} \e{-y_2 x_2+\frac{x_2}{z_2}+\paren{\frac{z_2+z_1 z_3}{x_2}}\paren{\frac{1}{z_3}-\frac{1}{z_1}}} \frac{dx_2}{\abs{x_2}},
\end{align*}
and \eqref{eq:K22IM} and \eqref{eq:K22FIM} (recall \eqref{eq:wtildeKdef}) follow.

If instead we substituted $x_1 \mapsto x_1 x_2 x_3/x_4$ and evaluated the $x_2$ and $x_3$ integrals in \eqref{eq:Iw22}, we have the Stade-type integral \eqref{eq:K22Stade}.

Mellin expanding the $GL(2)$ Bessel functions with \eqref{eq:ZMellin}, applying Mellin inversion in $u_2$, evaluating the $u_1$ integral using \eqref{eq:BetaEval} and substituting $s \mapsto s-\frac{\mu_2-\mu_4}{2}$, $\ell \mapsto \ell+\delta_4$ using \eqref{eq:GdShift} gives \eqref{eq:Kw22MB}.

\subsection{Integral representations for $w_{121}$.}
After substituting \\ $(u_1,u_2,u_4,u_5, u_6) \mapsto (x_1, x_2, x_3, x_4, x_5)$, we have
\begin{equation}
\label{eq:Iw121}
\begin{aligned}
	\mathcal{I}_{121} =& (-1)^{\delta_1} I_{-\mu,\delta}(y^\iota) \int_{\R^5} \wtilde{I}_{\mu,\delta}\paren{x_3-x_2 x_5-x_1 x_4,x_3-x_2x_5,x_3} \\
	& \qquad \e{-y_3 x_1-y_1 x_5-\frac{x_1}{x_3-x_2 x_5-x_1 x_4}-\frac{x_2 x_4}{x_3-x_2 x_5}+\frac{x_5}{x_3}} dx.
\end{aligned}
\end{equation}
Substituting $(x_1,x_2,x_3, x_4, x_5) \mapsto (\frac{z_2-z_1}{x_4}, \frac{z_3-z_2}{x_5}, z_3, x_4, x_5)$ gives the inverse Mellin transform as
\begin{align*}
	\wcheck{K}_{121}(y,z) =& (-1)^{\delta_1} \int_{\R^2} \e{-\frac{z_2-z_1}{x_4}\paren{y_3+\frac{1}{z_1}}-y_1 x_5-\frac{(z_3-z_2) x_4}{z_2 x_5}+\frac{x_5}{z_3}} \frac{dx_4 \, dx_5}{\abs{x_4 x_5}},
\end{align*}
and hence \eqref{eq:K121IM}.
After substituting $x_4 \mapsto (z_1-z_2)/x_4$, we have Fourier-inverse Mellin transform \eqref{eq:K121FIM}.

If instead we substituted $(x_1,x_3,x_4) \mapsto (-u_1 x_2, u_2 u_1 x_2 x_5, -x_5(u_2+u_3-1/u_1)x_5)$ and evaluated the $x_2$ and $x_5$ integrals in \eqref{eq:Iw121}, we have the partial Stade-type integral \eqref{eq:K121Stade}.
Note: The $u_1$ coordinate has been isolated from the $GL(2)$ Bessel functions, but the integral is a new hypergeometric function, which certainly may be expressed in terms of integrals of $GL(2)$ Bessel functions, but we don't pursue that avenue here.

Mellin expanding the $GL(2)$ Bessel functions with \eqref{eq:ZMellin} and the two terms of the complex exponential (separately) using \eqref{eq:PsiThetaInvMellin} (and skipping ahead to $\theta=\frac{\pi}{2}$), applying Mellin inversion in $u_2,u_3$, evaluating the $u_1$ integral using \eqref{eq:BetaEval} and substituting $(s_1,s_2) \mapsto (s_1-\frac{\mu_1-\mu_3}{2},s_2-\frac{\mu_1-\mu_4}{2})$, $\ell_2 \mapsto \ell_2+\delta_2+\delta_3-2(\ell_2\delta_2+\ell_2\delta_3+\delta_2\delta_3) \pmod{4}$ (taking care here since the Mellin transform of the complex exponential requires $\delta \in \set{0,1}$) gives \eqref{eq:Kw121MB}.

\subsection{Integral representations for $w_{211}$.}
After substituting \\ $(u_1, u_2,u_3,u_4, u_5) \mapsto (x_1, x_2, x_3, x_4, x_5)$, we have
\begin{equation}
\label{eq:Iw211}
\begin{aligned}
	\mathcal{I}_{211} =& (-1)^{\delta_1} I_{-\mu,\delta}(y^\iota) \int_{\R^5} \wtilde{I}_{\mu,\delta}\paren{x_2-x_1 x_3, x_3 x_4-x_2 x_5, x_4} \\
	& \qquad \e{-y_3 x_1-y_2 x_3-\frac{x_4-x_1 x_5}{x_2-x_1 x_3}-\frac{x_2}{x_3 x_4-x_2 x_5}+\frac{x_5}{x_4}} dx.
\end{aligned}
\end{equation}
Substituting $x_1 \mapsto (x_2-z_1)/x_3$ and then $x_2 \mapsto \frac{x_3 x_4-z_2}{x_5}$ and $x_4 \mapsto z_3$ gives the inverse Mellin transform as
\begin{align*}
	& \wcheck{K}_{211}(y,z) \\
	&= (-1)^{\delta_1} \int_{\R^2} \e{-y_3\paren{-\frac{z_1}{x_3}-\frac{z_2}{x_3 x_5}+\frac{z_3}{x_5}}-y_2 x_3-\frac{z_1 x_5+z_2}{z_1 x_3}-\frac{z_3 x_3-z_2}{z_2 x_5}+\frac{x_5}{z_3}} \frac{dx_3 \, x_5}{\abs{x_3 x_5}},
\end{align*}
from which follows \eqref{eq:K211IM}.
After substituting $(x_3,x_5) \mapsto \paren{-t_1,-\frac{z_2+z_3 t_1}{z_1+t_1 t_2}}$, we have the Fourier-inverse Mellin transform \eqref{eq:K211FIM}.

If instead we substituted $(x_2,x_4,x_5) \mapsto ((u_1(1-u_2)+u_2)x_1 x_3,(1-u_1)u_2 u_3 x_1,(1-u_1)u_3)$ and then evaluated the $x_1$ and $x_3$ integrals in \eqref{eq:Iw211}, we have the partial Stade-type integral \eqref{eq:K211Stade}.
Note: Again, the $u_1$ coordinate has been isolated from the Bessel functions, but the integral is a new hypergeometric function.

Mellin expanding the $GL(2)$ Bessel functions with \eqref{eq:ZMellin} and the two terms of the complex exponential (separately) using \eqref{eq:PsiThetaInvMellin}, applying Mellin inversion in $u_3$, evaluating the $u_1,u_2$ integrals using \eqref{eq:BetaEval} and substituting $(s_1,s_3) \mapsto (s_1-\frac{\mu_2-\mu_4}{2},s_3-\frac{\mu_1+\mu_4}{2})$, $(\ell_1,\ell_3) \mapsto (\ell_1+\delta_1+\delta_2+\delta_3,\ell_3+\delta_1+\delta_4)$ then for the spare integral substituting $s_2 \mapsto -s_2+s_1-\mu_4$, $\ell_2 \mapsto \ell_2+\ell_1+\delta_4$ gives \eqref{eq:Kw211MB}.
Note: The $s_2$ integral in \eqref{eq:Kw211MB} should be taken first as it converges conditionally, though it can be bent to achieve absolute convergence, as described above.

\subsection{Integral representations for $w_{1111}$.}
After renaming $u \mapsto x$, we have
\begin{align*}
	\mathcal{I}_{1111}(y,\mu,\delta) =& (-1)^{\delta_1+\delta_2} I_{-\mu,\delta}(y^\iota) \int_{\R^6} \wtilde{I}_{\mu,\delta}(x_4-x_2 x_6-x_1 x_5+x_1 x_3 x_6, x_3 x_4 - x_2 x_5, x_4) \\
	& \e{-y_3 x_1-y_2 x_3-y_1 x_6-\frac{x_2 - x_1 x_3}{x_4-x_2 x_6-x_1 x_5+x_1 x_3 x_6}+\frac{x_4-x_2 x_6}{x_3 x_4-x_2 x_5}+\frac{x_5}{x_4}} dx.
\end{align*}
Substituting
\[ (x_1,x_2,x_5) \mapsto \paren{\frac{x_3 x_5 z_1-z_2+x_3 (1-x_5) z_3}{x_3^2(1-x_5) x_5 x_6}, -\frac{z_2-x_3 z_3}{x_3 x_5 x_6}, x_3 x_5 x_6} \]
gives the inverse Mellin transform as
\begin{align*}
	\wcheck{K}_{1111}(y,z) =& (-1)^{\delta_1+\delta_2} \int_{\R^3} e\biggl(-y_3 \frac{x_3 x_5 z_1-z_2+x_3 (1-x_5) z_3}{x_3^2(1-x_5) x_5 x_6}-y_2 x_3-y_1 x_6 \\
	& \qquad+\frac{x_3 z_1 - z_2}{x_3 x_6 (1-x_5) z_1}+\frac{z_2-x_3(1-x_5)z_3}{x_3 x_5 z_2}+\frac{x_3 x_5 x_6}{z_3}\biggr) \frac{dx_3 \, dx_5 \, dx_6}{\abs{x_3 x_5 x_6 (1-x_5)}},
\end{align*}
from which we get \eqref{eq:K1111IM}.
If we instead substitute $x_2 \mapsto -\frac{x_2-x_3 x_4}{x_5}$ and then
\[ x_5 \mapsto \frac{x_4-x_5+x_1 x_3 x_6\pm\sqrt{(x_1 x_3 x_6+x_4-x_5)^2+4x_1 x_6(x_2-x_3 x_4)}}{2x_1}, \]
the Fourier-inverse Mellin transform becomes \eqref{eq:K1111FIM}.

For the Stade-type and Mellin-Barnes integrals, we substitute
\[ (x_3,x_5,x_6) \mapsto \paren{x_3+\frac{x_2}{x_1},x_5+\frac{x_4}{x_1}+\frac{x_3 x_4}{x_2},x_6+\frac{x_4}{x_2}+\frac{x_5}{x_3}}, \]
then $(x_2,x_5) \mapsto (x_2/y_2,x_5/y_1)$ so that
\begin{equation}
\label{ex:LastIw1111}
\begin{aligned}
	& \mathcal{I}_{1111}(y,\mu,\delta) = \\
	& (-1)^{\delta_1+\delta_3} \chi_{\frac{3}{2}+\mu_2+\mu_4-\mu_3}^{\delta_2+\delta_3+\delta_4}(y_1) \chi_{2+\mu_2+\mu_4}^{\delta_2+\delta_4}(y_2) \chi_{\frac{3}{2}+\mu_2+\mu_3+\mu_4}^{\delta_2+\delta_3+\delta_4}(y_3) \int_{\R^6} \chi_{-1+\mu_2-\mu_1}^{\delta_1+\delta_2}(x_1) \chi_{-1+\mu_3-\mu_2}^{\delta_2+\delta_3}(x_2) \\
	& \chi_{-1+\mu_2-\mu_1}^{\delta_1+\delta_2}(x_3) \chi_{-1+\mu_4-\mu_3}^{\delta_3+\delta_4}(x_4) \chi_{-1+\mu_3-\mu_2}^{\delta_2+\delta_3}(x_5) \\
	& e\biggl(\frac{1-x_2}{x_1}+\frac{1-x_5}{x_3}+\frac{1}{x_6}+\frac{x_5}{x_4 y_1}-y_3 x_1-y_1 y_2\frac{x_4}{x_2}+y_2\frac{(1-x_2)x_3}{x_2}+y_1\frac{(1-x_5) x_6}{x_5}\biggr) dx.
\end{aligned}
\end{equation}
Then the $x_1$, $x_3$, $x_4$ and $x_6$ integrals can be collapsed to $\mathcal{Z}$ functions resulting in the Stade-type integral \eqref{eq:K1111Stade}.
In the case $\delta=0$ and $\sgn(y)=(+,+,+)$, we expect that this should reduce to Stade's actual integral representation for the $GL(4)$ Whittaker function \cite[Thm 2.1]{Stade03} using the conclusion of \cite[Sect. 12.1]{GLnI}, but this is not trivial.

In \eqref{ex:LastIw1111}, we can simply Mellin expand all eight terms of the complex exponential and evaluate the six $x$ integrals.
As usual, we substitute to clean up the exponents on $y$, giving \eqref{eq:Kw1111MB}.

\section{Equating the Mellin-Barnes integrals to the Bessel functions.}
\label{sect:BesselMBProof}

In this section, we prove the Mellin-Barnes integral representations of the Bessel functions given in the previous section.
That is, we prove the identities \eqref{eq:Kw31MB},\eqref{eq:Kw22MB},\eqref{eq:Kw121MB},\eqref{eq:Kw211MB} by comparing the power series expansions with \eqref{eq:KwToJw} and \eqref{eq:CwToR} using \eqref{eq:Jw31},\eqref{eq:Jw22},\eqref{eq:Jw121} and \eqref{eq:Jw211} from the solutions of the differential equations.

\subsection{Hypergeometric Integrals}
We consider an integral
\begin{align*}
	\mathcal{F} = \mathcal{F}(t,\alpha,u,\beta,z,x) :=& \frac{1}{2} \sum_{\ell\in\set{0,1}} \int_{-i\infty}^{i\infty} \frac{G_\ell(s,t,\alpha)}{G_\ell(s,u,\beta)} \chi_{x-s}^\ell(z) f(\ell,s) \frac{ds}{2\pi i},
\end{align*}
where $t\in\C^p$, $\alpha \in \Z^p$, $u \in\C^q$, $\beta \in \Z^q$, the contour is taken so the arguments of all of the $G_\delta$ functions pass to the right of the poles and $f(\ell,s)$ has no poles to the left of $\Re(s)=0$ and is invariant under $\ell \mapsto \ell+2$.
Ignoring the contribution of $f(\ell,s)$ to the convergence of the relevant integrals, if $p>q$ or $p=q$ and $\abs{z} < 1$, we may shift the $u$ contour to $-\infty$; if $p<q$ or $p=q$ and $\abs{z} > 1$ and we may shift $u$ to $+\infty$ to obtain a hypergeometric series.
Similarly, in case $p=q$ and $\abs{z}=1$, we may apply either method by taking the limit as $z \to 1^\pm$, provided the resulting series converges absolutely there.
The convergence of the integrals involved in the contour shifting follows from \eqref{eq:Stirling}.

We assume the first case -- that is, $p>q$ or $p=q$ and $z \in (0,1)$; the opposite case may be handled by symmetry.
Additionally, we assume that $t_j \ne t_k \pmod{\Z}, j \ne k$, and shift the $s$ contour to $-\infty$ to obtain
\begin{align*}
	\mathcal{F} =& \sum_{\ell\in\set{0,1}} \sum_{j=1}^p \sum_{0 \le n \equiv \alpha_j+\ell \summod{2}} \frac{(2\pi i)^n \chi_{x+t_j+n}^\ell(z)}{n!} \frac{\prod_{k\ne j} G_{\alpha_k+\ell}(t_k-t_j-n)}{\prod_{k=1}^q G_{\beta_k+\ell}(u_k-t_j-n)} f(\ell, -t_j-n).
\end{align*}
Substituting $\ell \mapsto \ell+\alpha_j$ (all terms are invariant under $\ell \mapsto \ell+2$), note that $\chi_n^n(z)=z^n$, so we have
\begin{align*}
	\mathcal{F} =& \sum_{j=1}^p \chi_{x+t_j}^{\alpha_j}(z) \sum_{\ell\in\set{0,1}} \sum_{0 \le n \equiv \ell \summod{2}} \frac{(2\pi i z)^n}{n!} \frac{\prod_{k\ne j} G_{\alpha_j+\alpha_k+\ell}(t_k-t_j-n)}{\prod_{k=1}^q G_{\alpha_j+\beta_k+\ell}(u_k-t_j-n)} f(\ell+\alpha_j, -t_j-n).
\end{align*}

The combination of \eqref{eq:GdReflect}, \eqref{eq:GdToRGamma} and \eqref{eq:RdShift} becomes
\begin{align}
\label{eq:Gminusn}
	G_{\gamma+\ell}(-s-n) =& \frac{\pi (-1)^\gamma i^n (2\pi)^{s+n}}{R_\gamma(1+s) \Gamma(1+s+n)},
\end{align}
when $n \equiv \ell \pmod{2}$.
Applying this to $\mathcal{F}$ and combining the two parities for the $n$ sum gives
\begin{equation}
\label{eq:mcFEval}
\begin{aligned}
	\mathcal{F} =& (-1)^{A+B} \pi^{p-q-1} (2\pi)^{U-T} \sum_{j=1}^p (-1)^{(p-q)\alpha_j} (2\pi)^{(p-q)t_j}\chi_{x+t_j}^{\alpha_j}\paren{z} \frac{\prod_{k=1}^q R_{\alpha_j+\beta_k}(1+t_j-u_k)}{\prod_{k\ne j} R_{\alpha_j+\alpha_k}(1+t_j-t_k)} \\
	& \qquad \times \sum_{n=0}^\infty \frac{((2\pi i)^{p-q} z)^n}{n!} \frac{\prod_{k=1}^q \Gamma(1+t_j-u_k+n)}{\prod_{k\ne j} \Gamma(1+t_j-t_k+n)} f(\alpha_j+n, -t_j-n).
\end{aligned}
\end{equation}
where $T=t_1+\ldots+t_p$, $U=u_1+\ldots+u_q$, $A = \alpha_1+\ldots+\alpha_p$ and $B=\beta_1+\ldots+\beta_q$.

\subsection{The $w_{31}$ Mellin-Barnes integral}
For $w_{31}$, deriving the power series expansion \eqref{eq:KwToJw},\eqref{eq:Jw31} is a straight-forward application of \eqref{eq:mcFEval} using $A=\Delta$, $x=\frac{3}{2}$, $z=-y_3$, $p=4$, $q=0$, $t=-\mu$, $\alpha=\Delta-\delta$, $S=T=B=0$, and $f=1$.

\subsection{The $w_{22}$ Mellin-Barnes integral}
This case is again a relatively straight-forward application of \eqref{eq:mcFEval} using $p=6$, $q=0$, $x=2$, $z=y_2$, $t=(\mu_j+\mu_k)_{j<k}$, $\alpha=(\delta_j+\delta_k)_{j<k}$, $A=\Delta$, $T=U=B=0$, and $f(\ell,s)=1/G_\Delta(2s)$.

\subsection{The $w_{121}$ Mellin-Barnes integral}
For $w_{121}$, to maintain absolute convergence, one can either shift both integrals in stages or substitute $(u_1,u_2)=(s_1+s_2,s_1-s_2)$ and shift first $\Re(u_1) \to -\infty$, then shift $\Re(u_2) \to \pm\infty$ as appropriate, at which point the first index becomes $\Min{m_1,m_2}$ while the second is $\Max{m_1,m_2}-\Min{m_1,m_2}$.
In any case, the original integral is clearly Weyl-invariant and the residue at, say, $s_2=\mu_4-m_2$ has no poles at $s_1=-\mu_4-m_1$, so the final series representation is the sum over $m_1,m_2 \in \N_0$ of the residue at $(s_1,s_2)=(-\mu_1-m_1,\mu_4-m_2)$, summed over the Weyl group.
That residue is
\begin{align*}
	& \res_{s_1=-\mu_4-m_1} \res_{s_2=\mu_4-m_2} \what{K}_{121}(s,\ell,\mu,\delta) \\
	&= \frac{4(-1)^{\delta_2+\delta_3+\delta_4+m_1}(2\pi i)^{m_1+m_2}}{m_1!\,m_2!} \\
	& \qquad \times \frac{G_{\ell_1}(-\mu_1-m_1,(\mu_2,\mu_3,\mu_4),(\delta_2,\delta_3,\delta_4)) G_{\ell_2}(\mu_4-m_2,-(\mu_1,\mu_2,\mu_3),\Delta-(\delta_1,\delta_2,\delta_3))}{G_{\ell_1+\ell_2+\Delta}(\mu_4-\mu_1-m_1-m_2)} \\
	&= 4 C_{121}(\mu,\delta) a_{121,m_1,m_2}(\mu),
\end{align*}
if $m_1 \equiv \ell_1+\delta_1 \pmod{2}$, $m_2 \equiv \ell_2+\Delta-\delta_4 \pmod{2}$ and zero otherwise.

\subsection{The $w_{211}$ Mellin-Barnes integral}
Define
\begin{align*}
	& F_{211}((s_1,s_3),(\ell_1,\ell_3),\mu,\delta) = \\
	& \paren{\prod_{j=1}^3 G_{\ell_1+\Delta-\delta_j-\delta_4}(1-s_1+\mu_j+\mu_4) G_{\ell_3+\Delta-\delta_j}(1-s_3+\mu_j)} (-1)^{\ell_3} \frac{1}{2} \sum_{\ell_2\in\set{0,1}} \\
	& \times \int_{\Re(s_2)=\frac{1}{7}} \frac{G_{\ell_2}(s_2,-(\mu_1,\mu_2,\mu_3),\Delta-(\delta_1,\delta_2,\delta_3))}{G_{\ell_1+\ell_2+\Delta-\delta_4}(s_2+s_1+\mu_4) G_{\ell_1+\ell_2+\delta_4}(s_2+1-s_1+\mu_4) G_{\ell_2+\ell_3}(s_2+1-s_3)} \frac{ds_2}{2\pi i}
\end{align*}
so that (using \eqref{eq:GdReflect}) we are trying to show
\begin{align*}
	K_{211}(y,\mu,\delta) =& \frac{(-1)^\Delta}{4} \sum_{\ell_1,\ell_3\in\set{0,1}} \int_{\Re(s_1,s_3)=(\frac{1}{7},\frac{1}{7})} \chi_{2-s_1}^{\ell_1}(y_2) \chi_{\frac{3}{2}-s_3}^{\ell_3}(y_3) \paren{\prod_{j<k} G_{\ell_1+\delta_j+\delta_k}(s_1+\mu_j+\mu_k)} \\
	& \qquad \times G_{\ell_3}(s_3,-\mu,\Delta-\delta) F_{211}((s_1,s_3),(\ell_1,\ell_3),\mu,\delta) \frac{ds_1\, ds_3}{(2\pi i)^2}.
\end{align*}

From \eqref{eq:mcFEval}, with $p=q=3$, $x=0$, $z=f=1$, $t=-(\mu_1,\mu_2,\mu_3)$, $\alpha=\Delta-(\delta_1,\delta_2,\delta_3)$, $u=(s_1+\mu_4,1-s_1+\mu_4,1-s_3)$, $\beta=(\ell_1+\ell_2+\Delta-\delta_4,\ell_1+\ell_2+\delta_4,\ell_2+\ell_3)$, $T=\mu_4$, $U=2-s_3+2\mu_4$, $A=\delta_4$, $B=\ell_3+\Delta$, we have
\begin{align*}
	& F_{211}((s_1,s_3),(\ell_1,\ell_3),\mu,\delta) = \\
	& \paren{\prod_{j=1}^3 G_{\ell_1+\Delta-\delta_j-\delta_4}(1-s_1+\mu_j+\mu_4) G_{\ell_3+\Delta-\delta_j}(1-s_3+\mu_j)} (-1)^{\Delta-\delta_4} (2\pi)^{2-s_3+\mu_4} \pi^{-1} \\
	& \times \sum_{j=1}^3 \frac{R_{\ell_1+\delta_j+\delta_4}(1-s_1-\mu_j-\mu_4) R_{\ell_1+\Delta-\delta_j-\delta_4}(s_1-\mu_j-\mu_4) R_{\ell_3+\Delta-\delta_j}(s_3-\mu_j)}{\prod_{k \in \set{1,2,3}\setminus\set{j}} R_{\delta_j+\delta_k}(1+\mu_k-\mu_j)} \\
	& \times \sum_{n=0}^\infty \frac{1}{n!} \frac{\Gamma(1-s_1-\mu_j-\mu_4+n) \Gamma(s_1-\mu_j-\mu_4+n) \Gamma(s_3-\mu_j+n)}{\prod_{k \in \set{1,2,3}\setminus\set{j}} \Gamma(1+\mu_k-\mu_j+n)}.
\end{align*}

Applying \eqref{eq:RGtoPoch} gives
\begin{align*}
	& F_{211}((s_1,s_3),(\ell_1,\ell_3),\mu,\delta) = \\
	& (-1)^{\Delta+\ell_1+\ell_3} (2\pi)^{-\mu_4} \pi^2 \sum_{j=1}^3 (2\pi)^{-3\mu_j} G_{\ell_1+\delta_j+\delta_4}(1-s_1-\mu_j-\mu_4) \\
	& \prod_{k \in \set{1,2,3}\setminus\set{j}} \frac{G_{\ell_1+\Delta-\delta_k-\delta_4}(1-s_1+\mu_k+\mu_4) G_{\ell_3+\Delta-\delta_k}(1-s_3+\mu_k)}{R_{\delta_j+\delta_k}(1+\mu_k-\mu_j)} \\
	& \sum_{n=0}^\infty \frac{1}{n!} \frac{\Poch{1-s_1-\mu_j-\mu_4}{n} \Poch{s_1-\mu_j-\mu_4}{n} \Poch{s_3-\mu_j}{n}}{\prod_{k \in \set{1,2,3}\setminus\set{j}} \Gamma(1+\mu_k-\mu_j+n)}.
\end{align*}

Using the regularized (pole-free) hypergeometric function, we have
\begin{equation}
\label{eq:F211Fstar}
\begin{aligned}
	& F_{211}((s_1-m_1,s_3-m_2),(\ell_1+m_1,\ell_3+m_2),\mu,\delta) = \\
	& (-1)^{\Delta+\ell_1+\ell_3} i^{m_1} (2\pi)^{3s_1+2s_3-3m_1-2m_2} \pi^{-3} \sum_{w \in \Weyl_3/\Weyl_2} R_{\ell_1+\delta^w_3+\delta^w_4}(1-s_1-\mu^w_3-\mu^w_4) \\
	& \times \prod_{k \in \set{1,2}} \frac{R_{\ell_1+\Delta-\delta^w_k-\delta^w_4}(1-s_1+\mu^w_k+\mu^w_4) R_{\ell_3+\Delta-\delta^w_k}(1-s_3+\mu^w_k)}{R_{\delta^w_3+\delta^w_k}(1+\mu^w_k-\mu^w_3)} \\
	& \times \Gamma(1+m_1-s_1-\mu^w_3-\mu^w_4) \prod_{k \in \set{1,2}} \Gamma(1+m_1-s_1+\mu^w_k+\mu^w_4) \Gamma(1+m_2-s_3+\mu^w_k) \\
	& \times \pFqStar32{1+m_1-s_1-\mu^w_3-\mu^w_4, -m_1+s_1-\mu^w_3-\mu^w_4, -m_2+s_3-\mu^w_3}{1+\mu^w_1-\mu^w_3, 1+\mu^w_2-\mu^w_3}{1},
\end{aligned}
\end{equation}
for $m \in \Z^2$, after applying \eqref{eq:GdToRGamma} and \eqref{eq:RdShift}, where $\Weyl_n$ is the subgroup of the Weyl group that permutes the indices $\set{1,\ldots,n}$.
For $\Re(\mu)=0$, this converges and has no poles on $\Re(s_1), \Re(s_3) < 1$.

Applying \eqref{eq:mcFEval} twice with first
\begin{gather*}
\begin{aligned}
	p=&6,& q=&0,& x=&2,& z=&y_2,& A=&\Delta,& T=&B=U=0,
\end{aligned}\\
\begin{aligned}
	t=&(\mu_j+\mu_k)_{j<k},& \alpha=&(\delta_j+\delta_k)_{j<k},& f(a,b) =&F_{211}((b,s_3),(a,\ell_3),\mu,\delta)
\end{aligned}
\end{gather*}
then
\begin{gather*}
\begin{aligned}
	p=&4,& q=&0,& x=&\tfrac{3}{2},& z=&y_3,& A=&\Delta,& T=&B=U=0,
\end{aligned}\\
\begin{aligned}
	t=&-\mu,& \alpha=&\Delta-\delta,& f(a,b) =&F_{211}((\cdot,b),(\cdot,a),\mu,\delta),
\end{aligned}
\end{gather*}
we have
\begin{equation}
\label{eq:Kw211MBtoPS}
\begin{aligned}
	& K_{211}(y,\mu,\delta) = \\
	& (-1)^\Delta \pi^8 \sum_{j<k} \sum_{j'=1}^4 \frac{(2\pi)^{6(\mu_j+\mu_k)+4\mu_{j'}}\chi_{2+\mu_j+\mu_k}^{\delta_j+\delta_k}(y_2) \chi_{\frac{3}{2}-\mu_{j'}}^{\Delta-\delta_{j'}}(y_3)}{\prod_{\substack{j''<k''\\(j'',k'')\ne(j,k)}} R_{\delta_j+\delta_k+\delta_{j''}+\delta_{k''}}(1+\mu_j+\mu_k-\mu_{j''}-\mu_{k''})} \\
	& \times \frac{1}{\prod_{j''\ne j'} R_{\delta_{j'}+\delta_{j''}}(1-\mu_{j'}+\mu_{j''})} \\
	& \times \sum_{m_1,m_2 \ge 0} \frac{\paren{-64 \pi^6 y_2}^{m_1} \paren{16 \pi^4 y_3}^{m_2}}{\prod_{j''<k''} \Gamma(1+\mu_j+\mu_k-\mu_{j''}-\mu_{k''}+m_1) \prod_{j''=1}^4 \Gamma(1-\mu_{j'}+\mu_{j''}+m_2)} \\
	& \times F_{211}((-\mu_j-\mu_k-m_1,\mu_{j'}-m_2),(m_1+\delta_j+\delta_k,m_2+\Delta-\delta_{j'}),\mu,\delta).
\end{aligned}
\end{equation}
Note: We used $m_1! \, m_2! = \Gamma(1+m_1) \Gamma(1+m_2)$ to simplify the expression a little.

By the $\Weyl_3$ invariance, to verify this identity, it is sufficient to check that:
\begin{enumerate}[label=F\arabic*.]
\item the term with $j=1,k=2,j'=1$ is $0$,
\item the term with $j=1,k=3,j'=1$ is $0$,
\item the term with $j=1,k=4,j'=1$ is $0$,
\item the term with $j=1,k=4,j'=4$ is $0$,
\item the term with $j=1,k=2,j'=4$ is $C_{211}(\mu,\delta) J_{211}(y,\mu,\delta)$,
\item the term with $j=2,k=4,j'=1$ is $C_{211}(\mu^w,\delta^w) J_{211}(y,\mu^w,\delta^w)$ where $w=w_{(1\,4)}^{-1}$,
\item the term with $j=2,k=3,j'=1$ is $C_{211}(\mu^w,\delta^w) J_{211}(y,\mu^w,\delta^w)$ where $w=w_{(1\,3\,4)}^{-1}$,
\item the term with $j=1,k=4,j'=2$ is $C_{211}(\mu^w,\delta^w) J_{211}(y,\mu^w,\delta^w)$ where $w=w_{(1\,4\,2)}^{-1}$.
\end{enumerate}

\subsubsection{The empty terms}
For F5, we have
\[ F_{211}((-\mu_1-\mu_2-m_1,\mu_1-m_2),(m_1+\delta_1+\delta_2,m_2+\Delta-\delta_1),\mu,\delta) = 0, \]
because each of the three terms has a factor $R_0(1)=0$.
Similarly, for F6,
\[ F_{211}((-\mu_1-\mu_3-m_1,\mu_1-m_2),(m_1+\delta_1+\delta_3,m_2+\Delta-\delta_1),\mu,\delta) = 0, \]
and for F7,
\[ F_{211}((-\mu_1-\mu_4-m_1,\mu_1-m_2),(m_1+\delta_1+\delta_4,m_2+\Delta-\delta_1),\mu,\delta) = 0. \]

To check F8, we need to show
\[ F_{211}((-\mu_1-\mu_4-m_1,\mu_4-m_2),(m_1+\delta_1+\delta_4,m_2+\Delta-\delta_4),\mu,\delta) = 0, \]
and removing out common factors (with the identity $R_\delta(1-s) = (-1)^\delta R_\delta(1+s)$) from the two surviving summands reduces this to a hypergeometric identity, i.e.
\begin{align*}
	0 =& \pFqDagger32{1 + m_1 + \mu_1 - \mu_2, -m_1 + \mu_3 -\mu_4, -m_2 - \mu_2 + \mu_4}{1 - \mu_2 + \mu_3, 1 + \mu_1 - \mu_2}{1} \\
	& - \pFqDagger32{1 + m_1 + \mu_1 - \mu_3, -m_1 + \mu_2 -\mu_4, -m_2 - \mu_3 + \mu_4}{1 + \mu_1 - \mu_3, 1 + \mu_2 - \mu_3}{1}
\end{align*}
The case $m_1=0$ follows from \eqref{eq:pFqFirstZero} and Gauss' hypergeometric identity \eqref{eq:Gauss2F1}, while the contiguous relation \cite[eq (3.4)]{Bailey02} implies that the right-hand side of the previous display satisfies the recurrence relation
\begin{align*}
	0 =&m_1 (m_1 + \mu_1 - \mu_2) (m_1 + \mu_1 - \mu_3) (1 + m_1 - \mu_2 - \mu_3) a_{m_1-1} \\
	& \qquad - (1 + m_1 - 2 \mu_2 - 2 \mu_3) (1+m_1+\mu_4-\mu_3) (1+m_1+\mu_4-\mu_2) (m_1 - \mu_2 - \mu_3) a_{m_1+1} \\
	& \qquad - (1 + 2 m_1 - 2 \mu_2 - 2 \mu_3) \biggl(m_1 (1 + m_1 - 2 \mu_2 - 2 \mu_3) (1 + 2 m_2 - 2 \mu_4) \\
	& \qquad \qquad + (\mu_2 + \mu_3) ((1 + 2 m_2 + \mu_1 - 2 \mu_4) (-1 + \mu_2 + \mu_3) + \mu_1 (1 + \mu_1) + \mu_2 \mu_3)\biggr) a_{m_1},
\end{align*}
and this gives the identity for $m_1 \in \N_0$.

\subsubsection{The symmetries}
For $1 \le j < k \le 4$ and $1 \le j' \le 4$, define
\begin{align*}
	& f_{j,k,j',m_1,m_2} = \\
	& \frac{(-1)^\Delta \pi^8 (2\pi)^{6(\mu_j+\mu_k)+4\mu_{j'}} F_{211}((-\mu_j-\mu_k-m_1,\mu_{j'}-m_2),(m_1+\delta_j+\delta_k,m_2+\Delta-\delta_{j'}),\mu,\delta)}{\paren{\prod_{\substack{j''<k''\\(j'',k'')\ne(j,k)}} R_{\delta_j+\delta_k+\delta_{j''}+\delta_{k''}}(1+\mu_j+\mu_k-\mu_{j''}-\mu_{k''})} \paren{\prod_{j''\ne j'} R_{\delta_{j'}+\delta_{j''}}(1-\mu_{j'}+\mu_{j''})}} \\
	& \times \frac{1}{\paren{\prod_{j''<k''} \Gamma(1+\mu_j+\mu_k-\mu_{j''}-\mu_{k''}+m_1)} \paren{\prod_{j''=1}^4 \Gamma(1-\mu_{j'}+\mu_{j''}+m_2)}}.
\end{align*}

By \eqref{eq:F211Symmetries} to see F5 and the symmetries F6, F7 and F8, it suffices to note that
\begin{itemize}
\item $f_{1,2,4,m_1,m_2} = C^*_{211}(\mu,\delta) a^*_{211,m_1,m_2}(\mu)$,
\item $f_{2,4,1,m_1,m_2} = C^*_{211}(\mu^w,\delta^w) b^*_{211,m_1,m_2}(\mu^w)$ where $w=w_{(1\,4)}^{-1}$,
\item $f_{2,3,1,m_1,m_2} = C^*_{211}(\mu^w,\delta^w) c^*_{211,m_1,m_2}(\mu^w)$ where $w=w_{(1\,3\,4)}^{-1}$,
\item $f_{1,4,2,m_1,m_2} = C^*_{211}(\mu^w,\delta^w) b^*_{211,m_1,m_2}(\mu^w)$, where $w=w_{(1\,4\,2)}^{-1}$,
\end{itemize}
which follow from \eqref{eq:F211Fstar}, after writing
\[ C^*_{211}(\mu,\delta) = (-1)^{\delta_3} (2\pi)^{2\mu_1+2\mu_2-\mu_3-3\mu_4} \pi^5 \prod_{(j,k) \in \mathcal{S}_{211}} R_{\delta_j+\delta_k}(1+\mu_j-\mu_k)\Gamma(1+\mu_j-\mu_k). \]
In these cases, only one term of $F_{211}$ is nonzero, after applying $R_0(1)=0$.

\subsection{The $w_{1111}$ Mellin-Barnes integral}
Define
\begin{align*}
	& F_{1111}(s,\ell,\mu,\delta) =\frac{1}{2} (-1)^{\delta_3+\delta_4} \sum_{\eta \in \set{0,1}} \int_{\Re(u)=\epsilon} \\
	& \frac{G_{\eta}(u,(\mu_3,\mu_4),(\delta_3,\delta_4)) G_{\ell_1+\ell_2+\eta+\delta_1+\delta_2}(u+1-s_2-\mu_1-\mu_2) G_{\ell_2+\ell_3+\eta+\Delta}(u+1-s_2-s_3)}{G_{\ell_1+\eta}(u+1-s_1) G_{\ell_2+\eta}(u+1-s_2,-(\mu_1,\mu_2),(\delta_1,\delta_2)) G_{\ell_3+\eta+\delta_1+\delta_2}(u+1-s_3-\mu_1-\mu_2)} \frac{du}{2\pi i},
\end{align*}
so that
\begin{align*}
	K_{1111}(y,\mu,\delta) =& \frac{(-1)^\Delta}{8} \sum_{\ell\in\set{0,1}^3} \int_{\Re(s)=(2\epsilon,2\epsilon,2\epsilon)} \chi_{\frac{3}{2}-s_1}^{\ell_1}(-y_1) \chi_{2-s_2}^{\ell_2}(-y_2) \chi_{\frac{3}{2}-s_3}^{\ell_3}(-y_3) \\
	& \times G_{\ell_1}(s_1,(\mu_1,\mu_2),(\delta_1,\delta_2)) G_{\ell_2}(s_2,(\mu_1+\mu_2,\mu_3+\mu_4),(\delta_1+\delta_2,\delta_3+\delta_4)) \\
	& \times G_{\ell_3}(s_3,-(\mu_1,\mu_2),\Delta-(\delta_1,\delta_2)) F_{1111}(s,\ell,\mu,\delta) \frac{ds}{(2\pi i)^3}.
\end{align*}

We apply \eqref{eq:mcFEval} with
\begin{gather*}
\begin{aligned}
	x=&0,& z=&1,& f=&1,& p&=q=4,
\end{aligned}\\
\begin{aligned}
	t=&(\mu_3,\mu_4,1-s_1-s_2-\mu_1-\mu_2,1-s_2-s_3),& T=&2-2s_2-s_3-2\mu_1-2\mu_2,
\end{aligned}\\
\begin{aligned}
	\alpha=&(\delta_3,\delta_4,\ell_1+\ell_2+\delta_1+\delta_2,\ell_2+\ell_3+\Delta), &A=&\ell_1+\ell_3,
\end{aligned}\\
\begin{aligned}
	u=&(1-s_1,1-s_2-\mu_1,1-s_2-\mu_2,1-s_3-\mu_1-\mu_2), &U=&4-s_1-2s_2-s_3-2\mu_1-2\mu_2,
\end{aligned}\\
\begin{aligned}
	\beta=&(\ell_1,\ell_2+\delta_1,\ell_2+\delta_2,\ell_3+\delta_1+\delta_2), &B=&\ell_1+\ell_3,
\end{aligned}
\end{gather*}
then in terms of the regularized hypergeometric functions, we get
\[ F_{1111}(s,\ell,\mu,\delta) = F_A(s,\ell,\mu,\delta)+F_B(s,\ell,\mu,\delta)+F_C(s,\ell,\mu,\delta)+F_D(s,\ell,\mu,\delta), \]
where
\begingroup
\allowdisplaybreaks
\begin{align*}
	F_A :=& \frac{G_{\ell_1+\delta_3}(s_1+\mu_3) G_{\ell_2+\delta_1+\delta_3}(s_2+\mu_1+\mu_3) G_{\ell_2+\delta_2+\delta_3}(s_2+\mu_2+\mu_3)}{R_{\delta_3+\delta_4}(1+\mu_3-\mu_4) R_{\ell_1+\ell_2+\Delta-\delta_4}(s_1+s_2-\mu_4) R_{\ell_2+\ell_3+\Delta-\delta_3}(s_2+s_3+\mu_3)} \\*
	& \qquad \times (-1)^{\delta_3+\delta_4} (2\pi)^{-2+s_1+2s_2+s_3+2\mu_3-2\mu_4} \pi^3 G_{\ell_3+\Delta-\delta_4}(s_3-\mu_4) \\*
	& \qquad \times \pFqStar43{s_1+\mu_3, s_2+\mu_1+\mu_3, s_2+\mu_2+\mu_3, s_3-\mu_4}{1+\mu_3-\mu_4, s_1+s_2-\mu_4, s_2+s_3+\mu_3}{1}, \\
	F_B :=& \frac{G_{\ell_1+\delta_4}(s_1+\mu_4) G_{\ell_2+\delta_1+\delta_4}(s_2+\mu_1+\mu_4) G_{\ell_2+\delta_2+\delta_4}(s_2+\mu_2+\mu_4)}{R_{\delta_3+\delta_4}(1+\mu_4-\mu_3) R_{\ell_1+\ell_2+\Delta-\delta_3}(s_1+s_2-\mu_3) R_{\ell_2+\ell_3+\Delta-\delta_4}(s_2+s_3+\mu_4)} \\*
	& \qquad \times (-1)^{\delta_3+\delta_4} G_{\ell_3+\Delta-\delta_3}(s_3-\mu_3) (2\pi)^{-2+s_1+2s_2+s_3+2\mu_4-2\mu_3} \pi^3 \\*
	& \qquad \times \pFqStar43{s_1+\mu_4, s_2+\mu_1+\mu_4, s_2+\mu_2+\mu_4, s_3-\mu_3}{1+\mu_4-\mu_3, s_1+s_2-\mu_3, s_2+s_3+\mu_4}{1}, \\
	F_C :=&\frac{G_{\ell_1+\delta_1}(1-s_1-\mu_1) G_{\ell_1+\delta_2}(1-s_1-\mu_2) G_{\ell_2+\delta_1+\delta_2}(1-s_2-\mu_1-\mu_2)}{R_{\ell_1+\ell_2+\Delta-\delta_4}(s_1+s_2-\mu_4) R_{\ell_1+\ell_2+\Delta-\delta_3}(s_1+s_2-\mu_3) R_{\ell_1+\ell_3+\delta_3+\delta_4}(1-s_1+s_3-\mu_1-\mu_2)} \\*
	& \qquad \times (2\pi)^{2-3s_1-2s_2+s_3+2\mu_3+2\mu_4} \pi^3 G_{\ell_1+\ell_2+\ell_3}(1-s_1-s_2+s_3) \\*
	& \qquad \times \pFqStar43{1-s_1-\mu_1, 1-s_1-\mu_2, 1-s_2-\mu_1-\mu_2, 1-s_1-s_2+s_3}{2-s_1-s_2+\mu_4, 2-s_1-s_2+\mu_3, 1-s_1+s_3-\mu_1-\mu_2}{1}, \\
	F_D :=& \frac{G_{\ell_2+\delta_3+\delta_4}(1-s_2-\mu_3-\mu_4) G_{\ell_3+\Delta-\delta_1}(1-s_3+\mu_1) G_{\ell_3+\Delta-\delta_2}(1-s_3+\mu_2)}{R_{\ell_1+\ell_3+\delta_3+\delta_4}(1+s_1-s_3-\mu_3-\mu_4) R_{\ell_2+\ell_3+\Delta-\delta_3}(s_2+s_3+\mu_3) R_{\ell_2+\ell_3+\Delta-\delta_4}(s_2+s_3+\mu_4)} \\*
	& \qquad \times (2\pi)^{2+s_1-2s_2-3s_3+2\mu_1+2\mu_2} \pi^3 G_{\ell_1+\ell_2+\ell_3+\Delta}(1+s_1-s_2-s_3) \\*
	& \qquad \times \pFqStar43{1-s_2-\mu_3-\mu_4, 1-s_3+\mu_1, 1-s_3+\mu_2, 1+s_1-s_2-s_3}{1+s_1-s_3-\mu_3-\mu_4, 2-s_2-s_3-\mu_3, 2-s_2-s_3-\mu_4}{1},
\end{align*}
\endgroup
using the identity $R_\delta(2-s) = (-1)^{\delta+1} R_\delta(s)$.
We split $\what{K}_{1111}$ correspondingly into
\[ \what{K}_{1111}(s,\ell,\mu,\delta) = \what{K}_A(s,\ell,\mu,\delta)+\what{K}_B(s,\ell,\mu,\delta)+\what{K}_C(s,\ell,\mu,\delta)+\what{K}_D(s,\ell,\mu,\delta), \]
and we imagine (as opposed to actually writing it out) factoring each $\what{K}$ term into a trigonometric part (the $R_\delta$ functions), the gamma factors (the $G_\delta$ functions), and the (regularized) hypergeometric function.
We first eliminate the potential poles arising from the trigonometric factors, then consider those of the gamma factors, while the hypergeometric factor has no poles by design.

We first shift the $s_2$ contour to $-\infty$, then the $s_1$ contour, then the $s_3$ contour.

\subsubsection{The trigonometric poles}
Let us rule out the potential poles of the trigonometric factors:
The identities \eqref{eq:4F3Denom1}, \eqref{eq:RdReflect} and \eqref{eq:RdShift} imply
\begin{align*}
	\res_{s_2=-s_1+\mu_4-n} \what{K}_A = -\res_{s_2=-s_1+\mu_4-n} \what{K}_C, \qquad \N_0 \ni n \equiv 1+\ell_1+\ell_2+\Delta-\delta_4 \pmod{2}, \\
	\res_{s_2=-s_3-\mu_3-n} \what{K}_A = -\res_{s_2=-s_3-\mu_3-n} \what{K}_D, \qquad \N_0 \ni n \equiv 1+\ell_2+\ell_3+\Delta-\delta_3 \pmod{2}, \\
	\res_{s_2=-s_3-\mu_4-n} \what{K}_B = -\res_{s_2=-s_3-\mu_4-n} \what{K}_D, \qquad \N_0 \ni n \equiv 1+\ell_2+\ell_3+\Delta-\delta_4 \pmod{2}, \\
	\res_{s_2=-s_1+\mu_3-n} \what{K}_B = -\res_{s_2=-s_1+\mu_3-n} \what{K}_C, \qquad \N_0 \ni n \equiv 1+\ell_1+\ell_2+\Delta-\delta_3 \pmod{2},
\end{align*}
so the sum has no poles there.

There remains the possibility of a pole at $s_1=s_3+\mu_3+\mu_4-m_1$, $\N_0 \ni m_1 \equiv \ell_1+\ell_3+\delta_3+\delta_4 \pmod{2}$ for $\what{K}_C$ and $\what{K}_D$.
Notice that
\[ \res_{s_2=-\mu_i-\mu_j-m_2} \what{K}_C = 0, \qquad m_2 \in \N_0 \]
unless $i=3,j=4$ and $m_2 \equiv \ell_2+\delta_i+\delta_j \pmod{2}$; in this case, at $s_1=s_3+\mu_3+\mu_4-m_1$, $\N_0 \ni m_1 \equiv \ell_1+\ell_3+\delta_3+\delta_4 \pmod{2}$, the zero of $R_{\ell_1+\ell_3+\delta_3+\delta_4}(1-s_1+s_3-\mu_1-\mu_2)$ in the denominator is canceled by the zero of $G_{\ell_1+\ell_2+\ell_3}(1+s_1-s_2-s_3)$ in the numerator.
Similarly,
\[ \res_{s_2=-\mu_i-\mu_j-m_2} \what{K}_D = 0, \qquad m_2 \in \N_0 \]
unless $i=1,j=2$ and $m_2 \equiv \ell_2+\delta_i+\delta_j \pmod{2}$, but if $s_1=s_3+\mu_3+\mu_4-m_1$, $\N_0 \ni m_1 \equiv \ell_1+\ell_3+\delta_3+\delta_4 \pmod{2}$, then we apply \eqref{eq:4F3Denom1}, at which point the product
\[ \Poch{1+s_1-s_2-s_3}{m_1} G_{\ell_1+\ell_2+\ell_3+\Delta}(1+s_1-s_2-s_3) \]
is zero.

\subsubsection{The empty terms}
Now we rule out the potential poles which do not arise from the action of the Weyl group applied to $s_1=-\mu_1-m_1$, $s_2=-\mu_1-\mu_2-m_2$, $s_3=-\mu_1-\mu_2-\mu_3-m_3$.

By the $w_{(1\,2)}$ and $w_{(3\,4)}$ invariance, it's enough to check that the following residues are all zero:\\
\begin{center}\begin{tabular}{c|c|c|c}
	&$s_2$ & $s_1$ & $s_3$\\ \hline
	F1& $-\mu_1-\mu_2-m_2$ & $-\mu_3-m_1$ & -- \\ \hline
	F2& $-\mu_1-\mu_2-m_2$ & $-\mu_1-m_1$ & $-\mu_1-\mu_3-\mu_4-m_3$ \\ \hline
	F3& $-\mu_1-\mu_2-m_2$ & $-\mu_1-m_1$ & $-\mu_2-\mu_3-\mu_4-m_3$ \\ \hline
	F4& $-\mu_1-\mu_3-m_2$ & $-\mu_2-m_1$ & -- \\ \hline
	F5& $-\mu_1-\mu_3-m_2$ & $-\mu_4-m_1$ & -- \\ \hline
	F6& $-\mu_1-\mu_3-m_2$ & $-\mu_1-m_1$ & $-\mu_1-\mu_2-\mu_4-m_3$ \\ \hline
	F7& $-\mu_1-\mu_3-m_2$ & $-\mu_1-m_1$ & $-\mu_2-\mu_3-\mu_4-m_3$ \\ \hline
	F8& $-\mu_1-\mu_3-m_2$ & $-\mu_3-m_1$ & $-\mu_1-\mu_2-\mu_4-m_3$ \\ \hline
	F9& $-\mu_1-\mu_3-m_2$ & $-\mu_3-m_1$ & $-\mu_2-\mu_3-\mu_4-m_3$ \\ \hline
	F10& $-\mu_3-\mu_4-m_2$ & $-\mu_1-m_1$ & -- \\ \hline
	F11& $-\mu_3-\mu_4-m_2$ & $-\mu_3-m_1$ & $-\mu_1-\mu_2-\mu_3-m_3$ \\ \hline
	F12& $-\mu_3-\mu_4-m_2$ & $-\mu_3-m_1$ & $-\mu_1-\mu_2-\mu_4-m_3$ \\ \hline
\end{tabular}\end{center}
All of these follow from the identity \eqref{eq:4F3Denom2}, except F2, F3, F10.
For F10, it is sufficient to apply \eqref{eq:4F3Denom2} to the residues in $s_3$ except at $s_3=-\mu_1-\mu_3-\mu_4-m_3$ and $s_3=-\mu_2-\mu_3-\mu_4-m_3$.

The four remaining cases leave us to show
\begingroup
\allowdisplaybreaks
\begin{align*}
	\text{F2:}&&\pFqDagger43{-m_1 - \mu_1 + \mu_3, -m_2 - \mu_2 + \mu_3, -m_2 - \mu_1 + \mu_3, -m_3 + \mu_2 - 
  \mu_4}{1 + \mu_3 - \mu_4, -m_1 - m_2 - \mu_1 + \mu_3, -m_2 - m_3 - \mu_1 + \mu_3}{1} = \\*
  && \pFqDagger43{-m_1 + \mu_4 - \mu_1, -m_2 + \mu_4 - \mu_2, -m_2 + \mu_4 - \mu_1, -m_3 + \mu_2 - 
  \mu_3}{1 + \mu_4 - \mu_3, -m_1 - m_2 + \mu_4 - \mu_1, -m_2 - m_3 + \mu_4 - \mu_1}{1}, \\
	\text{F3:}&&\pFqDagger43{-m_1 - \mu_1 + \mu_3, -m_2 - \mu_2 + \mu_3, -m_2 - \mu_1 + \mu_3, -m_3 + \mu_1 - 
  \mu_4}{1 + \mu_3 - \mu_4, -m_1 - m_2 - \mu_1 + \mu_3, -m_2 - m_3 - \mu_2 + \mu_3}{1} = \\*
  && \pFqDagger43{-m_1 - \mu_1 + \mu_4, -m_2 - \mu_2 + \mu_4, -m_2 - \mu_1 + \mu_4, -m_3 + \mu_1 - 
  \mu_3}{1 - \mu_3 + \mu_4, -m_1 - m_2 - \mu_1 + \mu_4, -m_2 - m_3 - \mu_2 + \mu_4}{1}, \\
	\text{F10a:}&&\pFqDagger43{-m_1 - \mu_1 + \mu_3, -m_2 + \mu_1 - \mu_4, -m_2 + \mu_2 - \mu_4, -m_3 + \mu_2 - 
  \mu_4}{1 + \mu_3 - \mu_4, -m_1 - m_2 + \mu_2 - \mu_4, -m_2 - m_3 + \mu_2 - \mu_4}{1} = \\*
  && \pFqDagger43{-m_1 - \mu_1 + \mu_4, -m_2 + \mu_1 - \mu_3, -m_2 + \mu_2 - \mu_3, -m_3 + \mu_2 - 
  \mu_3}{1 - \mu_3 + \mu_4, -m_1 - m_2 + \mu_2 - \mu_3, -m_2 - m_3 + \mu_2 - \mu_3}{1}, \\
	\text{F10b:}&&\pFqDagger43{-m_1 - \mu_1 + \mu_3, -m_2 + \mu_1 - \mu_4, -m_2 + \mu_2 - \mu_4, -m_3 + \mu_1 - 
  \mu_4}{1 + \mu_3 - \mu_4, -m_1 - m_2 + \mu_2 - \mu_4, -m_2 - m_3 + \mu_1 - \mu_4}{1} = \\*
  && \pFqDagger43{-m_1 - \mu_1 + \mu_4, -m_2 + \mu_1 - \mu_3, -m_2 + \mu_2 - \mu_3, -m_3 + \mu_1 - 
  \mu_3}{1 - \mu_3 + \mu_4, -m_1 - m_2 + \mu_2 - \mu_3, -m_2 - m_3 + \mu_1 - \mu_3}{1},
\end{align*}
\endgroup
for all $m \in \N_0^3$.

Note that \eqref{eq:4F3genrel} implies
\begin{equation}
\label{eq:4F3normalizedgenrel}
\begin{aligned}
	0 =& (a_3 a_4 (b_2-a_2) + b_1 (a_3 (a_2 - a_4) + a_2(a_4 - b_2)) \\
 	& \qquad + 
 a_1 (a_3 a_4 + (a_2 - a_3 - a_4) b_2 + b_1 (-a_2 + b_2))) c_{0,0,0} \\
	&+ (a_1 - b_1) (a_2 - b_2)(b_1-b_2) c_{0,-1,0} \\
	&- (b_1-a_1)(a_3 - b_2) (b_2 - a_4) c_{0,0,-1} \\
	&- (a_3 - b_1) (b_1 - a_4) (a_2-b_2) c_{-1,0,0},
\end{aligned}
\end{equation}
where
\begin{align*}
	c_{m_1,m_2,m_3} =& \pFqDagger43{a_1-m_1, a_2-m_3, a_3-m_2, a_4-m_2}{b_1-m_1-m_2, b_2-m_2-m_3, b_3}{z}.
\end{align*}

For F2, after the appropriate substitutions, we see that \eqref{eq:4F3normalizedgenrel} implies both the left- and right-hand sides satisfy
\begin{align*}
	0 =& \bigl(-m_1^2 (m_2 + \mu_1 + \mu_2 - \mu_3 - \mu_4) + m_1 (m_2 - \mu_1 + \mu_2) (m_2 + \mu_1 + \mu_2 - \mu_3 - \mu_4) \\
	& \qquad + m_2 m_3 (-m_2 + m_3 - 2 \mu_2 + \mu_3 + \mu_4)\bigr) c_{0,0,0} \\
	& - m_2 (m_1 - m_3) (m_2 + \mu_1 + \mu_2 - \mu_3 - \mu_4) c_{0,-1,0} - m_2 m_3 (m_3 + \mu_1 - \mu_2) c_{0,0,-1} \\
	& + m_1 (m_1 + \mu_1 - \mu_2) (m_2 + \mu_1 + \mu_2 - \mu_3 - \mu_4) c_{-1,0,0}.
\end{align*}
The other cases are identical.

\subsubsection{The actual residues}
By the $w_{(1\,2)}$ and $w_{(3\,4)}$ invariance, it's enough to check the following residues:\\
\begin{center}\begin{tabular}{c|c|c|c|c}
	$s_2$ & $s_1$ & $s_3$ & resulting residue & Weyl element\\ \hline
	$-\mu_1-\mu_2-m_2$ & $-\mu_1-m_1$ & $-\mu_1-\mu_2-\mu_3-m_3$ & $8 a_{1111,m}(\mu,\delta)$ & --\\ \hline
	$-\mu_1-\mu_3-m_2$ & $-\mu_1-m_1$ & $-\mu_1-\mu_2-\mu_3-m_3$ & $8 b_{1111,m}(\mu^w,\delta^w)$ & $w=w_{(2\,3)}^{-1}$ \\ \hline
	$-\mu_1-\mu_3-m_2$ & $-\mu_1-m_1$ & $-\mu_1-\mu_3-\mu_4-m_3$ & $8 c_{1111,m}(\mu^w,\delta^w)$ & $w=w_{(2\,3\,4)}^{-1}$ \\ \hline
	$-\mu_1-\mu_3-m_2$ & $-\mu_3-m_1$ & $-\mu_1-\mu_2-\mu_3-m_3$ & $8 d_{1111,m}(\mu^w,\delta^w)$ & $w=w_{(2\,1\,3)}^{-1}$ \\ \hline
	$-\mu_1-\mu_3-m_2$ & $-\mu_3-m_1$ & $-\mu_1-\mu_3-\mu_4-m_3$ & $8 e_{1111,m}(\mu^w,\delta^w)$ & $w=w_{(2\,1\,3\,4)}^{-1}$ \\ \hline
	$-\mu_3-\mu_4-m_2$ & $-\mu_3-m_1$ & $-\mu_1-\mu_3-\mu_4-m_3$ & $8 f_{1111,m}(\mu^w,\delta^w)$ & $w=w_{(1\,3)(2\,4)}^{-1}$ \\ \hline
\end{tabular}\end{center}
Here we are using $z_{1111,m}(\mu,\delta) = C_{w_l}(\mu, \delta) z_{1111,m}(\mu)$ for brevity, and as usual, we have the conditions that the residues are zero unless $m_1 \equiv \ell_1+\delta_i$, $m_2 \equiv \ell_2+\delta_i+\delta_j$, $m_3 \equiv \ell_3+\delta_i+\delta_j+\delta_k \pmod{2}$.

\section{Interchange of Integrals}
\label{sect:IoI}
\subsection{The BKY Lemma}
The reader may find the formulation of the following lemmas somewhat silly, but rather than attach some new notation to the problem, the author has simply listed out all of the relevant cases.
\begin{lem}
Let $C \ge 1$ and suppose
\[ x \in \piecewise{[-4C,-C] \cup [C,4C] & \If C > 1, \\ [-4,4] & \If C=1.} \]
Then for $f(x)$ given by any of
\begin{align*}
	\frac{C x}{1+x^2}, && \frac{C^2}{1+x^2}, && \frac{\sqrt{1+x^2}}{C}, && \frac{C}{\sqrt{1+x^2}}, && \frac{C^2 x}{\sqrt{1+x^2}}, && \frac{x}{C}, && \frac{C}{x},
\end{align*}
we have
\[ \paren{C \frac{d}{dx}}^j f(x) \ll_j 1, \]
for all $j \ge 1$, provided $C>1$ for the case $f(x) = C/x$.
\end{lem}
\begin{proof}
The lemma is trivially true for $C=1$, so suppose $C > 1$, then all but one of these may be written in the form
\[ \sgn(x)^\delta \paren{\frac{x}{C}}^\alpha \paren{1+x^{-2}}^\beta = \sgn(x)^\delta \paren{\frac{x}{C}}^\alpha \sum_{k=0}^\infty \binom{\beta}{k} x^{-2k}, \]
which clearly satisfies the conclusion, by the known properties of the binomial series (i.e. the ratio test).
In the exceptional case, we have
\[ C \frac{d}{dx} \frac{C^2 x}{\sqrt{1+x^2}} = \frac{C^3}{(1+x^2)^{3/2}}, \]
and the same argument applies.
\end{proof}

Directly from \cite[Lem. 8.1]{BKY01}, we have
\begin{lem}
Let $\epsilon>0$, $D \ge 1$, $\log_2 C \in \N_0$, $A_i \in \R$ and $A_7=A_8=0$ whenever $C=1$.
Suppose $f$ is smooth and compactly supported on
\[ \piecewise{[-4C,-C] \cup [C,4C] & \If C > 1, \\ [-4,4] & \If C=1,} \]
with $f^{(j)}(x) \ll_{j,\epsilon} D^{j\epsilon/10} C^{-j}$.
For
\[ \phi(x) = \frac{A_1 x}{1+x^2}+\frac{A_2}{1+x^2}+A_3\sqrt{1+x^2}+\frac{A_4}{\sqrt{1+x^2}}+\frac{A_5 x}{\sqrt{1+x^2}}+A_6 x+\frac{A_7}{x}, \]
consider the integral
\[ \mathcal{I}:=\int_\R f(x) \e{\phi(x)} dx. \]
We have $\mathcal{I} \ll_B D^{-B}$ for any $B>0$ unless
\[ C\abs{\phi'(x)} \le D^\epsilon \sqrt{1+\sum_i \wtilde{A}_i} \]
for some $x$ in the support of $f$, where
\[ \wtilde{A}_i = \abs{A_i} \times \piecewise{C & \If i=3,6,\\C^{-1} & \If i=1,4,7, \\ C^{-2} & \If i=2,5.} \]
\end{lem}

\begin{cor}
\label{cor:BKYLemma}
For $\phi,f,C,D,\mathcal{I},B,\epsilon$ as in the lemma, suppose each $A_i$ is a sum of terms $A_i = \sum_j A_{i,j}$ and define $\wtilde{A}_{i,j}$ as in the lemma.
Then at least one of the following is true:
\begin{enumerate}
\item[1.] $\mathcal{I} \ll_B D^{-B}$.
\item[2.] All $\wtilde{A}_{i,j} \le D^{3\epsilon}$.
\item[3.] The largest two terms are proportional, i.e. $\wtilde{A}_{i,j} \asymp \wtilde{A}_{i',j'}$ for some $(i,j) \ne (i',j')$.
\item[4.] The largest term is some $\wtilde{A}_{i,j}$ with $i\in\set{1,2,3,4}$ and $C=1$.
\end{enumerate}
\end{cor}
\begin{proof}
The right-hand side in the bound of the lemma is
\[ \le D^\epsilon \sqrt{1+\sum_{i, j} \wtilde{A}_{i,j}}. \]
If any of the $\wtilde{A}_{i,j} > D^{3\epsilon}$, we see that for some $x$ in the support of $f$, the absolute value of $C \phi'(x)$ is strictly less than the sum of the absolute values of its terms, hence the largest two terms must cancel to some degree.
The case 4 comes because for $i=1$,
\[ \frac{d}{dx} \frac{x}{1+x^2} = \frac{1-x^2}{1+x^2} \]
may be close to zero if $x$ is close $\pm 1$, and $i=2,3,4$ have similar difficulties (term-wise stationary points) near $x=0$.
\end{proof}
Note: Clearly in case 3, the largest two terms must also have differing signs in the derivative.
Similarly, there may be more than one or two terms which exceed the error bound, in which case extra cancellation must occur, but this formulation is sufficient for our purposes.
We also point out that replacing $x$ with $\abs{x}$ in any term of $\phi(x)$ does not change the result when $C>1$ as $\abs{x}$ is nicely differentiable away from zero.

\subsection{The setup}
For this section only, define the symbols $A \ll^* B$, $A \gg^* B$, $A \asymp^* B$ to mean $A \ll \abs{C}^{o(1)} B$, $A \gg \abs{C}^{-o(1)} B$, and $B \abs{C}^{-o(1)} \ll A \ll \abs{C}^{o(1)} B$, respectively.
We write the negation of $A \ll^* B$ as $A \ggg^* B$; this requires some forgiveness from the reader in regards to arbitrary constants, but one can check that nothing is circular about the following argument.

For the Strong Interchange of Integrals, we would like to define the Bessel function by the expression
\[ K(y,\mu,\delta) W_\sigma(t,\mu,\delta) = \int_{\wbar{U}_w(\R)} W_\sigma(ywxt,\mu,\delta) \wbar{\psi_I(x)} dx, \]
where $t \in Y$, $y \in Y_w$ and $W_\sigma(t,\mu,\delta)$ is the Whittaker function of $K$-type $\sigma$ and parameters $\mu,\delta$.
Unfortunately, the integral does not converge absolutely; the Strong Interchange of Integrals Conjecture of \cite[Sect. 10]{GLnI} instead proposes a choice of coordinates for such that the Riemann integral converges (conditionally, but) rapidly.

If we write $ywxt = x^* y^* k^*$ with $x^* \in U(\R)$, $y^* \in Y^+$, $k^* \in K$, the integral becomes
\begin{align}
\label{eq:KwNaiveDef}
	K_w(y,\mu,\delta) W_\sigma(t,\mu,\delta) =& \int_{\wbar{U}_w(\R)} \psi(x^*)\wbar{\psi_I(x)} W_\sigma(y^*,\mu,\delta) \sigma(k^*) dx.
\end{align}
We define
\begin{align}
\label{eq:KwPhase}
	\e{\phi(x)} =& \psi(x^*)\wbar{\psi_I(x)}
\end{align}
and call $\phi(x)=\phi_w(t,y,x)$ is the ``phase''.
Note that (the entries of) $W_\sigma(y^*,\mu,\delta)$ has super-polynomial decay in the coordinates of $y^*$ unless $y_i^* \ll 1$ and for the choice of coordinates of \cite[Sect. 10]{GLnI}, the entries of $\sigma(k^*)$ may be expressed as trigonometric polynomials in (the argument of) $\frac{1+i x_j}{\sqrt{1+x_j^2}}$ for each coordinate.
Thus the factor $W_\sigma(y^*,\mu,\delta) \sigma(k^*)$ satisfies
\[ \paren{\prod_j x_j^{m_j} \frac{\partial^{m_j}}{\partial x_j^{m_j}}} W_\sigma(y^*,\mu,\delta) \sigma(k^*) \ll_{m,\mu} 1. \]

Finally, \cite[Sect. 10]{GLnI} applies a dyadic partition of unity to the integral.
We roll $W_\sigma(y^*,\mu,\delta) \sigma(k^*)$ into some $g(\cdot,C) : \wbar{U}_w(\R) \to \C$ satisfying
\[ \paren{\prod_j x_j^{m_j} \frac{\partial^{m_j}}{\partial x_j^{m_j}}} g(x,C) \ll_m \abs{C}^{\epsilon \abs{m}} \]
that is smooth and supported on
\[ x_j \in \piecewise{[-4C_j,-C_j] \cup [C_j,4C_j] & \If C_j > 1, \\ [-4,4] & \If C_j=1.} \]
Then the explicit Strong Interchange of Integrals Conjecture becomes
\begin{align}
\label{eq:IoIBound}
	\int_{\wbar{U}_w(\R)} \e{\phi(x)} g(x,C) dx \ll_{A,\epsilon}& \abs{C}^{-A}
\end{align}
with $t \in Y$, $y \in Y_w$ satisfying $\abs{t_i}, \abs{y_i} \asymp 1$.
We say the oscillatory integral satisfying \eqref{eq:IoIBound} is ``negligible'' and we assume, for a contradiction, that this is not the case.
The use of contradiction is purely for notational simplicity and the contradiction will generally be reached when we have forced all $C_j \ll^* 1$, which implies $\abs{C} \ll 1$ so that \eqref{eq:IoIBound} is trivially true.

Suppose some addend $f(x)$ of $\phi(x)$ satisfies $\phi_i(x) \ll^* 1$.
It is generally the case that differentiating the terms of $\phi(x)$ save powers of the $C_j$, i.e. such an $f(x)$ will then satisfy
\[ \paren{\prod_j C_j^{m_j} \frac{\partial^{m_j}}{\partial x_j^{m_j}}} f(x) \ll_m \abs{C}^{\epsilon \abs{m}}, \]
and hence the same for $\e{f(x)}$.
We may then include such a factor in the weight function $g(x,C)$ without altering our assumptions, and so we will.
Similarly, for $C_j > 1$, we can write, e.g.
\[ \frac{x_j}{\sqrt{1+x_j^2}} = \sgn(x_j)+f(x_j), \qquad f(x_j) = \frac{-\sgn(x_j)}{\sqrt{1+x_j^2}\paren{\abs{x_j}+\sqrt{1+x_j^2}}} \]
where now $f(x_j)$ is smooth on the support of $g$ (which excludes $x_j=0$) with
\[ C_j^m \frac{d^m}{dx_j^m} f(x_j) \ll_m C_j^{-2}, \]
and this will allow us to further simplify the phase function.
The goal at each step is to reduce the number of non-negligible terms in $\phi(x)$.

We start with a simple case, which can be done for all $GL(n)$ and we work this out by hand.

\subsection{Interchange of integrals for $w_{n,1}$}

For the Weyl element $w_{n,1}$, naming the coordinate at position $1,i+1$ of the $x$-matrix $x_i$ and the same for $C_i$, the phase is
\begin{align*}
	\phi(x) =& -t_1 x_1+y^*_n x_n-\sum_{i=1}^{n-1} \frac{y_i x_i x_{i+1}}{\sqrt{1+x_i^2}},
\end{align*}
where
\begin{align*}
	y^*_i :=& \frac{\sqrt{1+x_{i+1}^2}}{\sqrt{1+x_i^2}}, \qquad i=1,\ldots,n-1, \\
	y^*_n :=& \frac{y_n}{\sqrt{1+x_2^2}\cdots\sqrt{1+x_{n-1}^2}(1+x_n^2)}.
\end{align*}
The middle term in the phase is $\ll 1$, so we move it to the weight function and replace the phase with
\[ \phi(x) = -t_1 x_1-\sum_{i=1}^{n-1} \frac{y_i x_i x_{i+1}}{\sqrt{1+x_i^2}}. \]
The conditions $y^*_i \ll^* 1$ imply
\begin{align}
\label{eq:wn1WhittBd}
	C_{i+1} \ll^* C_i
\end{align}
for $i = 1,\ldots, n-1$.

From \cref{cor:BKYLemma}, if it were the case that $C_n \ggg^* 1$, we would conclude the bound \eqref{eq:IoIBound} holds unless $C_{n-1}=1$, but this violates the assumption \eqref{eq:wn1WhittBd}.
Assuming $C_n \ll^* 1$, we may remove the $i=n-1$ term to the weight function and apply the same reasoning inductively to conclude the bound holds unless all $C_i \ll^* 1$, in which case the bound is trivially true.

\subsection{The algorithm}
\label{sect:algorithm}
For the remaining cases on $GL(4)$, we apply a computer algebra package.
The algorithm is quite simple and is given in Mathematica code in \cref{sect:algorithmCode}; most of the code presented there is actually for converting human-readable formulas into the format required for the algorithm.
We note that the code uses the coordinates $x_{i,j}$ for the entry at position $(i,j)$ in the $x$-matrix, rather than the coordinates \eqref{eq:GwlCoordsX} that we use in the paper.

To start, we take some subset $\mathcal{S}$ of the indices $j$ of the coordinates of $\wbar{U}_w$ and assume that $C_j=1$ for all $j \in \mathcal{S}$ while $C_j > 1$ for all $j \notin \mathcal{S}$.
We will repeat the following process for all possible choices of $\mathcal{S}$; the case where $\mathcal{S}$ contains all of the indices is trivial, so we exclude it.
The goal is to build a Boolean expression which minimally describes the every non-negligible case resulting from applying \cref{cor:BKYLemma} to the $x_j$ derivative for every coordinate of $\wbar{U}_w$ (including those of $\mathcal{S}$).

All of the constraints may be viewed as linear inequalities (sometimes equalities) on the logarithms of the $C_j$, up to a summand $\BigO{\epsilon \abs{C}}$; e.g. something of the form
\[ \frac{\sqrt{1+x_1^2}}{\sqrt{1+x_3^2}} \ll^* 1 \]
translates to
\[ \log C_1 \le \log C_3+\BigO{\epsilon \abs{C}}. \]
The term $+\BigO{\epsilon \abs{C}}$ is assumed on (and dropped from) all of the inequalities in the algorithm.

A complication for those indices $j \in \mathcal{S}$ is that $x_j$ and $1-x_j^2$ may fail to be $\asymp C_j$ or $\asymp C_j^2$, respectively, when considering the support of $g(x,C)$ at $C_j=1$; we use the phrase ``might be small'' to describe terms which contain such factors (but only those with $j \in \mathcal{S}$).
For an expression, whether it might be small or not, we compute its ``potential size'' by assuming $x_j \asymp C_j$ and $1-x_j^2 \asymp C_j^2$.

For each choice of $\mathcal{S}$, we iteratively build the Boolean expression from the initial constraints $y_i^* \ll^* 1$ and $C_j \ge 1$ (a conjunction of inequalities).
We first write the phase $\phi(x)$ as a sum of products (we refer to the products as the ``terms'' of $\phi(x)$ and the same for its derivatives) and for each coordinate $x_j$ we apply \cref{cor:BKYLemma} as follows:
In the derivative $C_j \frac{\partial}{\partial x_j} \phi(x)$, either all (a conjunction) of the terms are small, i.e. $\ll^* 1$, or (a disjunction) one of the terms must have the largest potential size and this must be $\ggg^* 1$ (a conjunction).
If the term of the largest potential size might be small, we may draw no further conclusions from \cref{cor:BKYLemma}; otherwise, there must be a second term of the same size (a disjunction on which term), and we may assume that this term is also not one which might be small (since the new term would then have the largest potential size).
The final Boolean expression is then the conjunction of the initial constraints and these Boolean expressions resulting from the conclusions of \cref{cor:BKYLemma} for each derivative.

We then rely on the computer algebra package to reduce the Boolean expression.
The code in \cref{sect:algorithmCode} optimizes the Boolean reduction process by starting with the simplest derivative and working up to the most complex and outputs only the non-trivial cases.

We remark that a minor tweak to the algorithm -- removing the initial conditions $y^*_i \ll^* 1$ and dropping $\psi_I(x^*)$ from the phase -- also proves a strong form (rapid convergence) of the Jacquet-Whittaker Direct Continuation Conjecture of \cite[Sect. 10]{GLnI} for $GL(4)$.
Finally, the algorithm can obviously be applied to $GL(5)$ and higher, but on the author's PC, Mathematica largely fails to simplify the resulting Boolean expression and for the few Weyl elements it succeeds, the expression does not reduce to a small number of cases (except for $w_{n,1}$).

\subsection{Interchange of integrals for $w_{22}$}
The phase is
\begin{align*}
	\phi(x) =& -t_2 x_3-\frac{x_3 x_5}{\sqrt{1+x_3^2}}+\frac{x_2 x_3}{\sqrt{1+x_3^2}}-\frac{y_1^* x_2 x_4}{\sqrt{1+x_4^2}}+\frac{y_3^* x_4 x_5}{\sqrt{1+x_4^2}},
\end{align*}
where
\begin{align*}
	y_1^* :=& \frac{\sqrt{1+x_4^2}\sqrt{1+x_5^2}}{\sqrt{1+x_2^2}\sqrt{1+x_3^2}}, \\
	y_2^* :=& \frac{y_2}{\sqrt{1+x_2^2}\paren{1+x_4^2}\sqrt{1+x_5^2}}, \\
	y_3^* :=& \frac{\sqrt{1+x_2^2}\sqrt{1+x_4^2}}{\sqrt{1+x_3^2}\sqrt{1+x_5^2}},
\end{align*}
and we have removed the term
\begin{align}
\label{eq:w22NegligiblePhase}
	y_2^* x_4 =& \frac{y_2 x_4}{\sqrt{1+x_2^2} \paren{1+x_4^2} \sqrt{1+x_5^2}}
\end{align}
to the weight function.

The output of the algorithm may be summarized in two cases:
\begin{enumerate}[label=\roman*.]
\item When
\[ C_2 \asymp C_4 \asymp C_3 \asymp C_5 \ggg^* 1, \]
we may take
\begin{align*}
	\phi(x) =& -t_2 x_3-\sgn(x_3) x_5+\sgn(x_3) x_2-\frac{\sgn(x_2) x_4 \abs{x_5}}{\abs{x_3}}+\frac{\sgn(x_5) \abs{x_2} x_4}{\abs{x_3}}.
\end{align*}
Considering the $x_4$ integral, we see that it is negligible unless $-\sgn(x_2) \abs{x_5}+\sgn(x_5) \abs{x_2} \ll^* C_2^{-1}$, and we substitute $x_2 \mapsto x_5+u/C_2$.
The $u$ integral is essentially non-oscillatory and we may take
\begin{align*}
	\phi(x) =& -t_2 x_3,
\end{align*}
so that the $x_3$ integral is negligible.

\item When
\[ C_4, C_3 \ll^* 1, \qquad C_2 \asymp C_5 \ggg^* 1, \]
we may take
\begin{align*}
	\phi(x) =& -\frac{x_3 x_5}{\sqrt{1+x_3^2}}+\frac{x_2 x_3}{\sqrt{1+x_3^2}}-\frac{\sgn(x_2) x_4 \abs{x_5}}{\sqrt{1+x_3^2}}+\frac{\sgn(x_5) x_4 \abs{x_2}}{\sqrt{1+x_3^2}}.
\end{align*}

If we look at the $x_2$ integral, we see it is negligible unless $\sgn(x_2) x_3 + \sgn(x_5) x_4 \ll^* C_2^{-1}$.
Substitute $x_4 \mapsto -\sgn(x_2 x_5) x_3+u/C_2$, then again the $u$ integral is essentially non-oscillatory and we may take
\begin{align*}
	\phi(x) =& 0.
\end{align*}
That is, we appear to have exhausted the oscillation.

\end{enumerate}

In the latter case, there doesn't appear to be enough oscillation for the \eqref{eq:IoIBound} or even smoothness in $t$, but we have saved a factor $C_2$ over the trivial bound on the integral of \eqref{eq:IoIBound}, which is enough for the dyadic partition of unity to converge absolutely in a tube domain around $\Re(\tilde{\mu})=0$.
Hence \eqref{eq:KwNaiveDef} converges conditionally, but not rapidly, there and the resulting function is clearly differentiable in $y_2$ (the only term of the phase involving $y_2$ is \eqref{eq:w22NegligiblePhase}), since the derivatives have improved convergence.
One can replace the use of Shalika's local multiplicity one theorem in \cite[Prop. 1]{GLnI} with the solution of the differential equations in \cref{sect:w22DEs} (and the obvious polynomial bound in $y_2$) to achieve the same result.
The one place we lose is the Analytic Continuation Conjecture of \cite[Sect. 4]{GLnI} must be proved by analyzing the functional equations and the solutions of the differential equations in the case $\Lambda \ne 0$.

Please note that we have not \emph{disproved} the Strong Interchange of Integrals Conjecture for this Weyl element, as a deeper analysis might still succeed, but we have given strong evidence that it is false, though we show it remains true for the other Weyl elements.

\subsection{Interchange of integrals for $w_{121}$}
The phase is
\begin{align*}
	\phi(x) =& -t_1 x_1-\frac{t_3 x_2 x_4}{\sqrt{1+x_2^2}}-\frac{t_3 \sqrt{1+x_4^2} x_6}{\sqrt{1+x_2^2}}-\frac{x_1 x_2}{\sqrt{1+x_1^2}}+y_1^* x_5+\frac{y_2^* x_5 x_6}{\sqrt{1+x_5^2}}+\frac{y_3^* x_6 x_4}{\sqrt{1+x_6^2}},
\end{align*}
where
\begin{align*}
	y_1^* :=& \frac{y_1 \sqrt{1+x_1^2}}{\sqrt{1+x_4^2}\paren{1+x_5^2}\sqrt{1+x_6^2}}, \\
	y_2^* :=& \frac{\sqrt{1+x_2^2}\sqrt{1+x_5^2}}{\sqrt{1+x_1^2}\sqrt{1+x_6^2}}, \\
	y_3^* :=& \frac{y_3 \sqrt{1+x_6^2}}{\sqrt{1+x_1^2}\paren{1+x_2^2}\sqrt{1+x_4^2}},
\end{align*}
and we have removed the term
\[ \frac{x_2 y_3}{\sqrt{1+x_1^2} \paren{1+x_2^2}} \]
to the weight function.

The algorithm yields no non-trivial cases.

\subsection{Interchange of integrals for $w_{211}$}
The phase is
\begin{align*}
	\phi(x) =& -t_1 x_1-\frac{t_2 x_1 x_2}{\sqrt{1+x_1^2}}-\frac{t_2 \sqrt{1+x_2^2} x_3}{\sqrt{1+x_1^2}}-\frac{x_2 x_4}{\sqrt{1+x_2^2}}+\frac{y_3 x_2 x_3}{\paren{1+x_1^2} \sqrt{1+x_2^2}} \\
	& \qquad -\frac{y_1^* x_3 x_5}{\sqrt{1+x_5^2}}+y_2^* x_5+\frac{y_3^* x_5 x_4}{\sqrt{1+x_5^2}},
\end{align*}
where
\begin{align*}
	y_1^* :=& \frac{\sqrt{1+x_4^2} \sqrt{1+x_5^2}}{\sqrt{1+x_2^2}\sqrt{1+x_3^2}}, \\
	y_2^* :=& \frac{y_2 \sqrt{1+x_1^2}}{\sqrt{1+x_4^2}\sqrt{1+x_3^2}\paren{1+x_5^2}}, \\
	y_3^* :=& \frac{y_3 \sqrt{1+x_3^2}\sqrt{1+x_5^2}}{\paren{1+x_1^2}\sqrt{1+x_2^2}\sqrt{1+x_4^2}},
\end{align*}
and we have removed the term
\[ \frac{x_1 y_3}{1+x_1^2} \]
to the weight function.

The output of the algorithm may be summarized as
\[ C_1, C_2, C_5 \ll^* 1, \qquad C_4 \asymp C_3 \ggg^* 1. \]
Applying this to the phase, we may take
\begin{align*}
	\phi(x) =& -\frac{t_2 \sqrt{1+x_2^2} x_3}{\sqrt{1+x_1^2}}-\frac{x_2 x_4}{\sqrt{1+x_2^2}}+\frac{y_3 x_2 x_3}{\paren{1+x_1^2} \sqrt{1+x_2^2}} \\
	& \qquad -\frac{\sgn(x_3) \abs{x_4} x_5}{\sqrt{1+x_2^2}}+\frac{\sgn(x_4) y_3 \abs{x_3} x_5}{\paren{1+x_1^2}\sqrt{1+x_2^2}}.
\end{align*}

If we look at the $x_4$ integral, this is small unless $\sgn(x_4) x_2+\sgn(x_3) x_5 \ll^* C_4^{-1}$.
Substitute $x_2 \mapsto -\sgn(x_4 x_3) x_5+u/C_4$, then the $u$ integral is essentially non-oscillatory and we may take
\begin{align*}
	\phi(x) =& -\frac{t_2 \sqrt{1+x_5^2} x_3}{\sqrt{1+x_1^2}},
\end{align*}
but then the $x_3$ integral is negligible.

\subsection{Interchange of integrals for $w_{1111}$}
The phase is
\begin{align*}
	\phi(x) =& -t_1 x_1-\frac{t_2 x_1 x_2}{\sqrt{1+x_1^2}}-\frac{t_2 \sqrt{1+x_2^2} x_3}{\sqrt{1+x_1^2}}-\frac{t_3 x_2 x_4}{\sqrt{1+x_2^2}}-\frac{t_3 \sqrt{1+x_4^2} x_3 x_5}{\sqrt{1+x_2^2} \sqrt{1+x_3^2}} \\
	& \qquad-\frac{t_3 \sqrt{1+x_4^2} \sqrt{1+x_5^2} x_6}{\sqrt{1+x_2^2} \sqrt{1+x_3^2}} + y_1^* x_6+\frac{y_2^* x_6 x_5}{\sqrt{1+x_6^2}}+\frac{y_3^* x_5 x_4}{\sqrt{1+x_5^2}} \\
	& \qquad+\frac{y_2 \sqrt{1+x_1^2} x_3}{\sqrt{1+x_2^2} \paren{1+x_3^2}}+\frac{y_3 x_2 x_3}{\paren{1+x_1^2} \sqrt{1+x_2^2}},
\end{align*}
where
\begin{align*}
	y_1^* :=& \frac{y_1 \sqrt{1+x_2^2} \sqrt{1+x_3^2}}{\sqrt{1+x_4^2}\sqrt{1+x_5^2}\paren{1+x_6^2}}, \\
	y_2^* :=& \frac{y_2 \sqrt{1+x_1^2} \sqrt{1+x_6^2}}{\sqrt{1+x_2^2}\paren{1+x_3^2}\sqrt{1+x_5^2}}, \\
	y_3^* :=& \frac{y_3 \sqrt{1+x_3^2} \sqrt{1+x_5^2}}{\paren{1+x_1^2}\sqrt{1+x_2^2}\sqrt{1+x_4^2}},
\end{align*}
and we have removed the term
\[ \frac{x_1 y_3}{1+x_1^2} \]
to the weight function.

The output of the algorithm may be summarized in two cases:
\begin{enumerate}[label=\roman*.]
\item When
\[ C_1, C_2, C_5, C_6 \ll^* 1, \qquad C_4 \asymp C_3 \ggg^* 1, \]
we may take
\begin{align*}
	\phi(x) =& -\frac{t_2 \sqrt{1+x_2^2} x_3}{\sqrt{1+x_1^2}}-\frac{t_3 x_2 x_4}{\sqrt{1+x_2^2}}-\frac{\sgn(x_3) t_3 \abs{x_4} x_5}{\sqrt{1+x_2^2}} \\
	& \qquad+\frac{\sgn(x_4) y_3 \abs{x_3} x_5}{\paren{1+x_1^2}\sqrt{1+x_2^2}}+\frac{y_3 x_2 x_3}{\paren{1+x_1^2} \sqrt{1+x_2^2}}.
\end{align*}
This case is identical to the main case of $w_{211}$ with $y_1 \mapsto t_3$.

\item When
\[ C_2, C_4, C_3, C_6 \ll^* 1, \qquad C_1 \asymp C_5 \ggg^* 1, \]
we may take
\begin{align*}
	\phi(x) =& -t_1 x_1-\frac{t_3 \sqrt{1+x_4^2} x_3 x_5}{\sqrt{1+x_2^2} \sqrt{1+x_3^2}}-\frac{t_3 \abs{x_5} \sqrt{1+x_4^2} x_6}{\sqrt{1+x_2^2} \sqrt{1+x_3^2}} \\
	& \qquad+\frac{\sgn(x_5) y_2 \abs{x_1} x_6}{\sqrt{1+x_2^2}\paren{1+x_3^2}}+\frac{y_2 \abs{x_1} x_3}{\sqrt{1+x_2^2} \paren{1+x_3^2}}.
\end{align*}
This case is similar to the previous one:
Considering the $x_5$ integral, we see that it is negligible unless $\sgn(x_5)x_3+x_6 \ll^* C_5^{-1}$, and we substitute $x_3 \mapsto -\sgn(x_5)x_6+u/C_5$.
Again, the $u$ integral is essentially non-oscillatory and we may take
\begin{align*}
	\phi(x) =& -t_1 x_1,
\end{align*}
so that the $x_1$ integral is negligible.
\end{enumerate}

\section{Tidying Up}
\label{sect:hiccups}
We hope the combination of the following methods would close the gaps in this differential equations-and-power series method for arbitrary Weyl elements on $GL(n)$.
That is, if there fails to be sufficient convergence, then there will be a corresponding $y$-coordinate for which smoothness and power-series-ness is obvious, while for the remaining coordinates, we can power-series expand by the rapid decay.

\subsection{The $w_{22}$ hiccup}
\label{sect:w22hiccup}
In this section, we prove the Analytic Continuation Conjecture for $K_{22}$.
To that end, we reintroduce the notation $K_{22}(y,\Lambda,\mu,\delta)$.
The proof of the Weak Interchange of Integrals implies that, for each $\Lambda$, $K_{22}(y,\Lambda,\mu,\delta)$ is holomorphic in a tube domain containing $i\mathfrak{a}^*_0(\Lambda)$.
We also know that $K_{22}(y,\Lambda,\mu,\delta)$ satisfies the differential equations of \cref{sect:w22DEs}.
The Asymptotics Theorem tells us that $K_{22}(y,0,\mu,\delta)$ is a particular linear combination of the power series solutions, which gives it an expression which is meromorphic on all of $\mu\in\C^4$.
We need to show then that $K_{22}(y,\Lambda,\mu,\delta)$ agrees with this extension of $K_{22}(y,0,\mu,\delta)$ to $i\mathfrak{a}^*_0(\Lambda)$.

Note that we have two cases to consider:
The first case, call it Case I, is $\Lambda=(\frac{k_1-1}{2},-\frac{k_1-1}{2},0,0)$, $\mu-\Lambda=(r_1,r_1,r_2,r_3)$ with $2 \le k_1 \in \N$, $r\in\C^3$, $2r_1+r_2+r_3=0$ and $\delta_1+\delta_2\equiv k_1 \pmod{2}$.
The second case, call it Case II, is $\Lambda=(\frac{k_1-1}{2},-\frac{k_1-1}{2},\frac{k_2-1}{2},-\frac{k_2-1}{2})$, $\mu-\Lambda=(r_1,r_1,-r_1,-r_1)$ with $2 \le k_2 \le k_1 \in \N$, $r_1\in\C$, $\delta_1+\delta_2\equiv k_1 \pmod{2}$, $\delta_3+\delta_4\equiv k_2 \pmod{2}$.

We have only considered the differential equations for $K_{22}$ in the case $\Lambda=0$, but their extension to $\Lambda\ne0$ follows the typical path for the method of Frobenius.
A set of representatives for $W/W_w$ is given by $\set{I,w_{(1\,3)},w_{(2\,3)},w_{(1\,4)},w_{(2\,4)},w_{(1\,3)(2\,4)}}$.
In Case I, we have $J_{22}(y,\mu^{w_{(1\,3)}})=\sgn(y_2)^{k_1-1} J_{22}(y,\mu^{w_{(2\,3)}})$ and $J_{22}(y,\mu^{w_{(1\,4)}})=\sgn(y_2)^{k_1-1} J_{22}(y,\mu^{w_{(2\,4)}})$, and we can see the remaining two solutions are
\begin{align*}
	Y_{22,1}(y,\mu) :=& \lim_{\mu_1-\mu_2 \to k_1-1} \frac{J_{22}(y,\mu^{w_{(1\,3)}})-\sgn(y_2)^{k_1-1} J_{22}(y,\mu^{w_{(2\,3)}})}{\sin \pi(\mu_1-\mu_2)}, \\
	Y_{22,2}(y,\mu) :=& \lim_{\mu_1-\mu_2 \to k_1-1} \frac{J_{22}(y,\mu^{w_{(1\,4)}})-\sgn(y_2)^{k_1-1} J_{22}(y,\mu^{w_{(2\,4)}})}{\sin \pi(\mu_1-\mu_2)}.
\end{align*}
In Case II, we furthermore have $J_{22}(y,\mu^{w_{(2\,4)}})=\sgn(y_2)^{k_2-1} J_{22}(y,\mu^{w_{(2\,3)}})$ and $Y_{22,2}(y,\mu) = \sgn(y_2)^{k_2-1} Y_{22,1}(y,\mu)$, and the remaining solutions are
\begin{align*}
	Y_{22,3}(y,\mu) :=& \lim_{\mu_3-\mu_4 \to k_2-1} \frac{J_{22}(y,\mu^{w_{(2\,4)}})-\sgn(y_2)^{k_2-1} J_{22}(y,\mu^{w_{(2\,3)}})}{\sin \pi(\mu_3-\mu_4)}, \\
	Y_{22,4}(y,\mu) :=& \lim_{\mu_3-\mu_4 \to k_2-1} \frac{Y_{22,2}(y,\mu)-\sgn(y_2)^{k_2-1} Y_{22,1}(y,\mu)}{\sin \pi(\mu_3-\mu_4)}.
\end{align*}
One could conceive that further degeneracy of the $J_{22}$ functions occurs at $d_1=d_2$ or $d_1=2d_2$, but it does not.

Note: If the reader is unsatisfied with the limit notation above, set, e.g., $\mu_z := (\frac{z}{2}+r_1,-\frac{z}{2}+r_1,r_2,r_3)$ and we mean
\begin{align*}
	&Y_{22,1}(y,(\tfrac{k_1-1}{2}+r_1,-\tfrac{k_1-1}{2}+r_1,r_2,r_3)) \\
	& \qquad = (-1)^{k_1-1} \left. \frac{\partial}{\partial z} \paren{J_{22}(y,\mu_z^{w_{(1\,3)}})-\sgn(y_2)^{k_1-1} J_{22}(y,\mu_z^{w_{(2\,3)}})} \right|_{z=k_1-1}.
\end{align*}

So we have verified the space of solutions to the $w_{22}$ differential equation is spanned by the Frobenius series solutions.
The analysis of the Mellin-Barnes integrals for the Whittaker function in \cite[Sect. 9.2]{GLnI} shows that the Whittaker functions are also linear combinations Frobenius series with leading terms $I_{\mu^w,\delta^w}$.
In particular, this holds at $\Re(\mu)=0$ and for the Whittaker function, the higher-weight case $\Re(\mu)=\Lambda \ne 0$ is, in fact, obtained by analytic continuation.
In combination with \cite[eq. (52)]{GLnI}, this implies the result, as in the proof of the Asymptotics Theorem in \cite[Sect. 11]{GLnI}.
The reader should be concerned about convergence of the $\wbar{U}_{22}(\R)$ integral, but actually, we may safely expand both $\e{y_2^* x_4}$ and the Whittaker function (considered as a function of $y_2^*$ alone) into power/Frobenius series because the higher terms actually make the $\wbar{U}_{22}(\R)$ integral converge faster and we only need the first term of each series.

\subsection{The $w_{211}$ hiccup}
\label{sect:w211hiccup}

Though it would be nice to prove something like Hashizume's results \cite{Hashi} for Bessel functions of arbitrary Weyl elements on $GL(n)$, the author was unable to force through any ellipticity arguments, and we now give a concrete method through which one may argue that the Bessel function lies in the span of the Frobenius series solutions for the Weyl element $w_{211}$.

First, note that $K_{211}(y,\mu,\delta)$ is a scalar-valued function satisfying the differential equations \eqref{eq:211DEs1} and \eqref{eq:211DEs2}, and we can only talk about the Bessel functions on representations/at $K$-types for which the Whittaker function is not identically zero.
If (the matrix-valued Whittaker function) $W_\sigma(t,\mu,\delta)$ is not identically zero take $i,j$ so that the entry $W_{\sigma,i,j}(t,\mu,\delta)$ is not identically zero.
Then the Mellin-Barnes integral \cite[eq. (49)]{GLnI} implies a rapidly convergent Frobenius series expansion which implies in turn holomorphy in $t$, so we can take $t$ such that the $Y$-coordinates of $w_{211} t w_{211}^{-1}$ are arbitrarily small and still $W_{\sigma,i,j}(t,\mu,\delta) \ne 0$, and we get
\begin{align*}
	K_{211}(y,\mu,\delta) =& \frac{I_{0,\delta}(yw_{211}tw_{211})}{W_{\sigma,i,j}(t,\mu,\delta)} \sum_{k=1}^4 \int_{\Re(s)=\epsilon} \what{W}_{\sigma,i,k}(s,\mu,\delta) F_{\sigma,k,j}(yw_{211}tw_{211},t,s,\mu,\delta) \\
	& \qquad \times \prod_{\ell=1}^3 (yw_{211}tw_{211})_\ell^{-s_\ell} \frac{ds}{(2\pi i)^3},
\end{align*}
where
\[ F_{\sigma,k,j}(y,t,s,\mu,\delta) = \int_{\wbar{U}_{211}(\R)} \wbar{\psi_t(u)} \psi_{yw_{211}tw_{211}^{-1}}(x^*) \sigma_{k,j}(k^*) I_{0,\delta}(y^*) \prod_{\ell=1}^3 (y^*_\ell)^{-s_\ell} du, \quad x^* y^* k^* = wu, \]
formally, but also actually if we were to hold the $u$ integral outside and include the $\mu$ integral between the $u$ and $s$ integrals.

We will derive a Mellin expansion of $F_{\sigma,k,j}$, justified by the absolute convergence of the original $u-\mu-s$ integral and the rapid convergence of \eqref{eq:PsiThetaInvMellin} for $\theta < \frac{\pi}{2}$, that can be used to shift the contours (in the Mellin expansion of $K_{211}$) to $-\infty$.
Without worrying about the particular locations of the poles of $\what{W}_{\sigma,i,k}$ and $\what{F}_{\sigma,k,j}$, it's enough to know that they occur along finitely many horizontal lines in each coordinate, so the contour shifting produces some linear combination of Frobenius series.

Now $\sigma_{k,j}$ is a trigonometric polynomial in the hyper-spherical coordinates (see \cite[Sect. 6.1]{GLnI}), so after applying the Iwasawa decomposition of \cref{sect:Iwasawa} and expanding, we see that $F_{\sigma,k,j}(y,t,s,\mu,\delta)$ is a linear combination of terms
\begin{align*}
	& F_m^*(y,t,u) :=\\
	& \int_{\R^5} e\Biggl(-t_1 x_1-\frac{t_2 x_1 x_2}{\sqrt{1+x_1^2}}-\frac{t_2 x_3 \sqrt{1+x_2^2}}{\sqrt{1+x_1^2}}-\frac{y_1 x_2 x_4}{\sqrt{1+x_2^2}}+\frac{y_3 x_2 x_3}{\paren{1+x_1^2} \sqrt{1+x_2^2}} -\frac{y_1 x_3 x_5 \sqrt{1+x_4^2}}{\sqrt{1+x_2^2}\sqrt{1+x_3^2}}\\
	& \qquad +\frac{y_2 x_5 \sqrt{1+x_1^2}}{\sqrt{1+x_4^2}\sqrt{1+x_3^2}\paren{1+x_5^2}} +\frac{y_3 x_4 x_5 \sqrt{1+x_3^2}}{\paren{1+x_1^2}\sqrt{1+x_2^2}\sqrt{1+x_4^2}}+\frac{x_1 y_3}{1+x_1^2}\Biggr) \\
	& \qquad \times \prod_{\ell=1}^5 \paren{1+x_\ell^2}^{\frac{-1+u_\ell}{2}}\paren[m_\ell]{\frac{1+i x_\ell}{\sqrt{1+x_\ell^2}}} dx,
\end{align*}
for some $m \in \Z^5$, where
\begin{align*}
	u_1=& -s_2+2s_3-\mu_3+\mu_4,& u_2=&s_1+s_3-\mu_1+\mu_4,& u_3 =&s_1+s_2-s_3-\mu_1+\mu_3, \\
	u_4=& -s_1+s_2+s_3-\mu_2+\mu_4,& u_5=&-s_1+2s_2-s_3-\mu_2+\mu_3.
\end{align*}

Now inside $F_m^*$, we Mellin expand all nine exponential terms and evaluate the $x$ integrals using \cite[eqs. (2.26)-(2.29)]{HWI}
\begin{align*}
	&\int_{-\infty}^{\infty} \chi_{a-1}^\delta(x) \paren{1+x^2}^{-\frac{b+a}{2}} \paren[m]{\frac{1+i x}{\sqrt{1+x^2}}}dx \\
	&=(-\sgn(m) i)^\delta \sum_{j=0}^{(\abs{m}-\delta)/2} \binom{\abs{m}}{2j+\delta} (-1)^j B\paren{\frac{a+2j+\delta}{2},\frac{b+\abs{m}-2j-\delta}{2}},
\end{align*}
we see that $F_m^*$ is a linear combination of functions of the form
\begin{align*}
	F_{m,\eta,\eta'}^\dagger(y,t,u) :=& \int_{\Re(v)=\epsilon} \chi_{-v_1}^{\eta_1'}(y_2) \chi_{-v_2-v_3-v_4}^{\eta_2'+\eta_3'+\eta_4'}(y_3) F_{m,\eta}^\sharp(y,t,q)) \prod_{j=1}^4 G_{\eta_j'}(v_j) \frac{dv}{(2\pi i)^4},
\end{align*}
\begin{align*}
	F_\eta^\sharp(y,t,q) :=& \int_{\Re(r)=\epsilon} \chi_{-r_1}^{\eta_1}(t_1) \chi_{-r_2-r_3}^{\eta_2+\eta_3}(t_2) \chi_{-r_4-r_5}^{\eta_4+\eta_5}(y_1) \prod_{j=1}^5 G_{\eta_j}(r_j) B\paren{\frac{q_1-r_1-r_2}{2},\frac{q_2+r_1-r_3}{2}} \\
	& \qquad B\paren{\frac{q_3-r_2-r_4}{2},\frac{q_4+r_2+r_3-r_5}{2}} B\paren{\frac{q_5-r_3-r_5}{2},\frac{q_6+r_3}{2}} \\
	& \qquad B\paren{\frac{q_7-r_4}{2},\frac{q_8+r_4+r_5}{2}} B\paren{\frac{q_9-r_5}{2},\frac{q_{10}+r_5}{2}} \frac{dr}{(2\pi i)^5}
\end{align*}
where $\eta \in \Z^5,\eta'\in\Z^4$,
\begin{align*}
	q_1 =& 1+m_{11}-v_2,& q_2=& m_{12}+v_1-v_2-2v_3-2v_4, \\
	q_3 =& 1+m_{21}-v_3& q_4=& m_{22}-u_2-v_4,\\
	q_5 =& 1+m_{31}-v_3& q_6=& m_{32}-u_3-v_1+v_3+v_4,\\
	q_7 =& 1+m_{41}-v_4& q_8=& m_{42}-u_4-v_1,\\
	q_9 =& 1+m_{51}-v_1-v_4& q_{10}=& m_{52}-u_5-v_1+v_4,
\end{align*}
and each $m_{\ell,1},m_{\ell,2}\ge0$, $m_{\ell,1}+m_{\ell,2}=\abs{m_\ell}$.

Recall the terminology following \eqref{eq:Stirling}.
Away from its poles, $F_\eta^\sharp$ has exponential convergence due to the exponential decay of the beta functions $B(a-r,b+r)$ in the region $\abs{\Im(r-\frac{a-b}{2})} < \abs{\Im(\frac{a+b}{2})}$.
Shifting the $r$ contours negative increases the power of $y_1$ and localizes the integral near $\Im(r)\approx 0$, so that $F_\eta^\sharp$ is essentially bounded by the polynomial part of the beta functions at $\Im(r)=0$, plus the residues picked up from the beta functions, which are also polynomial in $q$ and for which an analysis similar to the following applies.

For $\Re(r)$ highly negative and $\Im(r) \approx 0$, the polynomial parts of the beta functions want to localize the $v$ integral near $\Im(q_{2j})=0$, $j=1,\ldots,5$, while the odd indexed $q$ variables occur to a high power.
Treating $u$ as fixed for the moment, the combined effect, including the decay of the factors $G_{\eta_j'}(v_j)$ for $\Re(v_j) < \frac{1}{2}$ is to localize near $\Im(v) \approx 0$.

Lastly, the $s$ integrals have exponential convergence, coming from the exponential decay of $\what{W}_{\sigma,i,k}$ in each $\abs{\Im(s_j)}\gg1$, treating $\mu$ as fixed, so our earlier assumption that $s$ is essentially fixed is valid (as there is no exponential growth coming from $F_{m,\eta,\eta'}^\dagger$).

So we may shift the $s$, $v$ and $r$ contours to the left.
The factorial parts of the gamma functions in $F_\eta^\sharp$ and $F_{m,\eta,\eta'}^\dagger$ are either balanced (some of the beta functions) or induce factorial-type decay (the other beta functions and all of the $G_\delta$ functions), and the factorial part of $\what{W}_{\sigma,i,k}$ always induces factorial-type decay.
Hence we may shift the $s$, $v$ and $r$ contours to $-\infty$ and conclude that $K_{211}(y,\mu,\delta)$ is a sum of Frobenius series.

\appendix
\section{Mathematica Code for The Interchange of Integrals Algorithm}
\label{sect:algorithmCode}
Here we provide the code to implement the algorithm of \cref{sect:algorithm} as described in the text; the code has been formatted for direct copy-and-paste, aside from the page breaks.

The functions xforms, formIndex, logFormExponent, exprToIndices, phiToIndices, indicesToPhi are all concerned with converting human-readable input to the format required by the algorithm and vice versa.
The algorithm itself is comprised of initialCase, applyBKYLemma, listAllCases with supporting functions logCoefs, dervLogCoefs, exprDerv, doDerv, intyStarToLogs, yStarToLogs, hasIntersection, intMightBeSmall, intMightBeSmallDerv, mightBeSmall, and smallsSub.
The function doReduceAll optimizes the reduction of the Boolean expression with helpers indexFrequencies, doSimplify, toOrList and doReduce.
Finally, DoTest packages the algorithm into a simple format with assistance from properSubsets, testAllCases, printAllCases and makeVars.\\[0.15in]

\begingroup
\ttfamily
\setlength{\parindent}{0pt}

(* forms of the factors of the terms in the phase *)\\
xforms[x\char`_] := \{x/(1 + x\^{}2), 1/(1 + x\^{}2), Sqrt[1 + x\^{}2], \mlbb 1/Sqrt[1 + x\^{}2], x/Sqrt[1 + x\^{}2], x, 1\};\\
(* exponent of Cij in potential size of xforms *)\\
logCoefs = \{-1, -2, 1, -1, 0, 1, 0\}; \\
 (* exponent of Cij in potential size of the derivative of xforms, times Cij *)\\
dervLogCoefs = \{-2, -3, 0, -2, -3, 0, 0\} + 1;\\
(* differentiate a term and take its logarithm, asymptotically in the Cij *)\\
exprDerv[expr\char`_, var\char`_, logvars\char`_] := 
 Module[\{i\}, 
  If[expr[[var]] == 7, 0, \mlb
   Sum[If[i == var, dervLogCoefs[[expr[[i]]]], 
      logCoefs[[expr[[i]]]]] logvars[[i]], \mlbb \{i, 1, Length[logvars]\}]]]\\
doDerv[exprs\char`_, var\char`_, logvars\char`_] := 
 Map[exprDerv[\#, var, logvars] \&, exprs]\\
(* convert a term of phi to indices *)\\
formIndex[expr\char`_, var\char`_] := Module[\{forms = xforms[var]\},\mlb
  Piecewise[\{\mlbb
    \{1, ! FreeQ[expr, forms[[1]]]\},\mlbb
    \{2, ! FreeQ[expr, forms[[2]]]\},\mlbb
    \{5, ! FreeQ[expr, forms[[5]]]\},\mlbb
    \{4, ! FreeQ[expr, forms[[4]]]\},\mlbb
    \{3, ! FreeQ[expr, forms[[3]]]\},\mlbb
    \{6, ! FreeQ[expr, forms[[6]]]\},\mlbb
    \{7, True\}\mlbb
    \}]]\\
(* convert (1+xij\^{}2)\^{}(a/2) ~ Cij\^{}a -\textgreater a *)\\
logFormExponent[expr\char`_, x\char`_] :=\mlb
 Piecewise[\{\mlbb
   \{-2, ! FreeQ[expr, 1/(1 + x\^{}2)]\},\mlbb
   \{-1, ! FreeQ[expr, 1/Sqrt[1 + x\^{}2]]\},\mlbb
   \{1, ! FreeQ[expr, Sqrt[1 + x\^{}2]]\},\mlbb
   \{0, True\}\mlbb
   \}]\\
exprToIndices[expr\char`_, vars\char`_] := Map[formIndex[expr, \#] \&, vars]\\
phiToIndices[phi\char`_, vars\char`_, yStars\char`_] := 
 Map[exprToIndices[\#, vars] \&, 
  phi /. \mlb Table[
    Symbol["ys" \textless\textgreater ToString[i]] -\textgreater yStars[[i]], \{i, 1, 
     Length[yStars]\}]]\\
indicesToPhi[phiIndices\char`_, vars\char`_] := 
 Module[\{i, j\}, 
  Table[\mlb Product[xforms[vars[[j]]][[phiIndices[[i, j]]]], \{j, 1, 
     Length[vars]\}], \mlb \{i, 1, Length[phiIndices]\}]]\\
indicesToPhi[phiIndices\char`_, coefs\char`_, vars\char`_] := 
 coefs indicesToPhi[phiIndices, vars]\\
intyStarToLogs[yStar\char`_, vars\char`_, logvars\char`_] := 
 logvars.Map[\mlb logFormExponent[yStar, \#] \&, vars]\\
yStarToLogs[yStar\char`_, vars\char`_, logvars\char`_] := 
 Map[intyStarToLogs[\#, vars, logvars] \&, \mlb yStar]\\
indexFrequencies[phiIndices\char`_] := 
 Module[\{i\}, 
  Table[Length[DeleteCases[\mlb phiIndices[[All, i]], 7]], \{i, 1, 
    Length[phiIndices[[1]]]\}]]\\
properSubsets[set\char`_] := DeleteCases[Subsets[set], set]\\
hasIntersection[A\char`_, B\char`_] := ! AllTrue[A, FreeQ[B, \#] \&]\\
(* MightBeSmall checks if an expression has a factor, e.g. xij, which might be \mlb small when Cij==1; smalls is the list of such Cij *)\\
intMightBeSmall[expr\char`_, smalls\char`_] := 
 hasIntersection[\{1, 5, 6\}, Part[expr, smalls]]\\
intMightBeSmallDerv[expr\char`_, derv\char`_] := 
 hasIntersection[\{1, 2, 3, 4\}, expr[[derv]]]\\
mightBeSmall[expr\char`_, smalls\char`_, derv\char`_] := 
 intMightBeSmall[expr, DeleteCases[smalls, \mlb derv]] || 
  If[FreeQ[smalls, derv], False, intMightBeSmallDerv[expr, derv]]\\
smallsSub[logvars\char`_, smalls\char`_] := Map[logvars[[\#]] -\textgreater 0 \&, smalls]\\
initialCase[logystars\char`_, logvars\char`_, smalls\char`_] :=
 Simplify[\mlb
  AllTrue[logystars, \# \textless= 0 \&] \&\&\mlb AllTrue[logvars, \# \textgreater= 0 \&] /. smallsSub[logvars, smalls]]\\
applyBKYLemma[exprs\char`_, var\char`_, logvars\char`_, smalls\char`_] := Module[\{\mlb
   smallDerv = DeleteCases[doDerv[Select[exprs, \mlbb mightBeSmall[\#, smalls, var] \&], var, logvars], 0],\mlb
   largeDerv = DeleteCases[doDerv[Select[exprs, \mlbb !mightBeSmall[\#, smalls, var] \&], var, logvars], 0],\mlb
   i, j\},\mlb
  Join[\mlbb
    Flatten[Table[largeDerv[[i]] == largeDerv[[j]] \&\& \mlbbb
       AllTrue[Delete[largeDerv, \{\{i\}, \{j\}\}], \# \textless= largeDerv[[i]] \&] \&\&\mlbbb 
        largeDerv[[i]] \textgreater 0, \{i, 1, Length[largeDerv] - 1\}, \mlbbb \{j, i + 1, Length[largeDerv]\}]],\mlbb
    Table[AllTrue[largeDerv, \# \textless= smallDerv[[i]] \&] \&\& smallDerv[[i]] \textgreater 0, \mlbbb \{i, 1, Length[smallDerv]\}],\mlbb 
    \{AllTrue[Join[largeDerv, smallDerv], \# \textless= 0 \&]\}] \mlb/. smallsSub[logvars, smalls]]\\
(* doSimplify, doReduce and doReduceAll are just a little bit smarter simplification \mlbb than the Mathematica default *)\\
doSimplify[expr\char`_, logvars\char`_] := Simplify[expr, Element[logvars, Reals]]\\
listAllCases[exprsList\char`_, logystars\char`_, logvars\char`_, smalls\char`_] := 
 Module[\mlb \{exprOrder = Ordering[indexFrequencies[exprsList]]\},\mlb
  doSimplify[Prepend[Map[applyBKYLemma[exprsList, \#, logvars, smalls] \&, \mlbb exprOrder], initialCase[logystars, logvars, smalls]], logvars]]\\
(* Convert a Boolean expression to sum of products and then to a list *)\\
toOrList[expr\char`_] := 
 Module[\{temp = LogicalExpand[expr]\}, \mlb
  If[Head[temp] === Or, List @@ temp, \{temp\}]]\\
doReduce[cur\char`_, next\char`_, logvars\char`_] := 
 doSimplify[Reduce[Or @@ Map[Reduce, \mlb toOrList[cur \&\& next]]], logvars]\\
doReduceAll[allCases\char`_, logvars\char`_, smalls\char`_] := 
 Module[\{retval = allCases[[1]], i\},\mlb
  Do[retval = doReduce[retval, Or @@ (allCases[[i]]), logvars],\mlbb \{i, 2, Length[allCases]\}];\mlb
  (* Only show non-trivial cases *)\mlb
  doSimplify[retval \&\& (AnyTrue[logvars, \# \textgreater 0 \&] /. \mlbb smallsSub[logvars, smalls]), logvars]]\\
(* Some code to help run the algorithm *)\\
testAllCases[exprsList\char`_, logystars\char`_, logvars\char`_, cvars\char`_] := 
 ParallelTable[\{\mlbb Map[cvars[[\#]] == 1 \&, smalls], \mlbb 
   doReduceAll[listAllCases[exprsList, logystars, logvars, smalls], logvars, \mlbb smalls]\}, \mlb \{smalls, properSubsets[Range[Length[logvars]]]\}, \mlb
  Method -\textgreater "ItemsPerEvaluation" -\textgreater 1]\\
printAllCases[
  exprs\char`_] := (Do[If[!(expr[[2]] === False), Print[expr]],\mlb \{expr, exprs\}]; Length[exprs])\\
makeVars[prefix\char`_, vars\char`_] := Map[Symbol, Map[prefix \textless\textgreater \# \&, vars]]\\
DoTest[phi\char`_, ystar\char`_, vars\char`_] := 
 Module[\mlb\{xvars, logvars, cvars, logystars, PhiA, PhiB\},\mlb
  xvars = makeVars["x", vars];\mlb
  logvars = makeVars["logC", vars];\mlb
  cvars = makeVars["C", vars];\mlb
  logystars = yStarToLogs[ystar, xvars, logvars];\mlb
  PhiA = phiToIndices[List @@ phi, xvars, ystar];\mlb
  Print[Simplify[(List @@ phi)/indicesToPhi[PhiA, xvars]]];\mlb
  Print[AbsoluteTiming[
    PhiB = testAllCases[PhiA, logystars, logvars, cvars];]];\mlb
  Print[printAllCases[PhiB]];\mlb
  PhiB]\\

(* A sample application to the GL(4) long Weyl element *)\\
Phi1111b = DoTest[\mlb
  -t1 x12 - (t2 x12 x13)/Sqrt[1 + x12\^{}2] - (t3 x13 x14)/Sqrt[1 + x13\^{}2] \mlbb - (t2 Sqrt[1 + x13\^{}2] x23)/Sqrt[1 + x12\^{}2] \mlbb - (t3 Sqrt[1 + x14\^{}2] x23 x24)/(Sqrt[1 + x13\^{}2] Sqrt[1 + x23\^{}2]) \mlbb - (t3 Sqrt[1 + x14\^{}2] Sqrt[1 + x24\^{}2] x34)/(Sqrt[1 + x13\^{}2] Sqrt[1 + x23\^{}2]) \mlbb + (Sqrt[1 + x13\^{}2] Sqrt[1 + x23\^{}2] x34 y1)/\mlbbb (Sqrt[1 + x14\^{}2] Sqrt[1 + x24\^{}2] (1 + x34\^{}2)) \mlbb + (Sqrt[1 + x12\^{}2] x23 y2)/(Sqrt[1 + x13\^{}2] (1 + x23\^{}2)) \mlbb + (Sqrt[1 + x12\^{}2] x24 x34 y2)/(Sqrt[1 + x13\^{}2] (1 + x23\^{}2) Sqrt[1 + x24\^{}2]) \mlbb + (x13 x23 y3)/((1 + x12\^{}2) Sqrt[1 + x13\^{}2]) \mlbb + (x14 Sqrt[1 + x23\^{}2] x24 y3)/((1 + x12\^{}2) Sqrt[1 + x13\^{}2] Sqrt[1 + x14\^{}2]),\mlb
  \{(Sqrt[1 + x13\^{}2] Sqrt[1 + x23\^{}2] y1)/\mlbbb (Sqrt[1 + x14\^{}2] Sqrt[1 + x24\^{}2] (1 + x34\^{}2)), \mlbb (Sqrt[1 + x12\^{}2] Sqrt[1 + x34\^{}2] y2)/\mlbbb (Sqrt[1 + x13\^{}2] (1 + x23\^{}2) Sqrt[1 + x24\^{}2]), \mlbb (Sqrt[1 + x23\^{}2] Sqrt[1 + x24\^{}2] y3)/\mlbbb ((1 + x12\^{}2) Sqrt[1 + x13\^{}2] Sqrt[1 + x14\^{}2])\},\mlb
  \{"12", "13", "14", "23", "24", "34"\}];\\
LogicalExpand[FullSimplify[Phi1111b[[1, 2]]]]

\endgroup

\bibliographystyle{amsplain}

\bibliography{HigherWeight}

\end{document}